\documentclass[11pt,reqno,english]{amsart}
\usepackage[T1]{fontenc}
\usepackage[latin9]{inputenc}
\usepackage{babel}
\usepackage{mathrsfs}
\usepackage{mathtools}
\usepackage{bm}
\usepackage{amstext}
\usepackage{amsthm}
\usepackage{amssymb}
\usepackage{stmaryrd}
\usepackage{geometry}
\usepackage{setspace}
\usepackage{esint}
\usepackage[pdfusetitle,
 bookmarks=true,bookmarksnumbered=false,bookmarksopen=false,
 breaklinks=false,pdfborder={0 0 1},backref=false,colorlinks=false]
 {hyperref}

\makeatletter
\numberwithin{equation}{section}
\numberwithin{figure}{section}
\theoremstyle{remark}
\newtheorem*{notation*}{\protect\notationname}
\theoremstyle{plain}
\newtheorem{thm}{\protect\theoremname}[section]
\theoremstyle{remark}
\newtheorem{notation}[thm]{\protect\notationname}
\theoremstyle{definition}
\newtheorem{defn}[thm]{\protect\definitionname}
\theoremstyle{remark}
\newtheorem{rem}[thm]{\protect\remarkname}
\theoremstyle{plain}
\newtheorem{cor}[thm]{\protect\corollaryname}
\theoremstyle{plain}
\newtheorem{prop}[thm]{\protect\propositionname}
\theoremstyle{plain}
\newtheorem{lem}[thm]{\protect\lemmaname}
\theoremstyle{plain}
\newtheorem{assumption}[thm]{\protect\assumptionname}
\theoremstyle{remark}
\newtheorem*{acknowledgement*}{\protect\acknowledgementname}

\usepackage{cite}
\usepackage{bbm}

\providecommand{\leftsquigarrow}{%
  \mathrel{\mathpalette\reflect@squig\relax}%
}
\newcommand{\reflect@squig}[2]{%
  \reflectbox{$\m@th#1\rightsquigarrow$}%
}

\usepackage{enumitem}
	\setlist[itemize]{leftmargin=*}
	\setlist[enumerate]{leftmargin=*}

\usepackage{tikz}
\usetikzlibrary[patterns]

\makeatother

\providecommand{\acknowledgementname}{Acknowledgement}
\providecommand{\assumptionname}{Assumption}
\providecommand{\corollaryname}{Corollary}
\providecommand{\definitionname}{Definition}
\providecommand{\lemmaname}{Lemma}
\providecommand{\notationname}{Notation}
\providecommand{\propositionname}{Proposition}
\providecommand{\remarkname}{Remark}
\providecommand{\theoremname}{Theorem}

\begin{document}
\title[Mixing Time of the Curie--Weiss--Potts Model at Low Temperatures]{Sharp Mixing Time Asymptotics of Glauber Dynamics for the Curie--Weiss--Potts
Model at Low Temperatures}
\author{Seonwoo Kim and Jungkyoung Lee}
\begin{abstract}
In this article, we derive a sharp mixing time estimate of the Glauber
dynamics for the Curie--Weiss--Potts model in the low-temperature
regime. In contrast to the high-temperature regime studied by Cuff
et al. (J. Stat. Phys. 149: 432--477, 2012), in which the Gibbs measure
is concentrated around the equiproportional distribution of spins,
the Gibbs measure in the low-temperature regime is concentrated on
multiple states, each with a dominant number of a single spin. Consequently,
global mixing of the system requires sufficiently many transitions
between these states. Since these transitions are well explained by
the phenomenon of metastability, the theory of metastability plays
a central role in the analysis of slow mixing. In particular, the
sharp asymptotics for the mixing time is given by the mixing time
of the limit Markov chain, which describes the metastable behavior
of the dynamics, multiplied by the metastable transition time-scale.
As a byproduct, we verify that it does not exhibit a cutoff phenomenon.
\end{abstract}

\address{S. Kim. Department of Mathematics, Yonsei University, Republic of
Korea.}
\email{seonwookim@yonsei.ac.kr}
\address{J. Lee. Department of Mathematics Education, Inha University, Republic
of Korea.}
\email{jungkyoung@inha.ac.kr}

\maketitle
\tableofcontents{}

\section{\label{sec1}Introduction and Main Result}
\begin{notation*}
We gather a few notations that are repeatedly used in this article.
\begin{itemize}
\item For $a,b\in\mathbb{R}$, we write $\llbracket a,b\rrbracket:=[a,b]\cap\mathbb{Z}$,
$a\wedge b:=\min\{a,b\}$, and $a\vee b:=\max\{a,b\}$.
\item For $a\in\mathbb{R}$, define $[a]_{+}:=a\vee0=\max\{a,0\}$.
\item For any set $A$, its indicator function is written as $\bm{1}_{A}$.
\item By writing $f(N)=O(g(N))$ it means $|f(N)|\le Cg(N)$ where $C>0$
does not depend on $N$.
\item Write $f(N)\ll g(N)$ if $\lim_{N\to\infty}\frac{f(N)}{g(N)}=0$.
\item Write $f(N)\simeq g(N)$ if $\lim_{N\to\infty}\frac{f(N)}{g(N)}=1$.
\item For a function $f$ on $\mathbb{R}$ and $a\in\mathbb{R}$, we write
$f(a+):=\lim_{t\downarrow a}f(t)$ and $f(a-):=\lim_{t\uparrow a}f(t)$.
\item For a set $A$, denote by $\mathcal{H}_{A}$ the hitting time of $A$
with respect to the specified dynamics.
\end{itemize}
\end{notation*}

\subsection{\label{sec1.1}Mixing Time}

Consider a finite space $S$ and a Markov chain $\{X(t)\}_{t\ge0}$
therein. Denote by $X(t;\mu)$ the law of the process at time $t\ge0$
starting from initial distribution $\mu$. If $\mu=\delta_{x}$ for
some $x\in S$, we write $X(t;x):=X(t;\delta_{x})$. Suppose that
$X$ is ergodic, thus has a unique invariant distribution, say, $\pi$.
By a standard ergodicity argument, the distribution $X(t;x)$ converges
to $\pi$ as $t\to\infty$ for any given $x\in S$. The most standard
way to formulate this convergence is via their \emph{total variation
distance}: for any two probability measures $\mu,\nu$ on $S$, define
\[
d_{{\rm TV}}(\mu,\nu):=\sup_{A\subset S}|\mu(A)-\nu(A)|=\frac{1}{2}\sum_{x\in S}|\mu(x)-\nu(x)|.
\]
The convergence to equilibrium is then mathematically represented
as $\lim_{t\to\infty}d_{{\rm TV}}(X(t;x),\pi)=0$ for all $x\in S$.
This phenomenon is referred to as the \emph{mixing} property in the
literature.

Beyond the convergence itself to the equilibrium $\pi$, one can further
try to quantify this convergence as follows. For $\delta>0$, the
\emph{$\delta$-mixing time} of the process $\{X(t)\}_{t\ge0}$ is
defined as
\[
T_{\delta}^{{\rm mix}}(X):=\inf\left\{ t\ge0:\sup_{x\in S}d_{{\rm TV}}(X(t;x),\pi)\le\delta\right\} .
\]
This quantity measures the time required for the distribution of the
process to become closer than or equal to $\delta$ to the equilibrium
distribution $\pi$. There exists a vast literature of the study of
mixing time and its applications; we refer to the recent monograph
\cite{LPW17} for a comprehensive review.

\subsection{\label{sec1.2}Curie--Weiss--Potts Model}

This article primarily focuses on investigating the mixing time of
the \emph{Curie--Weiss--Potts} (or CWP) model. This is an interacting
spin system on the complete graph which serves as a mean-field approximation
of the standard Ising model \cite{Isi25} or Potts model \cite{Pot52}
on lattices. We rigorously describe the model as follows. For a positive
integer $N$, denote by $K_{N}=\llbracket1,N\rrbracket$ the set of
sites in the system. Let $\Omega_{N}:=\llbracket1,q\rrbracket^{K_{N}}$
be the configuration space of $q\ge2$ spins on $K_{N}$. In the case
of $q=2$, we simply refer to the model as the Curie--Weiss model
\cite{Wei07}.

Each configuration is represented as an element $\sigma=(\sigma_{1},\dots,\sigma_{N})\in\Omega_{N}$
where $\sigma_{u}\in\llbracket1,q\rrbracket$ denotes the spin at
site $u\in K_{N}$. Its Hamiltonian is given by
\begin{equation}
\mathbb{H}_{N}(\sigma):=-\frac{1}{2N}\sum_{u,v\in\llbracket1,N\rrbracket}\bm{1}_{\{\sigma_{u}=\sigma_{v}\}}\qquad\text{for}\quad\sigma\in\Omega_{N}.\label{eq:Hamiltonian}
\end{equation}
Then, the Curie--Weiss--Potts \emph{Gibbs measure} associated to
the Hamiltonian at inverse temperature $\beta>0$ is given as
\begin{equation}
\mu_{N}^{\beta}(\sigma):=\frac{1}{Z_{N}^{\beta}}e^{-\beta\mathbb{H}_{N}(\sigma)},\qquad\sigma\in\Omega_{N},\label{eq:muNbeta}
\end{equation}
where $Z_{N}^{\beta}=\sum_{\sigma\in\Omega_{N}}e^{-\beta\mathbb{H}_{N}(\sigma)}$
is the partition function which makes $\mu_{N}^{\beta}$ a probability
measure on $\Omega_{N}$.

\subsection{\label{sec1.3}Magnetization}

As the complete graph $K_{N}$ has no geometric structure, we may
study the CWP model solely in terms of its \emph{magnetization}. Define
a $(q-1)$-dimensional space $\Xi$ as
\begin{equation}
\Xi:=\left\{ \bm{x}=(x_{1},\dots,x_{q})\in\mathbb{R}^{q}:x_{1},\dots,x_{q}\ge0,\enspace\sum_{i=1}^{q}x_{i}=1\right\} .\label{eq:Xi-def}
\end{equation}
Let $\Xi^{\circ}$ be the interior of $\Xi$. Denote by $\Xi_{N}:=\Xi\cap(N^{-1}\mathbb{Z})^{q}$
its discretization. The magnetization vector of each $\sigma\in\Omega_{N}$
is defined via a projection function $\Pi_{N}:\Omega_{N}\to\Xi_{N}$
given as
\begin{equation}
\Pi_{N}(\sigma):=\left(\Pi_{N}^{1}(\sigma),\dots,\Pi_{N}^{q}(\sigma)\right),\label{eq:PiN}
\end{equation}
where $\Pi_{N}^{k}(\sigma)$ denotes the proportion of spins of type
$k$:
\[
\Pi_{N}^{k}(\sigma):=\frac{1}{N}\sum_{v=1}^{N}\bm{1}_{\{\sigma_{v}=k\}}.
\]
For each $\bm{x}\in\Xi$ define
\[
H(\bm{x}):=-\frac{1}{2}|\bm{x}|^{2}\qquad\text{and}\qquad S(\bm{x})=\sum_{k=1}^{q}x_{k}\log x_{k},
\]
where we adopt the convention $0\log0:=0$. Let $\pi_{N}^{\beta}$
be the pushforward measure of $\mu_{N}^{\beta}$ by $\Pi_{N}$:
\begin{equation}
\pi_{N}^{\beta}:=\mu_{N}^{\beta}\circ\Pi_{N}^{-1}.\label{eq:piNbeta}
\end{equation}
Since 
\begin{equation}
\mathbb{H}_{N}(\sigma)=NH(\bm{x})\qquad\text{if}\quad\Pi_{N}(\sigma)=\bm{x},\label{eq:HN-NH}
\end{equation}
via Stirling's formula, $\pi_{N}^{\beta}$ can be rewritten as
\begin{equation}
\begin{aligned}\pi_{N}^{\beta}(\bm{x})=\sum_{\sigma\in\Omega_{N}:\,\Pi_{N}(\sigma)=\bm{x}}\frac{1}{Z_{N}^{\beta}}e^{-\beta\mathbb{H}_{N}(\sigma)} & =\frac{N!}{(Nx_{1})!\cdots(Nx_{q})!}\frac{1}{Z_{N}^{\beta}}e^{-\beta NH(\bm{x})}\\
 & =:\frac{1}{(2\pi N)^{\frac{q-1}{2}}Z_{N}^{\beta}}e^{-\beta NF_{\beta,N}(\bm{x})},
\end{aligned}
\label{eq:piNbeta-formula}
\end{equation}
where $F_{\beta,N}(\bm{x}):=F_{\beta}(\bm{x})+\frac{1}{N}G_{\beta,N}(\bm{x})$
and
\begin{equation}
F_{\beta}(\bm{x}):=H(\bm{x})+\frac{1}{\beta}S(\bm{x}),\qquad G_{\beta,N}(\bm{x}):=\frac{1}{2\beta}\log\left(\prod_{k\in\llbracket1,q\rrbracket:\,x_{k}>0}x_{k}\right)+O\left(\frac{1}{N}\right).\label{eq:FbetaN-def}
\end{equation}
Here, $G_{\beta,N}$ converges uniformly to $G_{\beta}$ as $N\to\infty$
on every compact subset of $\Xi^{\circ}$, where
\begin{equation}
G_{\beta}(\bm{x}):=\frac{\log(x_{1}\cdots x_{q})}{2\beta}.\label{eq:Gbeta}
\end{equation}

\subsection{\label{sec1.4}Energy Landscape}

We review the results on the energy landscape described by $F_{\beta}$.
\begin{notation}
\label{nota:q-to-q-1}We may identify the $(q-1)$-dimensional space
$\Xi$ by using only the first $q-1$ coordinates. With a slight abuse
of notation, we may regard
\[
\Xi=\left\{ \bm{x}=(x_{1},\dots,x_{q-1})\in\mathbb{R}^{q-1}:x_{1},\dots,x_{q-1}\ge0,\enspace\sum_{k=1}^{q-1}x_{k}\le1\right\} .
\]
In this way, $F_{\beta}$ becomes an analytic function in $\Xi^{\circ}$
which extends continuously to $\Xi$. For each $k\in\llbracket1,q-1\rrbracket$,
denote by $\partial_{k}$ the $k$-th partial derivative and write
$\nabla:=(\partial_{1},\dots,\partial_{q-1})$.
\end{notation}

\begin{defn}
\label{def:C-critical}In the terminology of Notation \ref{nota:q-to-q-1},
denote by $\mathcal{C}$ the set of critical points of $F_{\beta}$
in $\Xi^{\circ}$.
\end{defn}

The set $\mathcal{C}$ is fully characterized in \cite[Sections 6 and 7]{Lee22}.
We summarize them in Appendix \ref{secA} mostly without proofs.

We say that $\varphi:[0,1]\to\Xi$ is a \emph{trajectory}\footnote{We use the term \emph{trajectory} to refer to a curve in continuum
space, and later in Section \ref{sec4} use the term \emph{path} to
denote a sequence in discrete space.} from $\bm{x}\in\Xi$ to $\bm{y}\in\Xi$ if it is continuous, $\varphi(0)=\bm{x}$,
and $\varphi(1)=\bm{y}$. Define
\[
\Phi_{\beta}(\bm{x},\bm{y}):=\inf_{\varphi}\max_{t\in[0,1]}F_{\beta}(\varphi(t)),
\]
the \emph{communication height} between $\bm{x}$ and $\bm{y}$, where
the infimum runs over all trajectories $\varphi$ from $\bm{x}$ to
$\bm{y}$. One may similarly define $\Phi_{\beta}(\mathcal{X},\mathcal{Y})$
for any two sets $\mathcal{X},\mathcal{Y}\subset\Xi$.

Let us write
\begin{equation}
{\bf e}:=\left(\frac{1}{q},\dots,\frac{1}{q}\right)\in\Xi,\label{eq:e-def}
\end{equation}
which represents the equiproportional vector. Recall from \eqref{eq:ukvk-def}
the definition of critical points ${\bf u}_{k}\in\mathcal{C}$ for
$k\in\llbracket1,q\rrbracket$. It is verified in \cite[Proposition 3.2-(1)]{Lee22}
that $\mathcal{U}:=\{{\bf u}_{1},\dots,{\bf u}_{q},{\bf e}\}$ contains
all possible local minima of $F_{\beta}$ in $\Xi$.
\begin{defn}
For each $k\in\llbracket1,q\rrbracket$, denote by $\mathcal{W}_{k}$
the connected component of $\{F_{\beta}<\Phi_{\beta}({\bf u}_{k},\mathcal{U}\setminus\{{\bf u}_{k}\})\}$
containing ${\bf u}_{k}$ and by $\mathcal{W}_{0}$ the connected
component of $\{F_{\beta}<\Phi_{\beta}({\bf e},\mathcal{U}\setminus\{{\bf e}\})\}$
containing ${\bf e}$. For $k,\ell\in\llbracket0,q\rrbracket$ with
$k\ne\ell$, define $\Sigma_{k,\ell}:=\overline{\mathcal{W}_{k}}\cap\overline{\mathcal{W}_{\ell}}$
which represents the set of saddle points connecting $\mathcal{W}_{k}$
and $\mathcal{W}_{\ell}$.
\end{defn}

\subsubsection*{Case 1: $q=2$}

The case of $q=2$ is elementary and represents the classical Curie--Weiss
model. In this case, there exists exactly one critical temperature
$\beta_{1}:=2$ at which a sharp phase transition occurs as follows:
\begin{itemize}
\item If $\beta\in(0,2]$, then $F_{\beta}$ has only one local minimum
${\bf e}=(\frac{1}{2},\frac{1}{2})$ which is the global minimum.
In particular, ${\bf e}$ is degenerate if and only if $\beta=2$.
\item If $\beta\in(2,\infty)$, then $F_{\beta}$ has two local minima ${\bf u}_{1},{\bf u}_{2}$
such that ${\bf u}_{2}=(1,1)-{\bf u}_{1}$. Here, $\Sigma_{1,2}=\{{\bf e}\}$.
\end{itemize}
See Figure \ref{fig1.1} for an illustration of the graph of $F_{\beta}$
if $q=2$.
\begin{center}
\begin{figure}
\begin{centering}
\begin{tikzpicture}
\begin{scope}[scale=1.2]
\draw[white] (-1.35,-0.5) rectangle (1.35,2.5);
\draw[thick,domain=-0.999:0.999,smooth,variable=\x] plot(\x,{2*(ln(1-\x)*(1-\x)+ln(1+\x)*(1+\x)-0.8*(\x)^2)/(ln(0.001)*(0.001)+ln(1.999)*(1.999)-0.8*0.999^2)});
\draw[densely dashed,->] (-1.2,0)--(1.2,0); \draw[densely dashed,->] (-1,-0.2)--(-1,2.2); \draw[densely dashed] (1,0)--(1,2.2);
\fill[] (0,0) circle (0.04);
\draw (0,-0.05) node[below]{${\bf e}$};
\end{scope}

\begin{scope}[scale=1.2,shift={(4,0)}]
\draw[white] (-1.35,-0.5) rectangle (1.35,2.5);
\draw[thick,domain=-0.999:0.999,smooth,variable=\x] plot(\x,{2*(ln(1-\x)*(1-\x)+ln(1+\x)*(1+\x)-(\x)^2)/(ln(0.001)*(0.001)+ln(1.999)*(1.999)-0.999^2)});
\draw[densely dashed,->] (-1.2,0)--(1.2,0); \draw[densely dashed,->] (-1,-0.2)--(-1,2.2); \draw[densely dashed] (1,0)--(1,2.2);
\fill[] (0,0) circle (0.04);
\draw (0,-0.05) node[below]{${\bf e}$};
\end{scope}

\begin{scope}[scale=1.2,shift={(8,0)}]
\fill[red!50!white] (-0.999,0) rectangle (0.999,0.3);
\fill[white,domain=-0.999:0.999,variable=\x] (-0.999,-0.1)--(-0.999,0)--(-0.999,{2*(ln(1+0.999)*(1+0.999)+ln(1-0.999)*(1-0.999)-1.2*(-0.999)^2-(ln(1-0.659)*(1-0.659)+ln(1+0.659)*(1+0.659)-1.2*(0.659)^2))/(ln(0.001)*(0.001)+ln(1.999)*(1.999)-1.2*0.999^2-(ln(1-0.659)*(1-0.659)+ln(1+0.659)*(1+0.659)-1.2*(0.659)^2))})--plot({\x},{2*(ln(1-\x)*(1-\x)+ln(1+\x)*(1+\x)-1.2*(\x)^2-(ln(1-0.659)*(1-0.659)+ln(1+0.659)*(1+0.659)-1.2*(0.659)^2))/(ln(0.001)*(0.001)+ln(1.999)*(1.999)-1.2*0.999^2-(ln(1-0.659)*(1-0.659)+ln(1+0.659)*(1+0.659)-1.2*(0.659)^2))})--(0.999,0)--(0.999,-0.1)--cycle;
\draw[white] (-1.35,-0.5) rectangle (1.35,2.5);
\draw[thick,domain=-0.999:0.999,smooth,variable=\x] plot(\x,{2*(ln(1-\x)*(1-\x)+ln(1+\x)*(1+\x)-1.2*(\x)^2-(ln(1-0.659)*(1-0.659)+ln(1+0.659)*(1+0.659)-1.2*(0.659)^2))/(ln(0.001)*(0.001)+ln(1.999)*(1.999)-1.2*0.999^2-(ln(1-0.659)*(1-0.659)+ln(1+0.659)*(1+0.659)-1.2*(0.659)^2))});
\draw[densely dashed,->] (-1.2,0)--(1.2,0); \draw[densely dashed,->] (-1,-0.2)--(-1,2.2); \draw[densely dashed] (1,0)--(1,2.2);
\fill[] (-0.659,0) circle (0.04); \fill[] (0.659,0) circle (0.04);
\draw (-0.659,-0.05) node[below]{${\bf u}_1$}; \draw (0.659,-0.05) node[below]{${\bf u}_2$};
\draw[thick] (0,0.422-0.05)--(0,0.422+0.05);
\draw (0,0.422+0.05) node[above]{${\bf e}$};
\end{scope}
\end{tikzpicture}
\par\end{centering}
\caption{\label{fig1.1}Graph of $F_{\beta}$ in the Curie--Weiss model where
the horizontal axis reads the first coordinate of the elements in
$\Xi=\{(x_{1},1-x_{1}):0\le x_{1}\le1\}$. If $\beta<2$ (left), then
${\bf e}$ is the only local (thus global) minimum which is non-degenerate.
If $\beta=2$ (middle), the graph shape remains to be the same but
in this case ${\bf e}$ is degenerate. Finally, if $\beta>2$ (right)
then there exist two symmetric local minima ${\bf u}_{1}$ and ${\bf u}_{2}$,
and their saddle point is the midpoint ${\bf e}$. In particular,
$H_{\beta}=F_{\beta}({\bf e})$. The metastable valleys $\mathcal{E}_{N}^{1}$
and $\mathcal{E}_{N}^{2}$ (cf. \eqref{eq:ENk-def}) are colored red.}
\end{figure}
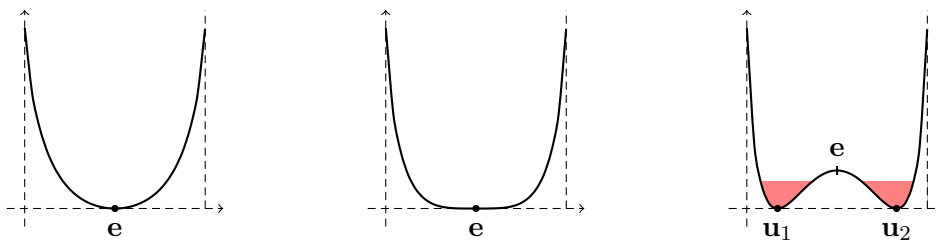
\par\end{center}

\subsubsection*{Case 2: $q\in\{3,4\}$}

Next, we consider the case when $q\in\{3,4\}$. Recall the critical
points ${\bf v}_{k},{\bf u}_{k,\ell}\in\mathcal{C}$ for $k,\ell\in\llbracket1,q\rrbracket$
from \eqref{eq:ukvk-def} and \eqref{eq:ukl-def}. In this case, the
model has three critical temperatures 
\[
0<\beta_{1}<\beta_{2}<q
\]
at which phase transitions occur. The following results are from \cite[Theorem 3.5]{Lee22}.
\begin{itemize}
\item If $\beta\in(0,\beta_{1}]$, then $F_{\beta}$ has only one local
minimum ${\bf e}$ which is the global minimum. If $\beta=\beta_{1}$,
${\bf e}$ is a degenerate minimum.
\item If $\beta\in(\beta_{1},q)$, then $F_{\beta}$ has $q+1$ local minima
${\bf u}_{1},\dots,{\bf u}_{q},{\bf e}$ such that
\[
\begin{aligned}F_{\beta}({\bf u}_{1})=\cdots=F_{\beta}({\bf u}_{q})>F_{\beta}({\bf e}) & \qquad\text{if}\quad\beta\in(\beta_{1},\beta_{2}),\\
F_{\beta}({\bf u}_{1})=\cdots=F_{\beta}({\bf u}_{q})=F_{\beta}({\bf e}) & \qquad\text{if}\quad\beta=\beta_{2},\\
F_{\beta}({\bf u}_{1})=\cdots=F_{\beta}({\bf u}_{q})<F_{\beta}({\bf e}) & \qquad\text{if}\quad\beta\in(\beta_{2},q).
\end{aligned}
\]
Moreover, $\Sigma_{0,k}=\{{\bf v}_{k}\}$ and $\Sigma_{k,\ell}=\varnothing$
for $k,\ell\in\llbracket1,q\rrbracket$.
\item If $\beta=q$, then $F_{\beta}$ has $q$ local minima ${\bf u}_{1},\dots,{\bf u}_{q}$
and $\Sigma_{k,\ell}=\{{\bf e}\}$ for any $k\ne\ell$. Here, ${\bf e}={\bf v}_{k}={\bf u}_{k,\ell}$
for any $k,\ell$, and this is a degenerate critical point. This is
the reason why this specific case will be excluded; it is impossible
to analyze the metastable transition near the degenerate critical
point ${\bf e}$ via our method.
\item If $\beta\in(q,\infty)$, then $F_{\beta}$ has $q$ local minima
${\bf u}_{1},\dots,{\bf u}_{q}$ and, for $k\ne\ell$,
\[
\begin{cases}
\text{if}\quad q=3\quad\text{then}\quad\Sigma_{k,\ell}=\{{\bf v}_{m}\}\quad\text{where}\quad\{k,\ell,m\}=\llbracket1,3\rrbracket,\\
\text{if}\quad q=4\quad\text{then}\quad\Sigma_{k,\ell}=\{{\bf u}_{k,\ell}\}.
\end{cases}
\]
\end{itemize}
Refer to Figure \ref{fig1.2} for illustrations for each case.

\subsubsection*{Case 3: $q\ge5$}

Finally, assume that $q\ge5$. As a contrary to the previous cases,
we have four critical temperatures 
\[
0<\beta_{1}<\beta_{2}<\beta_{3}<q,
\]
an additional critical temperature $\beta_{3}\in(\beta_{2},q)$ arising
here. The following results summarize \cite[Theorem 3.6]{Lee22}.
\begin{itemize}
\item If $\beta\in(0,\beta_{1}]$, then $F_{\beta}$ has only one local
minimum ${\bf e}$ which is thus global. If $\beta=\beta_{1}$, ${\bf e}$
is degenerate.
\item If $\beta\in(\beta_{1},\beta_{3})$, then $F_{\beta}$ has $q+1$
local minima ${\bf u}_{1},\dots,{\bf u}_{q},{\bf e}$ such that
\[
\begin{aligned}F_{\beta}({\bf u}_{1})=\cdots=F_{\beta}({\bf u}_{q})>F_{\beta}({\bf e}) & \qquad\text{if}\quad\beta\in(\beta_{1},\beta_{2}),\\
F_{\beta}({\bf u}_{1})=\cdots=F_{\beta}({\bf u}_{q})=F_{\beta}({\bf e}) & \qquad\text{if}\quad\beta=\beta_{2},\\
F_{\beta}({\bf u}_{1})=\cdots=F_{\beta}({\bf u}_{q})<F_{\beta}({\bf e}) & \qquad\text{if}\quad\beta\in(\beta_{2},\beta_{3}).
\end{aligned}
\]
Moreover, $\Sigma_{0,k}=\{{\bf v}_{k}\}$ and $\Sigma_{k,\ell}=\varnothing$
for $k,\ell\in\llbracket1,q\rrbracket$.
\item If $\beta=\beta_{3}$, then $F_{\beta}$ has $q+1$ local minima ${\bf u}_{1},\dots,{\bf u}_{q},{\bf e}$
such that
\[
F_{\beta}({\bf u}_{1})=\cdots=F_{\beta}({\bf u}_{q})<F_{\beta}({\bf e})<F_{\beta}({\bf u}_{1,2})=\cdots=F_{\beta}({\bf u}_{q-1,q})=F_{\beta}({\bf v}_{1})=\cdots=F_{\beta}({\bf v}_{q}).
\]
Moreover, $\Sigma_{0,k}=\{{\bf v}_{k}\}$ and $\Sigma_{k,\ell}=\{{\bf u}_{k,\ell}\}$
for $k,\ell\in\llbracket1,q\rrbracket$.
\item If $\beta\in(\beta_{3},q)$, then $F_{\beta}$ has $q+1$ local minima
${\bf u}_{1},\dots,{\bf u}_{q},{\bf e}$ such that
\[
F_{\beta}({\bf u}_{1})=\cdots=F_{\beta}({\bf u}_{q})<F_{\beta}({\bf e}),\qquad F_{\beta}({\bf u}_{1,2})=\cdots=F_{\beta}({\bf u}_{q-1,q})<F_{\beta}({\bf v}_{1})=\cdots=F_{\beta}({\bf v}_{q}).
\]
In addition, $\Sigma_{0,k}=\varnothing$ and $\Sigma_{k,\ell}=\{{\bf u}_{k,\ell}\}$
for $k,\ell\in\llbracket1,q\rrbracket$. The set $\{F_{\beta}<F_{\beta}({\bf v}_{1})\}$
has only two connected components, the well $\mathcal{W}_{0}$ and
the other which contains ${\bf u}_{1},\dots,{\bf u}_{q}$, say $\widehat{\mathcal{W}}_{1}$,
such that $\overline{\mathcal{W}_{0}}\cap\overline{\widehat{\mathcal{W}}_{1}}=\{{\bf v}_{1},\dots,{\bf v}_{q}\}$.
Moreover, the depth of each $\mathcal{W}_{k}$, $k\in\llbracket1,q\rrbracket$
is strictly bigger than the depth of $\mathcal{W}_{0}$; i.e.,
\begin{equation}
F_{\beta}({\bf u}_{1,2})-F_{\beta}({\bf u}_{1})>F_{\beta}({\bf v}_{1})-F_{\beta}({\bf e}).\label{eq:depth-bigger}
\end{equation}
\item If $\beta=[q,\infty)$, then $F_{\beta}$ has $q$ local minima ${\bf u}_{1},\dots,{\bf u}_{q}$
and $\Sigma_{k,\ell}=\{{\bf u}_{k,\ell}\}$ for $k,\ell\in\llbracket1,q\rrbracket$.
\end{itemize}
Refer again to Figure \ref{fig1.2} for illustrations. We remark here
that the proof of \eqref{eq:depth-bigger}, alone, is missing from
\cite{Lee22}. We present its proof in Lemma \ref{lem:depth-bigger}.
\begin{center}
\begin{figure}
\begin{centering}
\begin{tikzpicture}
\draw (0,1.5) node{\scriptsize $\beta \in (\beta_1,\beta_2)$};
\draw (2.4,1.5) node{\scriptsize $\beta = \beta_2$};
\draw (7.2,1.5) node{\scriptsize $\beta \in (\beta_2,q)$};
\draw (12,1.5) node{\scriptsize $\beta \in (q,\infty)$};

\draw (0,-4.5) node{\scriptsize $\beta \in (\beta_1,\beta_2)$};
\draw (2.4,-4.5) node{\scriptsize $\beta = \beta_2$};
\draw (4.8,-4.5) node{\scriptsize $\beta \in (\beta_2,\beta_3)$};
\draw (7.2,-4.5) node{\scriptsize $\beta = \beta_3$};
\draw (9.6,-4.5) node{\scriptsize $\beta \in (\beta_3,q)$};
\draw (12,-4.5) node{\scriptsize $\beta \in [q,\infty)$};

\draw (-1.3,0) node[left]{\scriptsize $q=3$};
\draw (-1.3,-2.4) node[left]{\scriptsize $q=4$};
\draw (-1.3,-6) node[left]{\scriptsize $q \ge 5$};

\begin{scope}[scale=0.6]
\draw[] (-2,-2) rectangle (2,2);
\fill[black!40!white] (0,0) circle (1); \draw (0,0) circle (1); \draw[thick] (-0.06,-0.06)--(0.06,0.06); \draw[thick] (0.06,-0.06)--(-0.06,0.06);
\draw[red,densely dashed] (0,0) circle (0.6);
\foreach \i in {60,180,300} {
\fill[black!15!white] ({1.4*cos(\i)},{1.4*sin(\i)}) circle (0.4); \draw ({1.4*cos(\i)},{1.4*sin(\i)}) circle (0.4); \fill ({1.4*cos(\i)},{1.4*sin(\i)}) circle (0.07);
\draw[red,densely dashed] ({1.4*cos(\i)},{1.4*sin(\i)}) circle (0.24);
\draw[thick] ({0.925*cos(\i)},{0.925*sin(\i)})--({1.075*cos(\i)},{1.075*sin(\i)});
}
\draw ({1.4*cos(60)},{1.4*sin(60)-0.1}) node[above]{\scriptsize ${\bf u}_k$};
\draw ({1*cos(60)},{1*sin(60)+0.1}) node[below]{\scriptsize ${\bf v}_k$};
\draw (0,0+0.1) node[below]{\scriptsize $\bf e$};
\end{scope}

\begin{scope}[scale=0.6,shift={(4,0)}]
\draw[] (-2,-2) rectangle (2,2);
\fill[black!40!white] (0,0) circle (0.6); \draw (0,0) circle (0.6); \draw[thick] (-0.06,-0.06)--(0.06,0.06); \draw[thick] (0.06,-0.06)--(-0.06,0.06);
\draw[red,densely dashed] (0,0) circle (0.36);
\foreach \i in {60,180,300} {
\fill[black!40!white] ({1.2*cos(\i)},{1.2*sin(\i)}) circle (0.6); \draw ({1.2*cos(\i)},{1.2*sin(\i)}) circle (0.6); \fill ({1.2*cos(\i)},{1.2*sin(\i)}) circle (0.07);
\draw[red,densely dashed] ({1.2*cos(\i)},{1.2*sin(\i)}) circle (0.36);
\draw[thick] ({0.525*cos(\i)},{0.525*sin(\i)})--({0.675*cos(\i)},{0.675*sin(\i)});
}
\end{scope}

\begin{scope}[scale=0.6,shift={(12,0)}]
\draw[] (-6,-2) rectangle (6,2);
\fill[black!15!white] (0,0) circle (0.55); \draw (0,0) circle (0.55); \draw[thick] (-0.06,-0.06)--(0.06,0.06); \draw[thick] (0.06,-0.06)--(-0.06,0.06);
\draw[red,densely dashed] (0,0) circle (0.33);
\foreach \i in {60,180,300} {
\fill[black!40!white] ({1.175*cos(\i)},{1.175*sin(\i)}) circle (0.625); \draw ({1.175*cos(\i)},{1.175*sin(\i)}) circle (0.625); \fill ({1.175*cos(\i)},{1.175*sin(\i)}) circle (0.07);
\draw[red,densely dashed] ({1.175*cos(\i)},{1.175*sin(\i)}) circle (0.375);
\draw[thick] ({0.475*cos(\i)},{0.475*sin(\i)})--({0.625*cos(\i)},{0.625*sin(\i)});
}
\end{scope}

\begin{scope}[scale=0.6,shift={(20,0)}]
\draw[] (-2,-2) rectangle (2,2);
\draw[thick] (-0.06,-0.06)--(0.06,0.06); \draw[thick] (0.06,-0.06)--(-0.06,0.06);
\foreach \i in {60,180,300} {
\fill[black!40!white] ({0.7/cos(30)*cos(\i)},{0.7/cos(30)*sin(\i)}) circle (0.7); \draw ({0.7/cos(30)*cos(\i)},{0.7/cos(30)*sin(\i)}) circle (0.7); \fill ({0.7/cos(30)*cos(\i)},{0.7/cos(30)*sin(\i)}) circle (0.07);
\draw[red,densely dashed] ({0.7/cos(30)*cos(\i)},{0.7/cos(30)*sin(\i)}) circle (0.42);
}
\foreach \i in {0,120,240} {
\draw[thick] ({0.7/sqrt(3)*cos(\i)+0.075*cos(\i+90)},{0.7/sqrt(3)*sin(\i)+0.075*sin(\i+90)})--({0.7/sqrt(3)*cos(\i)-0.075*cos(\i+90)},{0.7/sqrt(3)*sin(\i)-0.075*sin(\i+90)});
}
\draw ({0.7/cos(30)*cos(60)},{0.7/cos(30)*sin(60)-0.1}) node[above]{\scriptsize ${\bf u}_k$};
\draw ({0.35/cos(30)*cos(120)},{0.35/cos(30)*sin(120)-0.1}) node[above]{\scriptsize ${\bf v}_m$};
\draw ({0.7/cos(30)*cos(180)},{0.7/cos(30)*sin(180)-0.1}) node[above]{\scriptsize ${\bf u}_\ell$};
\end{scope}

\begin{scope}[scale=0.6,shift={(0,-4)}]
\draw[] (-2,-2) rectangle (2,2);
\fill[black!40!white] (0,0) circle (1); \draw (0,0) circle (1); \draw[thick] (-0.06,-0.06)--(0.06,0.06); \draw[thick] (0.06,-0.06)--(-0.06,0.06);
\foreach \i in {45,135,225,315} {
\fill[black!15!white] ({1.4*cos(\i)},{1.4*sin(\i)}) circle (0.4); \draw ({1.4*cos(\i)},{1.4*sin(\i)}) circle (0.4); \fill ({1.4*cos(\i)},{1.4*sin(\i)}) circle (0.07);
\draw[thick] ({0.925*cos(\i)},{0.925*sin(\i)})--({1.075*cos(\i)},{1.075*sin(\i)});
}
\end{scope}

\begin{scope}[scale=0.6,shift={(4,-4)}]
\draw[] (-2,-2) rectangle (2,2);
\fill[black!40!white] (0,0) circle (0.6); \draw (0,0) circle (0.6); \draw[thick] (-0.06,-0.06)--(0.06,0.06); \draw[thick] (0.06,-0.06)--(-0.06,0.06);
\foreach \i in {45,135,225,315} {
\fill[black!40!white] ({1.2*cos(\i)},{1.2*sin(\i)}) circle (0.6); \draw ({1.2*cos(\i)},{1.2*sin(\i)}) circle (0.6); \fill ({1.2*cos(\i)},{1.2*sin(\i)}) circle (0.07);
\draw[thick] ({0.525*cos(\i)},{0.525*sin(\i)})--({0.675*cos(\i)},{0.675*sin(\i)});
}
\end{scope}

\begin{scope}[scale=0.6,shift={(12,-4)}]
\draw[] (-6,-2) rectangle (6,2);
\fill[black!15!white] (0,0) circle (0.55); \draw (0,0) circle (0.55); \draw[thick] (-0.06,-0.06)--(0.06,0.06); \draw[thick] (0.06,-0.06)--(-0.06,0.06);
\foreach \i in {45,135,225,315} {
\fill[black!40!white] ({1.175*cos(\i)},{1.175*sin(\i)}) circle (0.625); \draw ({1.175*cos(\i)},{1.175*sin(\i)}) circle (0.625); \fill ({1.175*cos(\i)},{1.175*sin(\i)}) circle (0.07);
\draw[thick] ({0.475*cos(\i)},{0.475*sin(\i)})--({0.625*cos(\i)},{0.625*sin(\i)});
}
\end{scope}

\begin{scope}[scale=0.6,shift={(20,-4)}]
\draw[] (-2,-2) rectangle (2,2);
\draw[thick] (-0.06,-0.06)--(0.06,0.06); \draw[thick] (0.06,-0.06)--(-0.06,0.06);
\foreach \i in {45,135,225,315} {
\fill[black!40!white] ({0.7/cos(45)*cos(\i)},{0.7/cos(45)*sin(\i)}) circle (0.7); \draw ({0.7/cos(45)*cos(\i)},{0.7/cos(45)*sin(\i)}) circle (0.7); \fill ({0.7/cos(45)*cos(\i)},{0.7/cos(45)*sin(\i)}) circle (0.07);
}
\foreach \i in {0,90,180,270} {
\draw[thick] ({0.7*cos(\i)+0.075*cos(\i+90)},{0.7*sin(\i)+0.075*sin(\i+90)})--({0.7*cos(\i)-0.075*cos(\i+90)},{0.7*sin(\i)-0.075*sin(\i+90)});
}
\draw ({0.7/cos(45)*cos(45)},{0.7/cos(45)*sin(45)+0.1}) node[below]{\scriptsize ${\bf u}_k$};
\draw[thick] ({0.7*cos(90)},{0.7*sin(90)-0.2}) node[above]{\scriptsize ${\bf u}_{k,\ell}$};
\draw ({0.7/cos(45)*cos(135)},{0.7/cos(45)*sin(135)+0.1}) node[below]{\scriptsize ${\bf u}_\ell$};
\end{scope}

\begin{scope}[scale=0.6,shift={(0,-10)}]
\draw (-2,-2) rectangle (2,2);
\fill[black!40!white] (0,0) circle (1); \draw (0,0) circle (1); \draw[thick] (-0.06,-0.06)--(0.06,0.06); \draw[thick] (0.06,-0.06)--(-0.06,0.06);
\foreach \i in {36,108,180,252,324} {
\fill[black!15!white] ({1.4*cos(\i)},{1.4*sin(\i)}) circle (0.4); \draw ({1.4*cos(\i)},{1.4*sin(\i)}) circle (0.4); \fill ({1.4*cos(\i)},{1.4*sin(\i)}) circle (0.07);
\draw[thick] ({0.925*cos(\i)},{0.925*sin(\i)})--({1.075*cos(\i)},{1.075*sin(\i)});
}
\end{scope}

\begin{scope}[scale=0.6,shift={(4,-10)}]
\draw (-2,-2) rectangle (2,2);
\fill[black!40!white] (0,0) circle (0.6); \draw (0,0) circle (0.6); \draw[thick] (-0.06,-0.06)--(0.06,0.06); \draw[thick] (0.06,-0.06)--(-0.06,0.06);
\foreach \i in {36,108,180,252,324} {
\fill[black!40!white] ({1.2*cos(\i)},{1.2*sin(\i)}) circle (0.6); \draw ({1.2*cos(\i)},{1.2*sin(\i)}) circle (0.6); \fill ({1.2*cos(\i)},{1.2*sin(\i)}) circle (0.07);
\draw[thick] ({0.525*cos(\i)},{0.525*sin(\i)})--({0.675*cos(\i)},{0.675*sin(\i)});
}
\end{scope}

\begin{scope}[scale=0.6,shift={(8,-10)}]
\draw (-2,-2) rectangle (2,2);
\fill[black!15!white] (0,0) circle (0.55); \draw (0,0) circle (0.55); \draw[thick] (-0.06,-0.06)--(0.06,0.06); \draw[thick] (0.06,-0.06)--(-0.06,0.06);
\foreach \i in {36,108,180,252,324} {
\fill[black!40!white] ({1.175*cos(\i)},{1.175*sin(\i)}) circle (0.625); \draw ({1.175*cos(\i)},{1.175*sin(\i)}) circle (0.625); \fill ({1.175*cos(\i)},{1.175*sin(\i)}) circle (0.07);
\draw[thick] ({0.475*cos(\i)},{0.475*sin(\i)})--({0.625*cos(\i)},{0.625*sin(\i)});
}
\end{scope}

\begin{scope}[scale=0.6,shift={(12,-10)}]
\draw (-2,-2) rectangle (2,2);
\fill[black!15!white] (0,0) circle ({0.65/cos(54)-0.65}); \draw (0,0) circle ({0.65/cos(54)-0.65}); \draw[thick] (-0.06,-0.06)--(0.06,0.06); \draw[thick] (0.06,-0.06)--(-0.06,0.06);
\foreach \i in {36,108,180,252,324} {
\fill[black!40!white] ({0.65/cos(54)*cos(\i)},{0.65/cos(54)*sin(\i)}) circle (0.65); \draw ({0.65/cos(54)*cos(\i)},{0.65/cos(54)*sin(\i)}) circle (0.65); \fill ({0.65/cos(54)*cos(\i)},{0.65/cos(54)*sin(\i)}) circle (0.07);
\draw[thick] ({(0.65/cos(54)-0.65-0.075)*cos(\i)},{(0.65/cos(54)-0.65-0.075)*sin(\i)})--({(0.65/cos(54)-0.65+0.075)*cos(\i)},{(0.65/cos(54)-0.65+0.075)*sin(\i)});
}
\foreach \i in {0,72,144,216,288} {
\draw[thick] ({0.65/tan(36)*cos(\i)+0.075*cos(\i+90)},{0.65/tan(36)*sin(\i)+0.075*sin(\i+90)})--({0.65/tan(36)*cos(\i)-0.075*cos(\i+90)},{0.65/tan(36)*sin(\i)-0.075*sin(\i+90)});
}
\end{scope}

\begin{scope}[scale=0.6,shift={(16,-10)}]
\draw (-2,-2) rectangle (2,2);
\fill[black!15!white] (0,0) circle (0.3); \draw (0,0) circle (0.3); \draw[thick] (-0.06,-0.06)--(0.06,0.06); \draw[thick] (0.06,-0.06)--(-0.06,0.06);
\foreach \i in {36,108,180,252,324} {
\fill[black!40!white] ({0.675/cos(54)*cos(\i)},{0.675/cos(54)*sin(\i)}) circle (0.675); \draw ({0.675/cos(54)*cos(\i)},{0.675/cos(54)*sin(\i)}) circle (0.675); \fill ({0.675/cos(54)*cos(\i)},{0.675/cos(54)*sin(\i)}) circle (0.07);
\fill[blue] ({0.3*cos(\i)-0.05},{0.3*sin(\i)-0.05}) rectangle ({0.3*cos(\i)+0.05},{0.3*sin(\i)+0.05});
}
\foreach \i in {0,72,144,216,288} {
\draw[thick] ({0.675/tan(36)*cos(\i)+0.075*cos(\i+90)},{0.675/tan(36)*sin(\i)+0.075*sin(\i+90)})--({0.675/tan(36)*cos(\i)-0.075*cos(\i+90)},{0.675/tan(36)*sin(\i)-0.075*sin(\i+90)});
}
\end{scope}

\begin{scope}[scale=0.6,shift={(20,-10)}]
\draw (-2,-2) rectangle (2,2);
\draw[thick] (-0.06,-0.06)--(0.06,0.06); \draw[thick] (0.06,-0.06)--(-0.06,0.06);
\foreach \i in {36,108,180,252,324} {
\fill[black!40!white] ({0.7/cos(54)*cos(\i)},{0.7/cos(54)*sin(\i)}) circle (0.7); \draw ({0.7/cos(54)*cos(\i)},{0.7/cos(54)*sin(\i)}) circle (0.7); \fill ({0.7/cos(54)*cos(\i)},{0.7/cos(54)*sin(\i)}) circle (0.07);
}
\foreach \i in {0,72,144,216,288} {
\draw[thick] ({0.7/tan(36)*cos(\i)+0.075*cos(\i+90)},{0.7/tan(36)*sin(\i)+0.075*sin(\i+90)})--({0.7/tan(36)*cos(\i)-0.075*cos(\i+90)},{0.7/tan(36)*sin(\i)-0.075*sin(\i+90)});
}
\end{scope}
\end{tikzpicture}
\par\end{centering}
\caption{\label{fig1.2}Energy landscape of the CWP model if $\beta>\beta_{1}$.
The (cross-marked) middle point represents the equiproportional vector
${\bf e}$ and the $q$ outside (dot-marked) circle centers represent
the local minima ${\bf u}_{1},\dots,{\bf u}_{q}$. Each gray-colored
circle represents an energetic well $\mathcal{W}_{k}$, $k\in\llbracket0,q\rrbracket$,
such that dark color indicates the deepest wells and light color indicates
shallower wells. The saddle points with height $H_{\beta}$ are marked
between the wells. In the case of $q\ge5$ and $\beta\in(\beta_{3},q)$,
the $q$ saddle points ${\bf v}_{1},\dots,{\bf v}_{q}$ between $\mathcal{W}_{0}$
and $\widehat{\mathcal{W}}_{1}$ are marked as blue squares. The metastable
valleys $\mathcal{E}_{N}^{k}$, $k\in\llbracket0,q\rrbracket$ (cf.
\eqref{eq:ENk-def} and \eqref{eq:EN0-def}) are marked by red dashed
lines for $q=3$. If $\beta\ge\beta_{3}$ and $q\ge4$, one should
interpret that $\overline{\mathcal{W}}_{k}\cap\overline{\mathcal{W}}_{\ell}=\{{\bf u}_{k,\ell}\}\protect\ne\emptyset$
for any $k\protect\ne\ell$, even though ${\bf u}_{k,\ell}$ is illustrated
only if $k,\ell$ are nearest neighbors (due to dimensional restriction).}
\end{figure}
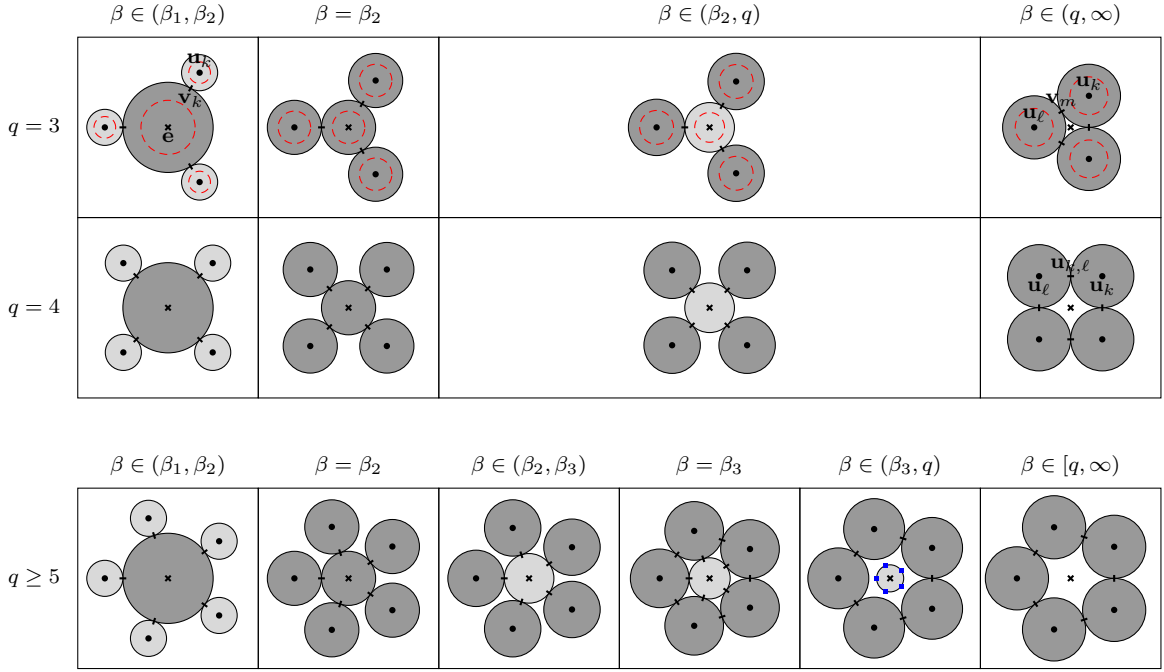
\par\end{center}

\subsection{\label{sec1.5}Glauber Dynamics}

For $\sigma\in\Omega_{N}$, $v\in K_{N}$, and $k\in\llbracket1,q\rrbracket$,
denote by $\sigma^{v,k}$ the configuration whose spin $\sigma_{v}$
at site $v$ is flipped to $k$, i.e.,
\[
\left(\sigma^{v,k}\right)_{u}:=\begin{cases}
\sigma_{u} & \text{if}\quad u\ne v,\\
k & \text{if}\quad u=v.
\end{cases}
\]
Then, consider an infinitesimal generator $\mathcal{L}_{N}$ which
acts on each $f:\Omega_{N}\to\mathbb{R}$ as
\[
(\mathcal{L}_{N}f)(\sigma):=\frac{1}{N}\sum_{v=1}^{N}\sum_{k=1}^{q}c_{v,k}(\sigma)\left(f\left(\sigma^{v,k}\right)-f(\sigma)\right),
\]
where (cf. \eqref{eq:Hamiltonian})
\begin{equation}
c_{v,k}(\sigma):=\sqrt{\frac{\mu_{N}^{\beta}(\sigma^{v,k})}{\mu_{N}^{\beta}(\sigma)}}=\exp\left\{ -\frac{\beta}{2}\left[\mathbb{H}_{N}\left(\sigma^{v,k}\right)-\mathbb{H}_{N}(\sigma)\right]\right\} .\label{eq:rate-def}
\end{equation}
It can be observed that this dynamics is reversible with respect to
the CWP measure $\mu_{N}^{\beta}$ (cf. \eqref{eq:muNbeta}). Henceforth,
denote by $\{\bm{\sigma}_{N}^{\beta}(t)\}_{t\ge0}$ the continuous-time
Markov chain in $\Omega_{N}$ associated with the generator $\mathcal{L}_{N}$,
which is a type of\emph{ Glauber dynamics} in $\Omega_{N}$. Denote
by $\mathbb{P}_{\sigma}^{N,\beta}$ and $\mathbb{E}_{\sigma}^{N,\beta}$
the law and the corresponding expectation of the process starting
from $\sigma\in\Omega_{N}$.

\subsection{\label{sec1.6}Main Result}

Assume $\beta\in(\beta_{1},\infty)$ such that $\beta\ne q$ if $q\in\{3,4\}$.
Let $H_{\beta}$ denote the height of the lowest saddle points between
${\bf u}_{k}$, $k\in\llbracket1,q\rrbracket$ (cf. Figures \ref{fig1.1}
and \ref{fig1.2}). More specifically:
\begin{itemize}
\item if $q=2$ then $H_{\beta}:=F_{\beta}({\bf e})$;
\item if $q=3$ then $H_{\beta}:=F_{\beta}({\bf v}_{1})$;
\item if $q=4$ then $H_{\beta}:=F_{\beta}({\bf v}_{1})$ if $\beta<q$
and $H_{\beta}:=F_{\beta}({\bf u}_{1,2})$ if $\beta>q$;
\item if $q\ge5$ then $H_{\beta}:=F_{\beta}({\bf v}_{1})$ if $\beta<\beta_{3}$
and $H_{\beta}:=F_{\beta}({\bf u}_{1,2})$ if $\beta\ge\beta_{3}$.
\end{itemize}
Denote by $D_{\beta}$ the depth of each $\mathcal{W}_{k}$, $k\in\llbracket1,q\rrbracket$:
\begin{equation}
D_{\beta}:=H_{\beta}-F_{\beta}({\bf u}_{1}).\label{eq:Dbeta-def}
\end{equation}
The following main result of this article presents a sharp estimate
of the mixing time of the dynamics $\bm{\sigma}_{N}^{\beta}$ as $N\to\infty$.
Recall that $T_{\delta}^{{\rm mix}}(\bm{\sigma}_{N}^{\beta})$ denotes
the $\delta$-mixing time of $\bm{\sigma}_{N}^{\beta}$.
\begin{thm}
\label{thm:main}Fix $\delta>0$. For all $\beta>\beta_{1}$ such
that $\beta\ne q$ if $q\in\{3,4\}$,
\[
\lim_{N\to\infty}\frac{T_{\delta}^{{\rm mix}}\left(\bm{\sigma}_{N}^{\beta}\right)}{2\pi Ne^{ND_{\beta}}}=\mathfrak{T}(\delta),
\]
where $\mathfrak{T}(\delta)=\mathfrak{T}(\beta,q,\delta)$ is the
$\delta$-mixing time of another dynamics $\{\mathfrak{X}_{\beta}(t)\}_{t\ge0}$
which is rigorously defined in \eqref{eq:Tdelta}.
\end{thm}

Theorem \ref{thm:main} is proved at the end of Section \ref{sec3}.
\begin{rem}
\label{rem:other}We may consider other types of Glauber dynamics
for the CWP model as well. Two important examples are as follows:
\begin{itemize}
\item \emph{heat-bath} Glauber dynamics studied in \cite{CDLLPS12,LLP10}:
\begin{equation}
c_{v,k}^{{\rm HB}}(\sigma):=\frac{\mu_{N}^{\beta}\left(\sigma^{v,k}\right)}{\sum_{\ell=1}^{q}\mu_{N}^{\beta}\left(\sigma^{v,\ell}\right)}=\frac{e^{-\beta\mathbb{H}_{N}\left(\sigma^{v,k}\right)}}{\sum_{\ell=1}^{q}e^{-\beta\mathbb{H}_{N}\left(\sigma^{v,\ell}\right)}};\label{eq:rate-HB-def}
\end{equation}
\item \emph{Metropolis} dynamics studied in \cite{KS25,NS91}:
\begin{equation}
c_{v,k}^{{\rm MP}}(\sigma):=\frac{\mu_{N}^{\beta}\left(\sigma^{v,k}\right)}{\mu_{N}^{\beta}(\sigma)}\wedge1=e^{-\beta\left[\mathbb{H}_{N}\left(\sigma^{v,k}\right)-\mathbb{H}_{N}(\sigma)\right]_{+}}.\label{eq:rate-MP-def}
\end{equation}
\end{itemize}
For these models, our main result Theorem \ref{thm:main} remains
valid. The only difference for these alternative models is the limit
dynamics $\mathfrak{X}_{\beta}$; we discuss this issue briefly in
Section \ref{secC}.
\end{rem}

\subsection{\label{sec1.7}History and Discussions}

\subsubsection*{High-Temperature Regime: Fast Mixing and Cutoff Phenomenon}

In the high-temperature regime of $\beta\in(0,\beta_{1})$, precise
mixing time estimates were derived in \cite{LLP10} and \cite{CDLLPS12},
for the $q=2$ (Curie--Weiss) and $q\ge3$ (Curie--Weiss--Potts)
cases, respectively. As summarized in Section \ref{sec1.4}, if $\beta\in(0,\beta_{1})$,
the equiproportional vector ${\bf e}$ is the unique local (thus global)
minimum of $F_{\beta}$, and thus the Gibbs measure $\mu_{N}^{\beta}$
is concentrated on the disordered configurations. Due to this uniqueness,
following the Glauber dynamics, the magnetization converges rapidly
to near this minimum ${\bf e}$ and, consequently, it exhibits a logarithmic
\emph{fast} mixing; for any $\delta>0$, $T_{\delta}^{{\rm mix}}(\bm{\sigma}_{N}^{\beta})\simeq C(\beta,q)N\log N$
as $N\to\infty$ for some explicit constant $C(\beta,q)$.

One may understand this mixing mechanism in a more precise manner
near the typical scale $T_{N}:=C(\beta,q)N\log N$ in the following
sense: for any $\delta\in(0,1)$,
\begin{equation}
\lim_{N\to\infty}\frac{T_{\delta}^{{\rm mix}}\left(\bm{\sigma}_{N}^{\beta}\right)}{T_{1-\delta}^{{\rm mix}}\left(\bm{\sigma}_{N}^{\beta}\right)}=1.\label{eq:cutoff-def}
\end{equation}
In other words, the process exhibits a sharp mixing time phase transition
near $T_{N}$; in this regard, we call this \emph{cutoff} phenomenon
of mixing times. Further, it is shown \cite{CDLLPS12,LLP10} that
the process exhibits a cutoff with a \emph{window} of size $\alpha_{N}=N$,
i.e.,
\[
\alpha_{N}\ll T_{N}\qquad\text{and}\qquad\begin{cases}
\lim_{c\to-\infty}\lim_{N\to\infty}d_{{\rm TV}}\left(\bm{\sigma}_{N}^{\beta}(T_{N}+c\alpha_{N}),\mu_{N}^{\beta}\right)=1,\\
\lim_{c\to\infty}\lim_{N\to\infty}d_{{\rm TV}}\left(\bm{\sigma}_{N}^{\beta}(T_{N}+c\alpha_{N}),\mu_{N}^{\beta}\right)=0.
\end{cases}
\]
In this regard, we observe as a byproduct that the process does not
exhibit a cutoff phenomenon in the low-temperature regime:
\begin{cor}
\label{cor:no-cutoff}If $\beta\in(\beta_{1},\infty)$ such that $\beta\ne q$
if $q\in\{3,4\}$, the Glauber dynamics for the CWP model does not
exhibit a cutoff phenomenon in the sense of \eqref{eq:cutoff-def}.
\end{cor}

It is natural to expect that the corollary holds true since, by Theorem
\ref{thm:main}, the mixing behavior of our dynamics is described
asymptotically by a single limit Markov chain $\mathfrak{X}_{\beta}$,
thus the total variation distance to stationarity would undergo a
continuous transition from (near) $1$ to $0$ as time runs from $0$
to infinity. We provide a proof of Corollary \ref{cor:no-cutoff}
at the end of Section \ref{sec3}.

\subsubsection*{At Critical Temperature}

The critical behavior at $\beta=\beta_{1}$ was also studied in \cite{CDLLPS12,LLP10},
which requires a much more refined analysis near the typical mixing
time:
\[
\begin{cases}
c_{1}N^{\frac{3}{2}}\le T_{\delta}^{{\rm mix}}\left(\bm{\sigma}_{N}^{\beta}\right)\le c_{2}N^{\frac{3}{2}} & \text{if}\quad q=2,\\
c_{1}N^{\frac{4}{3}}\le T_{\delta}^{{\rm mix}}\left(\bm{\sigma}_{N}^{\beta}\right)\le c_{2}N^{\frac{4}{3}} & \text{if}\quad q\ge3,
\end{cases}
\]
where $c_{1}=c_{1}(\beta,q,\delta)>0$ and $c_{2}=c_{2}(\beta,q,\delta)>0$
are constants. In particular, the process does not exhibit cutoff
and alternatively has a scaling window of $N^{-\frac{1}{2}}$ (if
$q=2$) or $N^{-\frac{2}{3}}$ (if $q\ge3$).

\subsubsection*{Low-Temperature Regime: Slow Mixing and Metastability}

In contrast, in the low-temperature regime, that is, when $\beta>\beta_{1}$,
the energy landscape possesses multiple local minima. As the Gibbs
measure is concentrated around two or more configurations, the Glauber
dynamics exhibits metastability. Since transitions between these configurations
occur on exponentially long time-scales, the dynamics would have (exponentially)
slow mixing. It was shown \cite{CDLLPS12} that there exist constants
$c_{1}=c_{1}(\beta,q,\delta)>0$ and $c_{2}=c_{2}(\beta,q,\delta)>0$
such that
\[
T_{\delta}^{{\rm mix}}\left(\bm{\sigma}_{N}^{\beta}\right)\ge c_{1}e^{c_{2}N}\qquad\text{if}\quad\beta>\beta_{1}.
\]
However, it was unable to pinpoint the exact exponential scale nor
its subexponential prefactor in the limit $N\to\infty$. To achieve
this, a comprehensive metastability framework is required. We explain
this in a full detail in Section \ref{sec2}.

\subsubsection*{Idea of Proof}

Since the Gibbs measure is concentrated near multiple number of metastable
states, in order to approach stationarity, the dynamics must undergo
a sufficient number of transitions between these states. To analyze
this behavior, we rely on the theory of metastability, in particular
the framework of Markov chain model reduction developed by Claudio
Landim and his collaborators \cite{BL10,BL12b,BL15,Lan15,Lan19,LLM18,LMS25}.
The proof is based on the following three key ingredients.
\begin{itemize}
\item Recurrence property (cf. Theorem \ref{thm:CWP-Rec}): Independently
of the initial configuration, the dynamics quickly enters one of the
metastable sets on which the Gibbs measure is essentially concentrated.
\item Local mixing property (cf. Theorem \ref{thm:CWP-M}): Starting from
a configuration in a metastable set, the distribution of the dynamics
relaxes to the Gibbs measure conditioned on that set before any transition
to another metastable set.
\item Markov chain model reduction (cf. Theorem \ref{thm:CWP-CD}): In the
time-scale of transitions between metastable sets, the original dynamics
can be approximated by a reduced Markov chain on the set of indices
of metastable states. Since local equilibration occurs prior to each
transition, the global mixing of the dynamics is governed by such
successive transitions. As a consequence, the mixing time of the original
dynamics is asymptotically given by the mixing time of the reduced
Markov chain multiplied by the transition time-scale.
\end{itemize}
For a detailed discussion on the relation between metastability and
slow mixing, we refer to \cite{LLM18,Lee25}.

\subsubsection*{Organization of the Article}

In Section \ref{sec2}, we present general strategy to study slow
mixing via theory of metastability. In Section \ref{sec3}, we review
metastability result for the CWP model, and prove Theorem \ref{thm:main}
and Corollary \ref{cor:no-cutoff} by assuming two properties: recurrence
property (Theorem \ref{thm:CWP-Rec}) and local mixing property (Theorem
\ref{thm:CWP-M}). Finally, in Sections \ref{sec4} and \ref{sec5},
we prove two theorems respectively.

\section{\label{sec2}General Strategy for Slow-Mixing Systems}

In this section, we present our general strategy to prove a sharp
mixing time estimate, as in Theorem \ref{thm:main}, for a general
class of slowly mixing metastable processes.

\subsection{\label{sec2.1}Review on Metastability}

As explained in the previous section, our starting point is the idea
of \emph{Markov chain model reduction} applied to metastable systems,
initiated in \cite{BL10} and widely studied during the past two decades
by Claudio Landim and his many collaborators. We refer to \cite{Lan19}
for an extensive literature and a detailed description of this approach.

In this subsection, we review previous works on the Markov chain model
reduction. To emphasize that the results of this section are universal
and applicable to a large class of models, we implement the notation
from the previous section but without the subscript (or superscript)
$\beta$.

Fix a continuous-time Markov chain $\{\bm{\sigma}_{N}(t)\}_{t\ge0}$
in a finite set $\Omega_{N}$, indexed by $N\ge1$, and denote by
$\mathbb{P}_{\sigma}^{N}$ and $\mathbb{E}_{\sigma}^{N}$, $\sigma\in\Omega_{N}$,
the law of $\bm{\sigma}_{N}$ starting from $\sigma$ and its corresponding
expectation, respectively. Suppose that there exist a finite index
set $\mathfrak{S}$ and a collection of disjoint sets $\mathcal{E}_{N}^{k}\subset\Omega_{N}$
for $k\in\mathfrak{S}$. Write $\mathcal{E}_{N}=\bigcup_{k\in\mathfrak{S}}\mathcal{E}_{N}^{k}$
and $\Delta_{N}=\Omega_{N}\setminus\mathcal{E}_{N}$. Define a projection
function $\Psi_{N}:\mathcal{E}_{N}\to\mathfrak{S}$ by declaring $\Psi_{N}(\sigma):=k$
for all $\sigma\in\mathcal{E}_{N}^{k}$. In addition, suppose that
a Markov chain $\{\mathfrak{X}(t)\}_{t\ge0}$ in $\mathfrak{S}$ is
given.
\begin{defn}[Description of Metastability]
 We say that the \emph{metastability of the collection of processes
$\bm{\sigma}_{N}$, $N\ge1$ is described by the limit Markov chain
$\mathfrak{X}$ in the time-scale of $\theta_{N}$} if the following
three conditions $\mathfrak{M}$, $\mathfrak{C}$, and $\mathfrak{D}$
are valid.
\end{defn}

The first condition $\mathfrak{M}$, the local mixing condition, states
that the distribution of the dynamics is well mixed inside each metastable
well before the exit. Suppose that $\bm{\sigma}_{N}$ has a unique
stationary distribution denoted by $\mu_{N}$. For $\mathcal{A}\subset\Omega_{N}$,
let $\{\bm{\sigma}_{N}^{\mathcal{A}}(t)\}_{t\ge0}$ be the Markov
process obtained from $\bm{\sigma}_{N}$ by forbidding any escapes
from $\mathcal{A}$, i.e., the process obtained by setting all jump
rates from $\mathcal{A}$ to $\Omega_{N}\setminus\mathcal{A}$ to
zero. We call this the \emph{reflected dynamics} of $\bm{\sigma}_{N}$
in $\mathcal{A}$. Let $\mu_{N}^{\mathcal{A}}$ denote the probability
measure $\mu_{N}$ conditioned on $\mathcal{A}$, i.e.,
\[
\mu_{N}^{\mathcal{A}}(\sigma):=\frac{\mu_{N}(\sigma)}{\mu_{N}(\mathcal{A})}\qquad\text{for}\quad\sigma\in\mathcal{A}.
\]
Assume that $\mu_{N}^{\mathcal{E}_{N}^{k}}$ is the unique stationary
distribution of the reflected dynamics $\bm{\sigma}_{N}^{\mathcal{E}_{N}^{k}}$
for each $k\in\mathfrak{S}$. This is obvious if, e.g., the original
dynamics $\bm{\sigma}_{N}$ is reversible and each set $\mathcal{E}_{N}^{k}$
is connected, i.e., if any point in $\mathcal{E}_{N}^{k}$ is reachable
via $\bm{\sigma}_{N}$ from any other point in $\mathcal{E}_{N}^{k}$
without leaving $\mathcal{E}_{N}^{k}$.
\begin{itemize}
\item \textbf{{[}$\mathfrak{M}$: (Local) Mixing{]}} There exists a subset
$\mathcal{B}_{N}^{k}\subset\mathcal{E}_{N}^{k}$ for each $k\in\mathfrak{S}$
such that:
\begin{enumerate}
\item for every $\delta>0$,
\[
\lim_{N\to\infty}\sup_{k\in\mathfrak{S}}\sup_{\sigma\in\mathcal{E}_{N}^{k}}\mathbb{P}_{\sigma}^{N}\left[\mathcal{H}_{\mathcal{B}_{N}^{k}}>\delta\theta_{N}\right]=0;
\]
\item there exists a time-scale $(\varrho_{N})_{N\ge1}$ with $\varrho_{N}\ll\theta_{N}$
such that
\[
\lim_{N\to\infty}\sup_{k\in\mathfrak{S}}\sup_{\sigma\in\mathcal{B}_{N}^{k}}\mathbb{P}_{\sigma}^{N}\left[\mathcal{H}_{\Omega_{N}\setminus\mathcal{E}_{N}^{k}}\le2\varrho_{N}\right]=0;
\]
\item and
\[
\lim_{N\to\infty}\sup_{k\in\mathfrak{S}}\sup_{\sigma\in\mathcal{B}_{N}^{k}}d_{{\rm TV}}\left(\bm{\sigma}_{N}^{\mathcal{E}_{N}^{k}}(\varrho_{N};\sigma),\mu_{N}^{\mathcal{E}_{N}^{k}}\right)=0.
\]
\end{enumerate}
\end{itemize}
The first display states that for each $k\in\mathfrak{S}$ there exists
a \emph{deeper} subset $\mathcal{B}_{N}^{k}\subset\mathcal{E}_{N}^{k}$
which is reached asymptotically instantaneously in the time-scale
of $\theta_{N}$. The second display states that, starting from $\mathcal{B}_{N}^{k}$,
the probability to escape $\mathcal{E}_{N}^{k}$ in the time-scale
of some $\varrho_{N}\ll\theta_{N}$ is asymptotically zero, thus the
trajectories of the original dynamics $\bm{\sigma}_{N}$ and the reflected
dynamics $\bm{\sigma}_{N}^{\mathcal{E}_{N}^{k}}$ can be identified
(via the canonical coupling) until time $2\varrho_{N}$ with high
probability. The third display states that the reflected dynamics
$\bm{\sigma}_{N}^{\mathcal{E}_{N}^{k}}$ mixes well inside $\mathcal{E}_{N}^{k}$
in the time-scale of $\varrho_{N}$. Thus, along with the second display,
$\bm{\sigma}_{N}$ also mixes well inside $\mathcal{E}_{N}^{k}$ before
exiting. Thus, combining all three statements, each exit from $\mathcal{E}_{N}^{k}$
is asymptotically Markovian.

The second and third conditions $\mathfrak{C}$ and $\mathfrak{D}$
constitute the model reduction part. Consider the \emph{trace} process
$\{\widetilde{\bm{\sigma}}_{N}(t)\}_{t\ge0}$ in $\mathcal{E}_{N}$,
which is defined by turning off the clock outside the set $\mathcal{E}_{N}$
(see \cite[Section 6.1]{BL10} for a rigorous definition).
\begin{itemize}
\item \textbf{{[}$\mathfrak{C}$: Convergence{]}} For every $k\in\mathfrak{S}$
and any sequence $(\sigma_{N})_{N}$ in $\mathcal{E}_{N}^{k}$, the
law of the accelerated \emph{order} process $\{\Psi_{N}(\widetilde{\bm{\sigma}}_{N}(\theta_{N}t))\}_{t\ge0}$
in $\mathfrak{S}$ starting from $\sigma_{N}$ converges weakly, as
$N\to\infty$, to the law of the limit process $\{\mathfrak{X}(t)\}_{t\ge0}$
starting from $k$. 
\item \textbf{{[}$\mathfrak{D}$: (Delta-)Negligibility{]}} The time spent
outside $\mathcal{E}_{N}$ is negligible: for all $T>0$ and $k\in\mathfrak{S}$,
\[
\lim_{N\to\infty}\sup_{\sigma\in\mathcal{E}_{N}^{k}}\mathbb{E}_{\sigma}^{N}\left[\int_{0}^{T}\bm{1}_{\left\{ \bm{\sigma}_{N}(\theta_{N}t)\in\Delta_{N}\right\} }\,{\rm d}t\right]=0.
\]
\end{itemize}
In words, condition $\mathfrak{C}$ indicates that the metastable
jumps between the sets $\mathcal{E}_{N}^{k}$, $k\in\mathfrak{S}$
are described by the limit Markov chain $\mathfrak{X}$. In addition,
condition $\mathfrak{D}$ means that, asymptotically, the process
does not spend time in the remainder set $\Delta_{N}$, so that $\mathcal{E}_{N}$
is indeed the collection of all possible metastable states in the
time-scale of $\theta_{N}$.
\begin{rem}
Quantitative study of metastability has been widely investigated especially
in the past two decades. We refer to the monographs \cite{BdH15,Lan19}
for a comprehensive review of metastability, and remark a few of the
important milestones below.

Markov chain model reduction provides a systematic way to describe
successive metastable transitions via a reduced Markov chain, which
is significantly simpler than the original system. By accelerating
the dynamics on an appropriate metastable time-scale, the law of the
trace process converges, in the Skorokhod topology, to the law of
this reduced Markov chain. This approach was developed for the supercritical
condensed zero-range process by Beltr\'an and Landim \cite{BL12a},
and subsequently extended to a general framework in \cite{BL10,BL12b,BL15}
of establishing the Markov chain convergence via uniqueness of the
martingale problem. A more general hierarchical structure of metastability
for metastable Markov chains was later studied in \cite{BL11,LX16}.

Potential-theoretic approach to metastability, initiated in a series
of papers \cite{BEGK02,BEGK04,BGK05}, enabled numerous breakthroughs
in the study of metastability, especially the quantitative results.
Refer to \cite{BdH15} for an intensive list of publications. Gaudilli\`ere
and Landim \cite{GL14} developed a non-reversible potential theory
based on a generalized Dirichlet principle, and this development enabled
to study genuinely non-reversible interacting particle systems \cite{KS21,Lan14,Seo19}.

As pointed out in condition $\mathfrak{M}$, an \emph{a priori} local
mixing property is mandatory to discuss any types of metastable transitions.
To address a large class of models in which strong local mixing condition
fails, Landim, Marcondes and Seo \cite{LMS23,LMS25} developed the
so-called \emph{resolvent approach} to metastability, which extracts
model reduction directly from asymptotic properties of solutions to
resolvent equations.

Overall, the theory has been successfully applied to a wide range
of systems including random walks on potential fields \cite{LMT15,LS18},
interacting particle systems \cite{BDG17,Kim21,Kim25,KS21,Seo19},
spin systems \cite{KS24,KS25,LS16,Lee22}, diffusion processes and
the associated parabolic equations \cite{LLS24a,LLS24b,LMS19,LM25,LS19}.

Recently, a remarkable relation between the model reduction theory
and the classical large deviations principle---namely, that metastable
time-scales and reduced Markov dynamics are encoded in a $\Gamma$-expansion
of the Donsker--Varadhan rate functionals for Markov chains---has
been developed and intensively studied in \cite{BGL22,KL25,Lan23,LLM25,LMS24}.
\end{rem}

According to these metastability descriptions $\mathfrak{M}$, $\mathfrak{C}$,
and $\mathfrak{D}$, the mixing behavior of the accelerated dynamics
$\{\bm{\sigma}_{N}(\theta_{N}t)\}_{t\ge0}$ would be close to that
of the limit process $\{\mathfrak{X}(t)\}_{t\ge0}$ for large $N$.
This suggests that the mixing time of the original dynamics $\bm{\sigma}_{N}$,
which would be very slow due to the presence of two or more metastable
sets $\mathcal{E}_{N}^{k}$, $k\in\mathfrak{S}$, should be described
by the corresponding mixing time of the limit process $\{\mathfrak{X}(t)\}_{t\ge0}$
multiplied by the time-scale of $\theta_{N}$. This is exactly the
content of Theorem \ref{thm:main}.

\subsection{\label{sec2.2}From Metastability to Mixing Time}

In order to prove the convergence of mixing times as presented in
Theorem \ref{thm:main}, we need to introduce an alternative mode
of convergence regarding the total variation distance between probability
measures. Denote by ${\bf P}_{k}$, $k\in\mathfrak{S}$, the law of
$\mathfrak{X}$ starting from $k$.
\begin{itemize}
\item \textbf{{[}$\mathfrak{C}_{{\rm TV}}$: Convergence in TV Distance{]}}
For any $k\in\mathfrak{S}$, $0\le t_{1}<t_{2}<\cdots<t_{n}$, sequences
$(\sigma_{N})_{N\ge1}$ and $(t_{j,N})_{N\ge1}$, $j\in\llbracket1,n\rrbracket$,
such that $\sigma_{N}\in\mathcal{E}_{N}^{k}$ and $t_{j,N}\to t_{j}$,
\[
\lim_{N\to\infty}d_{{\rm TV}}\left(\mathbb{P}_{\sigma_{N}}^{t_{1,N},\dots,t_{n,N}},\sum_{\ell_{1},\dots,\ell_{n}\in\mathfrak{S}}{\bf P}_{k}\left[\bigcap_{j=1}^{n}\{\mathfrak{X}(t_{j})=\ell_{j}\}\right]\bigotimes_{j=1}^{n}\mu_{N}^{\mathcal{E}_{N}^{\ell_{j}}}\right)=0,
\]
where $\mathbb{P}_{\sigma_{N}}^{t_{1,N},\dots,t_{n,N}}$ denotes the
\emph{accelerated} joint law of $(\bm{\sigma}_{N}(\theta_{N}t_{1,N}),\dots,\bm{\sigma}_{N}(\theta_{N}t_{n,N}))$
starting from $\sigma_{N}$.
\end{itemize}
According to \cite[Proposition 2.2]{LLM18}, conditions $\mathfrak{M}$,
$\mathfrak{C}$, and $\mathfrak{D}$ imply $\mathfrak{C}_{{\rm TV}}$:
\begin{prop}[{\cite[Proposition 2.2]{LLM18}}]
\label{prop:LLM}Suppose that conditions $\mathfrak{M}$, $\mathfrak{C}$,
and $\mathfrak{D}$ are in force. Furthermore, assume that
\begin{equation}
\lim_{N\to\infty}\frac{\mu_{N}(\Delta_{N})}{\mu_{N}\left(\mathcal{E}_{N}^{k}\right)}=0\qquad\text{for each}\quad k\in\mathfrak{S},\label{eq:LLM-muN}
\end{equation}
and that $\bm{\sigma}_{N}$ is reversible. Then, condition $\mathfrak{C}_{{\rm TV}}$
holds.
\end{prop}

\begin{rem}
The reversibility is not necessary, and one may check alternative
conditions such that the proposition still holds true. See \cite[eqs. (2.11) or (2.12)]{LLM18}.
\end{rem}

Now, we present a slightly different condition which states that the
convergence in $N\to\infty$ occurs for the total variation distance
between the dynamics and its stationary distribution. Suppose that
the limit chain $\{\mathfrak{X}(t)\}_{t\ge0}$ admits a stationary
distribution denoted by $\pi$.
\begin{itemize}
\item \textbf{{[}$\mathfrak{C}_{{\rm TV2}}(\pi)$: Convergence in TV Distance
2{]}} For all $k\in\mathfrak{S}$, $t>0$, sequences $(\sigma_{N})_{N\ge1}$
and $(t_{N})_{N\ge1}$ such that $\sigma_{N}\in\mathcal{E}_{N}^{k}$
and $t_{N}\to t$,
\[
\lim_{N\to\infty}d_{{\rm TV}}(\bm{\sigma}_{N}(\theta_{N}t_{N};\sigma_{N}),\mu_{N})=d_{{\rm TV}}(\mathfrak{X}(t;k),\pi).
\]
\end{itemize}
\begin{prop}
\label{prop:TV-TV2}Suppose that condition $\mathfrak{C}_{{\rm TV}}$
holds and that
\begin{equation}
\lim_{N\to\infty}\mu_{N}\left(\mathcal{E}_{N}^{k}\right)=\pi(k)\qquad\text{for all}\quad k\in\mathfrak{S}.\label{eq:TV-TV2-1}
\end{equation}
Then, condition $\mathfrak{C}_{{\rm TV2}}(\pi)$ holds.
\end{prop}

\begin{proof}
Fix $k\in\mathfrak{S}$, $t>0$, sequences $(\sigma_{N})_{N\ge1}$
and $(t_{N})_{N\ge1}$ such that $\sigma_{N}\in\mathcal{E}_{N}^{k}$
and $t_{N}\to t$. By condition $\mathfrak{C}_{{\rm TV}}$,
\begin{equation}
\lim_{N\to\infty}d_{{\rm TV}}\left(\bm{\sigma}_{N}(\theta_{N}t_{N};\sigma_{N}),\sum_{\ell\in\mathfrak{S}}{\bf P}_{k}[\mathfrak{X}(t)=\ell]\mu_{N}^{\mathcal{E}_{N}^{\ell}}\right)=0.\label{eq:TV-TV2-2}
\end{equation}
By the triangle inequality,
\begin{align*}
 & \left|d_{{\rm TV}}(\bm{\sigma}_{N}(\theta_{N}t_{N};\sigma_{N}),\mu_{N})-d_{{\rm TV}}\left(\sum_{\ell\in\mathfrak{S}}{\bf P}_{k}[\mathfrak{X}(t)=\ell]\mu_{N}^{\mathcal{E}_{N}^{\ell}},\mu_{N}\right)\right|\\
 & \le d_{{\rm TV}}\left(\bm{\sigma}_{N}(\theta_{N}t_{N};\sigma_{N}),\sum_{\ell\in\mathfrak{S}}{\bf P}_{k}[\mathfrak{X}(t)=\ell]\mu_{N}^{\mathcal{E}_{N}^{\ell}}\right).
\end{align*}
Thus by \eqref{eq:TV-TV2-2},
\[
\lim_{N\to\infty}d_{{\rm TV}}(\bm{\sigma}_{N}(\theta_{N}t_{N};\bm{x}_{N}),\mu_{N})=d_{{\rm TV}}\left(\sum_{\ell\in\mathfrak{S}}{\bf P}_{k}[\mathfrak{X}(t)=\ell]\mu_{N}^{\mathcal{E}_{N}^{\ell}},\mu_{N}\right).
\]
Combining with \eqref{eq:TV-TV2-1} and Lemma \ref{lem:TV-1},
\[
\lim_{N\to\infty}d_{{\rm TV}}(\bm{\sigma}_{N}(\theta_{N}t_{N};\bm{x}_{N}),\mu_{N})=\frac{1}{2}\sum_{\ell\in\mathfrak{S}}\left|{\bf P}_{k}[\mathfrak{X}(t)=\ell]-\pi(\ell)\right|=d_{{\rm TV}}(\mathfrak{X}(t;k),\pi).
\]
This concludes the proof of the proposition.
\end{proof}
\begin{rem}
In fact, the additional condition \eqref{eq:TV-TV2-1} readily follows
from condition $\mathfrak{C}$ if $\pi$ is the unique stationary
state. In a nutshell, this holds since $\mathfrak{C}$ implies that
the trace jump rates converge to the limit jump rates, and since $\pi$
is the unique stationary state, the stationary profile $\mu_{N}(\mathcal{E}_{N}^{k})$
also converges to $\pi(k)$ (cf. \cite[Lemma 6.1]{KS21}).
\end{rem}

The following property states that, starting from a configuration
in $\Delta_{N}=\Omega_{N}\setminus\mathcal{E}_{N}$, the process visits
$\mathcal{E}_{N}$ in a time-scale much smaller than the metastable
time-scale $\theta_{N}$.
\begin{itemize}
\item \textbf{{[}$\mathfrak{Rec}$: Recurrence Property{]}} There exists
a sequence $(\varrho_{N})_{N\ge1}$ with $\varrho_{N}\ll\theta_{N}$
such that
\[
\lim_{N\to\infty}\sup_{\sigma\in\Omega_{N}}\mathbb{P}_{\sigma}^{N}[\mathcal{H}_{\mathcal{E}_{N}}>\varrho_{N}]=0.
\]
\end{itemize}
Our main claim of this section is the following statement.
\begin{prop}
\label{prop:main}Suppose that $\pi$ is the unique stationary distribution
of the limit chain $\{\mathfrak{X}(t)\}_{t\ge0}$, and conditions
$\mathfrak{C}_{{\rm TV2}}(\pi)$ and $\mathfrak{Rec}$ hold. Then
for all fixed $\delta>0$,
\[
\lim_{N\to\infty}\frac{T_{\delta}^{{\rm mix}}(\bm{\sigma}_{N})}{\theta_{N}}=T_{\delta}^{{\rm mix}}(\mathfrak{X}).
\]
\end{prop}

To prove Proposition \ref{prop:main}, we need the following strict
decay of the total variation distance.
\begin{lem}
\label{lem:TV-strict-decay}Under the same hypotheses of Lemma \ref{thm:main},
for any $a>0$,
\[
\limsup_{N\to\infty}\sup_{\sigma\in\Omega_{N}}d_{{\rm TV}}\left(\bm{\sigma}_{N}\left(\theta_{N}\left(T_{\delta}^{{\rm mix}}(\mathfrak{X})+a\right);\sigma\right),\mu_{N}\right)<\delta.
\]
\end{lem}

\begin{proof}
Fix $t>0$. By the strong Markov property, for all $\sigma\in\Omega_{N}$
and $\mathcal{A}\subset\Omega_{N}$,
\[
\begin{aligned} & \mathbb{P}_{\sigma}^{N}[\bm{\sigma}_{N}(\theta_{N}t)\in\mathcal{A}]-\mu_{N}(\mathcal{A})\\
 & =\mathbb{E}_{\sigma}^{N}\left[\left\{ \mathbb{P}_{\bm{\sigma}_{N}(\mathcal{H}_{\mathcal{E}_{N}})}^{N}[\bm{\sigma}_{N}(\theta_{N}t-\mathcal{H}_{\mathcal{E}_{N}})\in\mathcal{A}]-\mu_{N}(\mathcal{A})\right\} \bm{1}_{\{\mathcal{H}_{\mathcal{E}_{N}}\le\varrho_{N}\}}\right]+R_{N}^{(1)}(\sigma,\mathcal{A}),
\end{aligned}
\]
where by condition $\mathfrak{Rec}$, 
\[
\limsup_{N\to\infty}\sup_{\sigma\in\Omega_{N}}\sup_{\mathcal{A}\subset\Omega_{N}}\left|R_{N}^{(1)}(\sigma,\mathcal{A})\right|=0.
\]
By decomposing the event $\{\mathcal{H}_{\mathcal{E}_{N}}<\varrho_{N}\}$
into $\bigcup_{k\in\mathfrak{S}}\{\mathcal{H}_{\mathcal{E}_{N}^{k}}<\varrho_{N}\}$,
the expectation in the penultimate display can be written as
\[
\begin{aligned} & \sum_{k\in\mathfrak{S}}\mathbb{E}_{\sigma}^{N}\left[\left\{ \mathbb{P}_{\bm{\sigma}_{N}\left(\mathcal{H}_{\mathcal{E}_{N}^{k}}\right)}^{N}\left[\bm{\sigma}_{N}\left(\theta_{N}t-\mathcal{H}_{\mathcal{E}_{N}^{k}}\right)\in\mathcal{A}\right]-\mu_{N}(\mathcal{A})\right\} \bm{1}_{\left\{ \mathcal{H}_{\mathcal{E}_{N}^{k}}\le\varrho_{N}\right\} }\right]\\
 & \le\sum_{k\in\mathfrak{S}}\mathbb{E}_{\sigma}^{N}\left[d_{{\rm TV}}\left(\bm{\sigma}_{N}\left(\theta_{N}t-\mathcal{H}_{\mathcal{E}_{N}^{k}};\bm{\sigma}_{N}\left(\mathcal{H}_{\mathcal{E}_{N}^{k}}\right)\right),\mu_{N}\right)\bm{1}_{\left\{ \mathcal{H}_{\mathcal{E}_{N}^{k}}\le\varrho_{N}\right\} }\right].
\end{aligned}
\]
By condition $\mathfrak{C}_{{\rm TV2}}(\pi)$ and the fact that $\varrho_{N}\ll\theta_{N}$,
the right-hand side equals
\[
\sum_{k\in\mathfrak{S}}\mathbb{P}_{\sigma}^{N}\left[\mathcal{H}_{\mathcal{E}_{N}^{k}}<\alpha_{N}\right]d_{{\rm TV}}(\mathfrak{X}(t;k),\pi)+R_{N}^{(2)}(\sigma),
\]
where 
\[
\lim_{N\to\infty}\sup_{\sigma\in\Omega_{N}}\left|R_{N}^{(2)}(\sigma)\right|=0.
\]
Since $\sum_{k\in\mathfrak{S}}\mathbb{P}_{\sigma}^{N}[\mathcal{H}_{\mathcal{E}_{N}^{k}}=\mathcal{H}_{\mathcal{E}_{N}}]=1$,
the last summation in the penultimate display is bounded above by
\[
\max_{k\in\mathfrak{S}}d_{{\rm TV}}(\mathfrak{X}(t;k),\pi).
\]

So far, we have proved that
\[
\begin{aligned}\limsup_{N\to\infty}\sup_{\sigma\in\Omega_{N}}d_{{\rm TV}}(\bm{\sigma}_{N}(\theta_{N}t;\sigma),\mu_{N})=\limsup_{N\to\infty} & \sup_{\sigma\in\Omega_{N}}\sup_{\mathcal{A}\subset\Omega_{N}}\left(\mathbb{P}_{\sigma}^{N}[\bm{\sigma}_{N}(\theta_{N}t)\in\mathcal{A}]-\mu_{N}(\mathcal{A})\right)\\
 & \le\max_{k\in\mathfrak{S}}d_{{\rm TV}}(\mathfrak{X}(t;k),\pi).
\end{aligned}
\]
Therefore, by Lemma \ref{lem:TV-3}, since $\pi$ is the unique stationary
distribution,
\begin{align*}
\limsup_{N\to\infty}\sup_{\sigma\in\Omega_{N}}d_{{\rm TV}}\left(\bm{\sigma}_{N}\left(\theta_{N}\left(T_{\delta}^{{\rm mix}}(\mathfrak{X})+a\right);\sigma\right),\mu_{N}\right) & \le\max_{k\in\mathfrak{S}}d_{{\rm TV}}\left(\mathfrak{X}\left(T_{\delta}^{{\rm mix}}(\mathfrak{X})+a;k\right),\pi\right)\\
 & <\max_{k\in\mathfrak{S}}d_{{\rm TV}}\left(\mathfrak{X}\left(T_{\delta}^{{\rm mix}}(\mathfrak{X});k\right),\pi\right)\le\delta.
\end{align*}
This proves Lemma \ref{lem:TV-strict-decay}.
\end{proof}
We are now ready to prove Proposition \ref{prop:main}.
\begin{proof}[Proof of Proposition \ref{prop:main}]
 Following the logic in \cite[Proposition 7.1]{Lee25}, we obtain
the lower bound part
\[
\liminf_{N\to\infty}\frac{T_{\delta}^{{\rm mix}}(\bm{\sigma}_{N})}{\theta_{N}}\ge T_{\delta}^{{\rm mix}}(\mathfrak{X}).
\]
In addition, following \cite[Proposition 7.3]{Lee25}, along with
Lemma \ref{lem:TV-strict-decay}, we have the upper bound part
\[
\limsup_{N\to\infty}\frac{T_{\delta}^{{\rm mix}}(\bm{\sigma}_{N})}{\theta_{N}}\le T_{\delta}^{{\rm mix}}(\mathfrak{X}).
\]
The two displayed inequalities complete the proof of Proposition \ref{prop:main}.
\end{proof}
\begin{rem}
We should remark that the strict inequality in Lemma \ref{lem:TV-strict-decay}
is crucial in the upper bound part.
\end{rem}

The strategy to prove our main theorem (Theorem \ref{thm:main}) is
to check all conditions regarding the metastability description presented
in this section, and then apply Propositions \ref{prop:LLM}, \ref{prop:TV-TV2},
and \ref{prop:main}. This procedure will be explained in Section
\ref{sec3}.

\section{\label{sec3}Description of Metastability of the CWP Model}

In this section, we present a detailed description of the metastable
behavior of the CWP model in terms of the general methodology explained
in Section \ref{sec2}.
\begin{assumption}
In the remainder of the article, we always assume that $\beta>\beta_{1}$
and that $\beta\ne q$ if $q\in\{3,4\}$.
\end{assumption}

\subsubsection*{Metastable Valleys}

We first introduce the metastable valleys. Define $\mathfrak{S}_{\beta}:=\{1,2\}$
if $q=2$ and, if $q\ge3$,
\begin{equation}
\mathfrak{S}_{\beta}:=\begin{cases}
\llbracket0,q\rrbracket & \text{if}\quad\beta\in(\beta_{1},\beta_{2}],\\
\llbracket1,q\rrbracket & \text{if}\quad\beta\in(\beta_{2},\infty).
\end{cases}\label{eq:Sbeta-def}
\end{equation}
Fix a sufficiently small number $\eta=\eta(\beta)>0$ such that there
is no critical point of $F_{\beta}$ in the domain
\begin{equation}
\{F_{\beta}({\bf v}_{1})-2\eta<F_{\beta}<F_{\beta}({\bf v}_{1})\}\cup\{F_{\beta}({\bf u}_{1,2})-2\eta<F_{\beta}<F_{\beta}({\bf u}_{1,2})\}.\label{eq:eta-def}
\end{equation}
Moreover, assume that
\begin{equation}
\eta<D_{\beta}\qquad\text{and}\qquad\eta<F_{\beta}({\bf v}_{1})-F_{\beta}({\bf e})\quad\text{if}\quad\beta<q.\label{eq:eta-small}
\end{equation}
Indeed, it is possible to take such $\eta$ since $\mathcal{C}$ is
finite by Lemma \ref{lem:critical-pts}. Then, the metastable valleys
$\mathcal{E}_{N}^{k}\subset\Omega_{N}$, $k\in\llbracket0,q\rrbracket$,
are defined as follows. First for $k\in\llbracket1,q\rrbracket$,
define (cf. \eqref{eq:PiN})
\begin{equation}
\mathcal{E}_{N}^{k}:=\Pi_{N}^{-1}(\mathcal{W}_{k}\cap\{F_{\beta}<H_{\beta}-\eta\}\cap\Xi_{N}).\label{eq:ENk-def}
\end{equation}
Then if $q\ge3$ and $\beta\in(\beta_{1},q)$, define
\begin{equation}
\mathcal{E}_{N}^{0}:=\Pi_{N}^{-1}(\mathcal{W}_{0}\cap\{F_{\beta}<F_{\beta}({\bf v}_{1})-\eta\}\cap\Xi_{N}),\label{eq:EN0-def}
\end{equation}
and for all other cases define $\mathcal{E}_{N}^{0}:=\varnothing$.
Let $\mathcal{E}_{N}:=\bigcup_{k\in\mathfrak{S}_{\beta}}\mathcal{E}_{N}^{k}$.
See Figures \ref{fig1.1} and \ref{fig1.2}.

\subsubsection*{Limit Markov Chain}

Let $\mathfrak{e}_{k}$, $k\in\llbracket1,q-1\rrbracket$, denote
the $k$-th unit vector in $\mathbb{R}^{q-1}$ and let $\mathfrak{e}_{q}:=(0,\dots,0)\in\mathbb{R}^{q-1}$.
In the terminology of Notation \ref{nota:q-to-q-1}, these are exactly
the $q$ unit vectors in $\mathbb{R}^{q}$. Define $(q-1)\times(q-1)$
matrices $\mathbb{A}^{k,\ell}$, $k,\ell\in\llbracket1,q\rrbracket$,
and $\mathbb{A}(\bm{x})$, $\bm{x}=(x_{1},\dots,x_{q})\in\Xi$, as
\begin{equation}
\mathbb{A}^{k,\ell}:=(\mathfrak{e}_{\ell}-\mathfrak{e}_{k})(\mathfrak{e}_{\ell}-\mathfrak{e}_{k})^{\dagger}\qquad\text{and }\qquad\mathbb{A}(\bm{x}):=\sum_{1\le k<\ell\le q}\sqrt{x_{k}x_{\ell}}\mathbb{A}^{k,\ell}.\label{eq:Akl-def}
\end{equation}
Since each $\mathbb{A}^{k,\ell}$ is positive definite, $\mathbb{A}$
satisfies \cite[display (A.1)]{LS18} and hence, by \cite[Lemma A.1]{LS18},
the matrices $(\nabla^{2}F_{\beta})({\bf v}_{1})\mathbb{A}({\bf v}_{1})^{\dagger}$
and $(\nabla^{2}F_{\beta})({\bf u}_{1,2})\mathbb{A}({\bf u}_{1,2})^{\dagger}$
have unique negative eigenvalues, which will be denoted respectively
by $-\mu$ and $-\mu'$.

Recall \eqref{eq:Gbeta}. Define the constants as
\begin{equation}
\begin{aligned}\omega & :=\frac{\mu}{\sqrt{-\det\left[\left(\nabla^{2}F_{\beta}\right)({\bf v}_{1})\right]}}e^{-\beta G_{\beta}({\bf v}_{1})},\\
\omega' & :=\frac{\mu'}{\sqrt{-\det\left[\left(\nabla^{2}F_{\beta}\right)({\bf u}_{1,2})\right]}}e^{-\beta G_{\beta}({\bf u}_{1,2})},
\end{aligned}
\label{eq:omega-def}
\end{equation}
and
\begin{equation}
\begin{aligned}\nu & :=\frac{1}{\beta\sqrt{\det\left[\left(\nabla^{2}F_{\beta}\right)({\bf e})\right]}}e^{-\beta G_{\beta}({\bf e})},\\
\nu' & :=\frac{1}{\beta\sqrt{\det\left[\left(\nabla^{2}F_{\beta}\right)({\bf u}_{1})\right]}}e^{-\beta G_{\beta}({\bf u}_{1})}.
\end{aligned}
\label{eq:nu-def}
\end{equation}

\begin{defn}
\label{def:limit-chain}Define $\{\mathfrak{X}_{\beta}(t)\}_{t\ge0}$
as the Markov chain in $\mathfrak{S}_{\beta}$ with jump rate $\mathfrak{r}_{\beta}(\cdot,\cdot)$
given as follows. For $q=2$,
\[
\mathfrak{r}_{\beta}(1,2)=\mathfrak{r}_{\beta}(2,1):=\frac{\beta\sqrt{F_{\beta}''({\bf e})F_{\beta}''({\bf u}_{1})}}{2}e^{-\beta(G_{\beta}({\bf e})-G_{\beta}({\bf u}_{1}))}.
\]
For $q\in\{3,4\}$,
\[
\mathfrak{r}_{\beta}(k,\ell):=\begin{cases}
\frac{\omega}{\nu'}\bm{1}_{\{\ell=0\}} & \text{if}\quad\beta\in(\beta_{1},\beta_{2}),\\
\frac{\omega}{\nu'}\bm{1}_{\{\ell=0\}}+\frac{\omega}{\nu}\bm{1}_{\{k=0\}} & \text{if}\quad\beta=\beta_{2},\\
\frac{\omega}{q\nu'} & \text{if}\quad\beta\in(\beta_{2},q),\\
\frac{\omega}{\nu'} & \text{if}\quad\beta\in(q,\infty),\quad q=3,\\
\frac{\omega'}{\nu'} & \text{if}\quad\beta\in(q,\infty),\quad q=4.
\end{cases}
\]
For $q\ge5$,
\[
\mathfrak{r}_{\beta}(k,\ell):=\begin{cases}
\frac{\omega}{\nu'}\bm{1}_{\{\ell=0\}} & \text{if}\quad\beta\in(\beta_{1},\beta_{2}),\\
\frac{\omega}{\nu'}\bm{1}_{\{\ell=0\}}+\frac{\omega}{\nu}\bm{1}_{\{k=0\}} & \text{if}\quad\beta=\beta_{2},\\
\frac{\omega}{q\nu'} & \text{if}\quad\beta\in(\beta_{2},\beta_{3}),\\
\frac{1}{\nu'}\left(\frac{\omega}{q}+\omega'\right) & \text{if}\quad\beta=\beta_{3},\\
\frac{\omega'}{\nu'} & \text{if}\quad\beta\in(\beta_{3},\infty).
\end{cases}
\]
\end{defn}

The dynamics $\mathfrak{X}_{\beta}$ always have a unique stationary
distribution $\pi_{\beta}$ on $\mathfrak{S}_{\beta}$, which is given
as $\pi_{\beta}=\frac{1}{2}(\delta_{1}+\delta_{2})$ for $q=2$ and,
for $q\ge3$,
\[
\pi_{\beta}=\begin{cases}
\delta_{0} & \text{if}\quad\beta\in(\beta_{1},\beta_{2}),\\
\frac{\nu}{\nu+\nu'q}\delta_{0}+\frac{\nu'}{\nu+\nu'q}\sum_{k=1}^{q}\delta_{k} & \text{if}\quad\beta=\beta_{2},\\
\frac{1}{q}\sum_{k=1}^{q}\delta_{k} & \text{if}\quad\beta\in(\beta_{2},\infty).
\end{cases}
\]
Denote by $\mathfrak{T}(\delta)$ the $\delta$-mixing time of $\mathfrak{X}_{\beta}$:
\begin{equation}
\mathfrak{T}(\delta)=\mathfrak{T}(\beta,q,\delta):=T_{\delta}^{{\rm mix}}(\mathfrak{X}_{\beta}).\label{eq:Tdelta}
\end{equation}
The metastable time-scale is defined as
\begin{equation}
\theta_{N}=\theta_{N}^{\beta}:=2\pi Ne^{ND_{\beta}}.\label{eq:thetaN-CWP}
\end{equation}
The following theorem describes the model reduction subject to the
CWP model. Recall that $\{\widetilde{\bm{\sigma}}_{N}^{\beta}(t)\}_{t\ge0}$
denotes the trace process in $\mathcal{E}_{N}=\bigcup_{k\in\mathfrak{S}_{\beta}}\mathcal{E}_{N}^{k}$
and $\Psi_{N}:\mathcal{E}_{N}\to\mathfrak{S}_{\beta}$ is the projection
function, i.e., $\Psi_{N}(\sigma)=k$ for all $\sigma\in\mathcal{E}_{N}^{k}$.
\begin{thm}[Model Reduction $\mathfrak{C}+\mathfrak{D}$ of the CWP Model]
\label{thm:CWP-CD}Fix $k\in\mathfrak{S}_{\beta}$.
\begin{enumerate}
\item For any sequence $(\sigma_{N})_{N}$ in $\mathcal{E}_{N}^{k}$, the
law of $\{\Psi_{N}(\widetilde{\bm{\sigma}}_{N}^{\beta}(\theta_{N}t))\}_{t\ge0}$
starting from $\sigma_{N}$ converges weakly to the law of $\{\mathfrak{X}_{\beta}(t)\}_{t\ge0}$
starting from $k$ as $N\to\infty$.
\item For any $T>0$,
\[
\lim_{N\to\infty}\sup_{\sigma\in\mathcal{E}_{N}^{k}}\mathbb{E}_{\sigma}^{N,\beta}\left[\int_{0}^{T}\bm{1}_{\left\{ \bm{\sigma}_{N}^{\beta}(\theta_{N}t)\in\Omega_{N}\setminus\mathcal{E}_{N}\right\} }\,{\rm d}t\right]=0.
\]
\end{enumerate}
\end{thm}

This result is a combination of \cite[Section 4.3]{LS16} and \cite[Theorems 4.2 and 4.5]{Lee22}.
\begin{rem}
In fact, the two conditions $\mathfrak{C}$ and $\mathfrak{D}$ were
verified in the literature \cite{LS16,Lee22} for the so-called \emph{proportions
chain} (cf. \eqref{eq:proportions-chain}), which is obtained by projecting
the original dynamics $\bm{\sigma}_{N}$ in $\Omega_{N}$ to $\Xi_{N}$
via $\Pi_{N}$ (cf. \eqref{eq:PiN}). However, the two conditions
are equivalent for these two processes since the metastable sets in
$\mathcal{E}_{N}$ are exactly the inverse images of the metastable
sets in $\Xi_{N}$ via $\Pi_{N}$; see \eqref{eq:FNk} and \eqref{eq:FN0}.
\end{rem}

\begin{rem}
Even though in \cite{LS16,Lee22} the results are proved only for
$q\ge3$, the result for the Curie--Weiss model ($q=2$) follows
directly from the ideas and computations presented therein. Actually,
the computations are much easier than the Potts case since the energy
landscape is one dimensional and everything is explicit. We chose
not to present the tedious verification of this case and just record
the results.
\end{rem}

\begin{rem}
In fact, the modes of convergence of $\{\bm{\sigma}_{N}^{\beta}(t)\}_{t\ge0}$
to the limit Markov chain $\{\mathfrak{X}_{\beta}(t)\}_{t\ge0}$ are
a slightly different in the previous two articles \cite{LS16,Lee22}.
In \cite{LS16}, the convergence is stated in terms of the so-called
\emph{soft} topology, which was introduced in \cite{Lan15} and is
adequate to describe the convergence of metastable process trajectories.
In \cite{Lee22}, the convergence is stated in the language of finite-dimensional
marginal distributions, a concept first formulated in \cite{LLM18}.
However, both modes of convergence require \emph{a priori} the trace
process convergence ($\mathfrak{C}$) and negligibility $(\mathfrak{D}$),
along with a few additional verifications. Thus, Theorem \ref{thm:CWP-CD}
is completely verified in \cite{LS16,Lee22}.
\end{rem}

Thus, the missing conditions to be checked are summarized as follows.
\begin{thm}[Local Mixing $\mathfrak{M}$ of the CWP Model]
\label{thm:CWP-M}There exists $\mathcal{B}_{N}^{k}\subset\mathcal{E}_{N}^{k}$
for each $k\in\mathfrak{S}_{\beta}$ such that:
\begin{enumerate}
\item for every $\delta>0$,
\begin{equation}
\lim_{N\to\infty}\sup_{k\in\mathfrak{S}_{\beta}}\sup_{\sigma\in\mathcal{E}_{N}^{k}}\mathbb{P}_{\sigma}^{N,\beta}\left[\mathcal{H}_{\mathcal{B}_{N}^{k}}>\delta\theta_{N}\right]=0;\label{eq:M1}
\end{equation}
\item there exists a time-scale $(\varrho_{N})_{N\ge1}$ with $\varrho_{N}\ll\theta_{N}$
such that
\begin{equation}
\lim_{N\to\infty}\sup_{k\in\mathfrak{S}_{\beta}}\sup_{\sigma\in\mathcal{B}_{N}^{k}}\mathbb{P}_{\sigma}^{N,\beta}\left[\mathcal{H}_{\Omega_{N}\setminus\mathcal{E}_{N}^{k}}\le2\varrho_{N}\right]=0;\label{eq:M2}
\end{equation}
\item and
\begin{equation}
\lim_{N\to\infty}\sup_{k\in\mathfrak{S}_{\beta}}\sup_{\sigma\in\mathcal{B}_{N}^{k}}d_{{\rm TV}}\left(\bm{\sigma}_{N}^{\mathcal{E}_{N}^{k}}(\varrho_{N};\sigma),\mu_{N}^{\mathcal{E}_{N}^{k}}\right)=0.\label{eq:M3}
\end{equation}
\end{enumerate}
\end{thm}

\begin{thm}[Recurrence Property $\mathfrak{Rec}$ of the CWP Model]
\label{thm:CWP-Rec}There exists $(\varrho_{N})_{N\ge1}$ with $\varrho_{N}\ll\theta_{N}$
such that
\[
\lim_{N\to\infty}\sup_{\sigma\in\Omega_{N}}\mathbb{P}_{\sigma}^{N,\beta}[\mathcal{H}_{\mathcal{E}_{N}}>\varrho_{N}]=0.
\]
\end{thm}

Provided that Theorems \ref{thm:CWP-M} and \ref{thm:CWP-Rec} hold
true, let us prove Theorem \ref{thm:main}. First, we have the following
relation between the invariant distributions $\mu_{N}^{\beta}$ and
$\pi_{\beta}$.
\begin{lem}
\label{lem:stat-dist-conv}We have
\begin{equation}
\lim_{N\to\infty}\mu_{N}^{\beta}\left(\mathcal{E}_{N}^{k}\right)=\pi_{\beta}(k)\qquad\text{for each}\quad k\in\mathfrak{S}_{\beta}.\label{eq:inv}
\end{equation}
Moreover, for each $k\in\mathfrak{S}_{\beta}$,
\begin{equation}
\lim_{N\to\infty}\frac{\mu_{N}^{\beta}(\Delta_{N})}{\mu_{N}^{\beta}\left(\mathcal{E}_{N}^{k}\right)}=0.\label{eq:inv-DeltaN}
\end{equation}
\end{lem}

\begin{proof}
First, we prove \eqref{eq:inv}. From \eqref{eq:piNbeta} and \eqref{eq:piNbeta-formula},
the measure $\mu_{N}^{\beta}$ concentrates on any macroscopic neighborhood
of the global minima. The case of $q=2$ is clear since ${\bf u}_{1}$
and ${\bf u}_{2}$ are the two global minima and the system is symmetric,
thus $\mu_{N}^{\beta}(\mathcal{E}_{N}^{1})=\mu_{N}^{\beta}(\mathcal{E}_{N}^{2})\simeq\frac{1}{2}$.

Now, assume that $q\ge3$. First, let $\beta\in(\beta_{1},\beta_{2})$.
Since ${\bf e}$ is the only global minimum of $F_{\beta}$, the measure
$\mu_{N}^{\beta}$ concentrates on $\mathcal{E}_{N}^{0}$, i.e.,
\[
\lim_{N\to\infty}\mu_{N}^{\beta}\left(\mathcal{E}_{N}^{0}\right)=1.
\]
Let $\beta=\beta_{2}$. By \cite[display (2.8)]{LS18},
\[
\lim_{N\to\infty}\mu_{N}^{\beta}\left(\mathcal{E}_{N}^{k}\right)=\begin{cases}
\frac{\nu}{\nu+q\nu'} & \text{if}\quad k=0,\\
\frac{\nu'}{\nu+q\nu'} & \text{if}\quad k\in\llbracket1,q\rrbracket.
\end{cases}
\]
For $\beta\in(\beta_{2},\infty)$, since ${\bf u}_{1},\dots,{\bf u}_{q}$
are the global minima, due to symmetry,
\[
\lim_{N\to\infty}\mu_{N}^{\beta}\left(\mathcal{E}_{N}^{k}\right)=\frac{1}{q}\qquad\text{for}\quad k\in\llbracket1,q\rrbracket.
\]
Finally, we consider \eqref{eq:inv-DeltaN}. According to the energy
landscape analysis in Section \ref{sec1.4} (cf. Figure \ref{fig1.2}),
with the aid of \eqref{eq:inv}, it only remains to verify that
\[
\lim_{N\to\infty}\frac{\mu_{N}^{\beta}(\Delta_{N})}{\mu_{N}^{\beta}\left(\mathcal{E}_{N}^{k}\right)}=0\qquad\text{if}\quad k\in\llbracket1,q\rrbracket,\quad q\ge3,\quad\beta\in(\beta_{1},\beta_{2}).
\]
This follows from \eqref{eq:piNbeta} and \eqref{eq:piNbeta-formula}
since each $\mathcal{E}_{N}^{k}$, $k\in\llbracket1,q\rrbracket$
is a macroscopic neighborhood of ${\bf u}_{k}$ and
\[
\Delta_{N}=\Pi_{N}^{-1}(\{F_{\beta}\ge H_{\beta}-\eta\}\cap\Xi_{N}),
\]
where $F_{\beta}({\bf u}_{k})<H_{\beta}-\eta$ by \eqref{eq:eta-small}.
This concludes the proof of Lemma \ref{lem:stat-dist-conv}.
\end{proof}
Now, we are in position to prove Theorem \ref{thm:main} and Corollary
\ref{cor:no-cutoff}.
\begin{proof}[Proof of Theorem \ref{thm:main}]
 By Theorems \ref{thm:CWP-CD} and \ref{thm:CWP-M}, the dynamics
$\bm{\sigma}_{N}^{\beta}$ satisfies the conditions $\mathfrak{C}$,
$\mathfrak{D}$, and $\mathfrak{M}$ with its limit chain $\mathfrak{X}_{\beta}$
and time-scale $\theta_{N}$. Since $\bm{\sigma}_{N}^{\beta}$ is
reversible, condition $\mathfrak{C}_{{\rm TV}}$ holds true by Proposition
\ref{prop:LLM} and \eqref{eq:inv-DeltaN}. Moreover, by Proposition
\ref{prop:TV-TV2} and \eqref{eq:inv}, condition $\mathfrak{C}_{{\rm TV2}}(\pi_{\beta})$
is satisfied. Finally, Proposition \ref{prop:main} and Theorem \ref{thm:CWP-Rec}
complete the proof.
\end{proof}
\begin{proof}[Proof of Corollary \ref{cor:no-cutoff}]
 By Theorem \ref{thm:main}, for any $\delta\in(0,1)$,
\[
\lim_{N\to\infty}\frac{T_{\delta}^{{\rm mix}}\left(\bm{\sigma}_{N}^{\beta}\right)}{T_{1-\delta}^{{\rm mix}}\left(\bm{\sigma}_{N}^{\beta}\right)}=\frac{\mathfrak{T}(\delta)}{\mathfrak{T}(1-\delta)}=\frac{T_{\delta}^{{\rm mix}}(\mathfrak{X}_{\beta})}{T_{1-\delta}^{{\rm mix}}(\mathfrak{X}_{\beta})}.
\]
Note that $\mathfrak{X}_{\beta}$ is a Markov chain in a finite set
$\mathfrak{S}_{\beta}$ with $|\mathfrak{S}_{\beta}|\ge2$ with a
unique stationary distribution $\pi_{\beta}$. Since $t\mapsto d_{{\rm TV}}(\mathfrak{X}_{\beta}(t;k),\pi_{\beta})$
is continuous for any $k\in\mathfrak{S}_{\beta}$, we have
\[
T_{\delta}^{{\rm mix}}(\mathfrak{X}_{\beta})<T_{\delta'}^{{\rm mix}}(\mathfrak{X}_{\beta})\qquad\text{if}\quad\delta>\delta'.
\]
The two displayed relations imply that \eqref{eq:cutoff-def} cannot
hold, which proves the corollary.
\end{proof}
The remainder of the article is devoted to proving Theorems \ref{thm:CWP-M}
and \ref{thm:CWP-Rec}.

\section{\label{sec4}Proof of Theorem \ref{thm:CWP-Rec}}

In this section, we prove Theorem \ref{thm:CWP-Rec}. For $A\subset\llbracket0,q\rrbracket$,
define
\[
\mathcal{E}_{N}^{A}:=\bigcup_{k\in A}\mathcal{E}_{N}^{k}.
\]
The proof relies on the following two lemmas, whose proofs are postponed
to Sections \ref{sec4.1} and \ref{sec4.3}, respectively.
\begin{lem}
\label{lem:rec1}There exists a sequence $(\varrho_{N})_{N\ge1}$
with $\varrho_{N}\ll\theta_{N}$ such that
\[
\lim_{N\to\infty}\sup_{\sigma\in\Omega_{N}}\mathbb{P}_{\sigma}^{N,\beta}\left[\mathcal{H}_{\mathcal{E}_{N}^{\llbracket0,q\rrbracket}}>\varrho_{N}\right]=0.
\]
\end{lem}

When $q\ge3$ and $\beta\in(\beta_{2},q)$, since $0\notin\mathfrak{S}_{\beta}$
but $\mathcal{E}_{N}^{0}\ne\varnothing$, we should deal with the
additional case when $\bm{\sigma}_{N}^{\beta}$ starts from the valley
$\mathcal{E}_{N}^{0}$.
\begin{lem}
\label{lem:rec2}Let $q\ge3$ and $\beta\in(\beta_{2},q)$. There
exists a sequence $(\varrho_{N})_{N\ge1}$ with $\varrho_{N}\ll\theta_{N}$
such that
\[
\lim_{N\to\infty}\sup_{\sigma\in\mathcal{E}_{N}^{0}}\mathbb{P}_{\sigma}^{N,\beta}[\mathcal{H}_{\mathcal{E}_{N}}>\varrho_{N}]=0.
\]
\end{lem}

First, we prove Theorem \ref{thm:CWP-Rec} with the aid of these two
lemmas.
\begin{proof}[Proof of Theorem \ref{thm:CWP-Rec}]
 Since $\mathfrak{S}_{\beta}=\llbracket0,q\rrbracket$ when $q\ge3$
and $\beta\in(\beta_{1},\beta_{2}]$, Theorem \ref{thm:CWP-Rec} follows
directly from Lemma \ref{lem:rec1}. In addition, since $\mathcal{E}_{N}^{0}=\varnothing$
when $q=2$ or $\beta\ge q$, it remains to prove the theorem when
$q\ge3$ and $\beta\in(\beta_{2},q)$. Let $(\varrho_{N}^{(1)})_{N\ge1}$
and $(\varrho_{N}^{(2)})_{N\ge1}$ be the sequences given by Lemmas
\ref{lem:rec1} and \ref{lem:rec2}, respectively, and let $\varrho_{N}^{(3)}:=\varrho_{N}^{(1)}\vee\varrho_{N}^{(2)}$.
Define $\varrho_{N}:=\sqrt{\varrho_{N}^{(3)}\theta_{N}}$. Note that
$\varrho_{N}^{(1)},\varrho_{N}^{(2)}\ll\varrho_{N}\ll\theta_{N}$.
For any $\sigma\in\Omega_{N}$, by Lemma \ref{lem:rec1},
\begin{equation}
\begin{aligned}\mathbb{P}_{\sigma}^{N,\beta}\left[\mathcal{H}_{\mathcal{E}_{N}}>\varrho_{N}\right]=\mathbb{P}_{\sigma}^{N,\beta} & \left[\mathcal{H}_{\mathcal{E}_{N}}>\varrho_{N},\enspace\mathcal{H}_{\mathcal{E}_{N}^{\llbracket0,q\rrbracket}}=\mathcal{H}_{\mathcal{E}_{N}}\le\varrho_{N}^{(1)}\right]\\
 & +\mathbb{P}_{\sigma}^{N,\beta}\left[\mathcal{H}_{\mathcal{E}_{N}}>\varrho_{N},\enspace\mathcal{H}_{\mathcal{E}_{N}^{\llbracket0,q\rrbracket}}=\mathcal{H}_{\mathcal{E}_{N}^{0}}\le\varrho_{N}^{(1)}\right]+R_{N}(\sigma),
\end{aligned}
\label{eq:CWP-Rec-pf}
\end{equation}
for some $R_{N}(\sigma)$ such that
\[
\limsup_{N\to\infty}\sup_{\sigma\in\Omega_{N}}|R_{N}(\sigma)|=0.
\]
Since $\varrho_{N}\gg\varrho_{N}^{(1)}$, the first probability in
the right-hand side of \eqref{eq:CWP-Rec-pf} is zero for large $N$.
By the strong Markov property at time $\mathcal{H}_{\mathcal{E}_{N}^{0}}$
and Lemma \ref{lem:rec2}, since $\varrho_{N}-\varrho_{N}^{(1)}\gg\varrho_{N}^{(2)}$,
the second probability vanishes uniformly over $\sigma\in\Omega_{N}$
as $N\to\infty$, concluding the proof.
\end{proof}

\subsection{\label{sec4.1}Proof of Lemma \ref{lem:rec1}}

In this subsection, we prove Lemma \ref{lem:rec1}. The main tool
is potential theory applied to metastable systems, a long-term project
initiated in the early 21st century that has resulted in numerous
breakthrough results in the metastability community. See the monograph
\cite{BdH15} for an extensive overview and literature. However, we
cannot apply the theory directly to the dynamics $\bm{\sigma}_{N}^{\beta}$
since the CWP configuration space $\Omega_{N}$ is exponentially big
in $N$, so that certain crude stationary estimates regarding the
stationary profile $\mu_{N}^{\beta}$ breaks down. The idea to overcome
this drawback is to consider instead the \emph{proportions chain},
which is the projected process of the original $\bm{\sigma}_{N}^{\beta}$
via the projection $\Pi_{N}:\Omega_{N}\to\Xi_{N}$ (cf. \eqref{eq:PiN}).
For all time $t\ge0$, define
\begin{equation}
\bm{S}_{N}^{\beta}(t):=\Pi_{N}\left(\bm{\sigma}_{N}^{\beta}(t)\right).\label{eq:proportions-chain}
\end{equation}
Since the CWP model has no geometry, the projected dynamics is again
an irreducible, reversible Markov chain in $\Xi_{N}$ whose unique
stationary state is $\pi_{N}^{\beta}$ (cf. \eqref{eq:piNbeta}).
The proof of this fact is elementary and we refer the readers to \cite[Proposition 2.1]{Lee22}.
Starting from each $\bm{x}=(x_{1},\dots,x_{q})\in\Xi_{N}$, the transition
rates of $\bm{S}_{N}^{\beta}$ are given as
\begin{equation}
\begin{aligned}r_{N}\left(\bm{x},\bm{x}-\frac{\mathfrak{e}_{k}}{N}+\frac{\mathfrak{e}_{\ell}}{N}\right) & =\frac{Nx_{k}}{N}\exp\left\{ -\frac{N\beta}{2}\left[H\left(\bm{x}-\frac{\mathfrak{e}_{k}}{N}+\frac{\mathfrak{e}_{\ell}}{N}\right)-H(\bm{x})\right]\right\} \\
 & =x_{k}\exp\left\{ -\frac{\beta}{2}\left(x_{k}-x_{\ell}-\frac{1}{N}\right)\right\} \qquad\text{for}\quad k,\ell\in\llbracket1,q\rrbracket,
\end{aligned}
\label{eq:rN-value}
\end{equation}
where $\mathfrak{e}_{k}$, $k\in\llbracket1,q\rrbracket$, denotes
the $k$-th unit vector in $\mathbb{R}^{q}$. Denote by ${\rm P}_{\bm{x}}^{N,\beta}$
and ${\rm E}_{\bm{x}}^{N,\beta}$ the law and the corresponding expectation,
respectively, of the proportions chain $\bm{S}_{N}^{\beta}$ starting
from $\bm{x}\in\Xi_{N}$.

Now recall \eqref{eq:ENk-def} and \eqref{eq:EN0-def}, and define
$\mathcal{F}_{N}^{k}:=\Pi_{N}(\mathcal{E}_{N}^{k})$ for each $k\in\llbracket0,q\rrbracket$.
Clearly, for $k\in\llbracket1,q\rrbracket$,
\begin{equation}
\mathcal{F}_{N}^{k}=\mathcal{W}_{k}\cap\{F_{\beta}<H_{\beta}-\eta\}\cap\Xi_{N}.\label{eq:FNk}
\end{equation}
Moreover, if $q\ge3$ and $\beta\in(\beta_{1},q)$ then
\begin{equation}
\mathcal{F}_{N}^{0}=\mathcal{W}_{0}\cap\{F_{\beta}<F_{\beta}({\bf v}_{1})-\eta\}\cap\Xi_{N},\label{eq:FN0}
\end{equation}
and otherwise $\mathcal{F}_{N}^{0}=\varnothing$. Write
\[
\mathcal{F}_{N}^{A}:=\bigcup_{x\in A}\mathcal{F}_{N}^{x}\qquad\text{for}\quad A\subset\llbracket0,q\rrbracket.
\]
Then, Lemma \ref{lem:rec1} is equivalent to the following statement:
\begin{lem}
\label{lem:rec1-1}There exists a sequence $(\varrho_{N})_{N\ge1}$
with $\varrho_{N}\ll\theta_{N}$ such that
\[
\lim_{N\to\infty}\sup_{\bm{x}\in\Xi_{N}}{\rm P}_{\bm{x}}^{N,\beta}\left[\mathcal{H}_{\mathcal{F}_{N}^{\llbracket0,q\rrbracket}}>\varrho_{N}\right]=0.
\]
\end{lem}

In the rest of Section \ref{sec4.1}, we prove Lemma \ref{lem:rec1-1}.

Fix $\bm{x}\in\Xi_{N}\setminus\mathcal{F}_{N}^{\llbracket0,q\rrbracket}$.
By the Markov inequality,
\begin{equation}
{\rm P}_{\bm{x}}^{N,\beta}\left[\mathcal{H}_{\mathcal{F}_{N}^{\llbracket0,q\rrbracket}}>\varrho_{N}\right]\le\frac{1}{\varrho_{N}}{\rm E}_{\bm{x}}^{N,\beta}\left[\mathcal{H}_{\mathcal{F}_{N}^{\llbracket0,q\rrbracket}}\right].\label{eq:pf1}
\end{equation}
According to \eqref{eq:magic-formula},
\begin{equation}
{\rm E}_{\bm{x}}^{N,\beta}\left[\mathcal{H}_{\mathcal{F}_{N}^{\llbracket0,q\rrbracket}}\right]=\frac{\sum_{\bm{y}\in\Xi_{N}}\pi_{N}^{\beta}(\bm{y})\mathfrak{h}_{\bm{x},\mathcal{F}_{N}^{\llbracket0,q\rrbracket}}(\bm{y})}{{\rm cap}_{N}\left(\bm{x},\mathcal{F}_{N}^{\llbracket0,q\rrbracket}\right)}.\label{eq:pf2}
\end{equation}
For the denominator of \eqref{eq:pf2}, we use the following lemma
which will be proved in Section \ref{sec4.2}. Given a path $\phi:\llbracket0,n\rrbracket\to\Xi_{N}$
(i.e. $r_{N}(\phi(i),\phi(i+1))>0$ for all $i\in\llbracket0,n-1\rrbracket$),
we call $\max_{i\in\llbracket0,n\rrbracket}F_{\beta}(\phi(i))$ the
\emph{energy} of $\phi$.
\begin{lem}
\label{lem:almost-desc-path}For any given $\delta>0$, there exists
a path from each $\bm{x}\in\Xi_{N}\setminus\mathcal{F}_{N}^{\llbracket0,q\rrbracket}$
to $\mathcal{F}_{N}^{\llbracket0,q\rrbracket}$ with energy at most
$F_{\beta}(\bm{x})+\delta$ for all sufficiently large $N$.
\end{lem}

We also need the following elementary (very) rough bound for $\pi_{N}^{\beta}$:
\begin{lem}
\label{lem:piNbeta-rough}For any $\bm{z},\bm{w}\in\Xi_{N}$, we have
$\pi_{N}^{\beta}(\bm{z})\le c_{1}N^{q}e^{\beta N(F_{\beta}(\bm{w})-F_{\beta}(\bm{z}))}\pi_{N}^{\beta}(\bm{w})$
where $c_{1}>0$ does not depend on $N$.
\end{lem}

\begin{proof}
Recalling \eqref{eq:piNbeta-formula},
\[
\pi_{N}^{\beta}(\bm{z})=\frac{e^{-\beta NH(\bm{z})}}{Z_{N}^{\beta}}\frac{N!}{(Nz_{1})!\cdots(Nz_{q})!}.
\]
Applying a rough Stirling-type estimate $1\le\frac{n!e^{n}}{n^{n}\sqrt{2\pi n}}\le2$
to $n=N$ and $n=Nz_{k}$ with $z_{k}\ge\frac{1}{N}$,
\[
\pi_{N}^{\beta}(\bm{z})\le\frac{e^{-\beta NH(\bm{z})}}{Z_{N}^{\beta}}\frac{2}{\prod_{k\in\llbracket1,q\rrbracket:\,z_{k}\ge\frac{1}{N}}z_{k}^{Nz_{k}+\frac{1}{2}}}\le\frac{2N^{\frac{q}{2}}e^{-\beta NF_{\beta}(\bm{z})}}{Z_{N}^{\beta}}.
\]
Similarly,
\[
\pi_{N}^{\beta}(\bm{w})\ge\frac{e^{-\beta NH(\bm{w})}}{Z_{N}^{\beta}}\frac{1}{(2\pi N)^{\frac{q-1}{2}}}\frac{1}{\prod_{k\in\llbracket1,q\rrbracket:\,w_{k}\ge\frac{1}{N}}\left(2w_{k}^{Nw_{k}+\frac{1}{2}}\right)}\ge\frac{e^{-\beta NF_{\beta}(\bm{w})}}{(2\pi N)^{\frac{q-1}{2}}2^{q}Z_{N}^{\beta}}.
\]
Combining the two displayed inequalities gives Lemma \ref{lem:piNbeta-rough}.
\end{proof}
Now applying Lemma \ref{lem:almost-desc-path} to any small $\epsilon>0$,
for large enough $N$, there exists a path $\varphi:\llbracket0,n\rrbracket\to\Xi_{N}$
from $\bm{x}$ to $\mathcal{F}_{N}^{\llbracket0,q\rrbracket}$ whose
energy is at most $F_{\beta}(\bm{x})+\epsilon$. By \eqref{eq:TP}
and the fact that $r_{N}(\bm{z},\bm{w})\ge N^{-1}e^{-\beta}$ for
any $\bm{z},\bm{w}\in\Xi_{N}$ such that $r_{N}(\bm{z},\bm{w})>0$,
\[
{\rm cap}_{N}\left(\bm{x},\mathcal{F}_{N}^{\llbracket0,q\rrbracket}\right)\ge\left(\sum_{i=0}^{n-1}\frac{1}{\pi_{N}^{\beta}(\varphi(i))r_{N}(\varphi(i),\varphi(i+1))}\right)^{-1}\ge\left(\sum_{i=0}^{n-1}\frac{Ne^{\beta}}{\pi_{N}^{\beta}(\varphi(i))}\right)^{-1}.
\]
By its definition, $F_{\beta}(\varphi(i))\le F_{\beta}(\bm{x})+\epsilon$
for all $i\in\llbracket0,n\rrbracket$, thus by Lemma \ref{lem:piNbeta-rough},
$\pi_{N}^{\beta}(\bm{x})\le c_{1}N^{q}e^{\beta N\epsilon}\pi_{N}^{\beta}(\varphi(i))$.
In addition, we may take $n<|\Xi_{N}|={N+q-1 \choose q}\le N^{q}$.
Thus,
\begin{equation}
{\rm cap}_{N}\left(\bm{x},\mathcal{F}_{N}^{\llbracket0,q\rrbracket}\right)\ge\left(N^{q}\frac{c_{1}N^{q}e^{\beta N\epsilon}\times Ne^{\beta}}{\pi_{N}^{\beta}(\bm{x})}\right)^{-1}=\frac{\pi_{N}^{\beta}(\bm{x})}{c_{1}e^{\beta}N^{2q+1}e^{\beta N\epsilon}}.\label{eq:pf3}
\end{equation}

For the numerator of \eqref{eq:pf2}, the estimate \eqref{eq:renewal}
gives, for $\bm{y}\in\Xi_{N}\setminus(\mathcal{F}_{N}^{\llbracket0,q\rrbracket}\cup\{\bm{x}\})$,
\begin{equation}
\mathfrak{h}_{\bm{x},\mathcal{F}_{N}^{\llbracket0,q\rrbracket}}(\bm{y})\le\frac{{\rm cap}_{N}(\bm{y},\bm{x})}{{\rm cap}_{N}\left(\bm{y},\mathcal{F}_{N}^{\llbracket0,q\rrbracket}\right)}.\label{eq:pf4}
\end{equation}
The denominator in the right-hand side of \eqref{eq:pf4} can be calculated
using \eqref{eq:pf3}. For the numerator we need:
\begin{lem}
\label{lem:cap-rough}For any $\bm{x}\in\Xi_{N}\setminus\mathcal{F}_{N}^{\llbracket0,q\rrbracket}$
and $\bm{y}\in\Xi_{N}\setminus(\mathcal{F}_{N}^{\llbracket0,q\rrbracket}\cup\{\bm{x}\})$,
\[
{\rm cap}_{N}(\bm{y},\bm{x})\le c_{2}\pi_{N}^{\beta}(\bm{x}),
\]
where $c_{2}>0$ does not depend on $N$.
\end{lem}

\begin{proof}
We apply \eqref{eq:DP} to the test function $f=1-\bm{1}_{\{\bm{x}\}}$.
Then,
\[
{\rm cap}_{N}(\bm{y},\bm{x})\le\mathscr{D}_{N}(1-\bm{1}_{\{\bm{x}\}})=\pi_{N}^{\beta}(\bm{x})\sum_{\bm{z}\in\Xi_{N}}r_{N}(\bm{x},\bm{z}).
\]
The number of $\bm{z}\in\Xi_{N}$ with $r_{N}(\bm{x},\bm{z})>0$ is
at most $q(q-1)$, and by \eqref{eq:rN-value}, $r_{N}(\bm{x},\bm{z})\le e^{\beta}$.
These observations prove the lemma.
\end{proof}
\begin{proof}[Proof of Lemma \ref{lem:rec1-1}]
 Fix $\epsilon>0$. Combining \eqref{eq:pf1} and \eqref{eq:pf2},
\begin{equation}
{\rm P}_{\bm{x}}^{N,\beta}\left[\mathcal{H}_{\mathcal{F}_{N}^{\llbracket0,q\rrbracket}}>\varrho_{N}\right]\le\frac{1}{\varrho_{N}}\frac{\sum_{\bm{y}\in\Xi_{N}}\pi_{N}^{\beta}(\bm{y})\mathfrak{h}_{\bm{x},\mathcal{F}_{N}^{\llbracket0,q\rrbracket}}(\bm{y})}{{\rm cap}_{N}\left(\bm{x},\mathcal{F}_{N}^{\llbracket0,q\rrbracket}\right)}.\label{eq:rec1-pf-1}
\end{equation}
The summation in the numerator may be restricted to $\bm{y}\in\Xi_{N}\setminus\mathcal{F}_{N}^{\llbracket0,q\rrbracket}$
since otherwise the equilibrium potential vanishes. For such $\bm{y}$,
by \eqref{eq:pf4}, Lemma \ref{lem:cap-rough}, and \eqref{eq:pf3},
\begin{equation}
\pi_{N}^{\beta}(\bm{y})\mathfrak{h}_{\bm{x},\mathcal{F}_{N}^{\llbracket0,q\rrbracket}}(\bm{y})\le\pi_{N}^{\beta}(\bm{y})\frac{c_{2}\pi_{N}^{\beta}(\bm{x})}{\pi_{N}^{\beta}(\bm{y})/(c_{1}e^{\beta}N^{2q+1}e^{\beta N\epsilon})}=c_{1}c_{2}e^{\beta}N^{2q+1}e^{\beta N\epsilon}\pi_{N}^{\beta}(\bm{x}).\label{eq:rec1-pf-2}
\end{equation}
Substituting \eqref{eq:rec1-pf-2} to \eqref{eq:rec1-pf-1}, via \eqref{eq:pf3}
applied on the denominator, gives
\[
{\rm P}_{\bm{x}}^{N,\beta}\left[\mathcal{H}_{\mathcal{F}_{N}^{\llbracket0,q\rrbracket}}>\varrho_{N}\right]\le\frac{1}{\varrho_{N}}\frac{|\Xi_{N}|\times c_{1}c_{2}e^{\beta}N^{2q+1}e^{\beta N\epsilon}\pi_{N}^{\beta}(\bm{x})}{\pi_{N}^{\beta}(\bm{x})/(c_{1}e^{\beta}N^{2q+1}e^{\beta N\epsilon})}\le\varrho_{N}^{-1}c_{1}^{2}c_{2}e^{2\beta}N^{5q+2}e^{2\beta N\epsilon}.
\]
Therefore, taking sufficiently small $\epsilon>0$ such that $2\beta\epsilon<D_{\beta}$,
we may choose $1\ll\varrho_{N}\ll\theta_{N}$ (cf. \eqref{eq:thetaN-CWP})
such that
\[
\sup_{\bm{x}\in\Xi_{N}}{\rm P}_{\bm{x}}^{N,\beta}\left[\mathcal{H}_{\mathcal{F}_{N}^{\llbracket0,q\rrbracket}}>\varrho_{N}\right]\ll1.
\]
This proves Lemma \ref{lem:rec1-1}.
\end{proof}
In the following subsection, we verify Lemma \ref{lem:almost-desc-path}
to complete the logic.

\subsection{\label{sec4.2}Proof of Lemma \ref{lem:almost-desc-path}}

First assume that $q=2$, the Curie--Weiss model. We say that a path
$\varphi$ \emph{descends} if $F_{\beta}(\varphi(i))>F_{\beta}(\varphi(i+1))$
for all $i$.
\begin{proof}[Proof of Lemma \ref{lem:almost-desc-path}: $q=2$]
 According to the graph shape of $F_{\beta}$ (See Figure \ref{fig1.1}),
as $\beta>\beta_{1}=2$, from any point $\bm{x}\in\Xi_{N}\setminus\mathcal{F}_{N}^{\llbracket1,2\rrbracket}$
there exists a trajectory in $\Xi$ to either ${\bf u}_{1}$ or ${\bf u}_{2}$
along which $F_{\beta}$ decreases. Following this trajectory from
$\bm{x}$ by a discrete path in $\Xi_{N}$, we arrive at a point which
is at most $\frac{1}{N}$ $L^{\infty}$-distance away from either
${\bf u}_{1}$ or ${\bf u}_{2}$, thus clearly an element of $\mathcal{F}_{N}^{\llbracket1,2\rrbracket}$
for large $N$. This is a descending path from $\bm{x}$ to $\mathcal{F}_{N}^{\llbracket1,2\rrbracket}$,
thus its energy is exactly $F_{\beta}(\bm{x})$, which concludes the
proof.
\end{proof}
In the remainder of Section \ref{sec4.2}, assume that $q\ge3$ and
fix $\delta>0$. The idea is to prove that from any point $\bm{x}\in\Xi_{N}\setminus\mathcal{F}_{N}^{\llbracket0,q\rrbracket}$,
there exists a path to another point $\bm{y}$ with $F_{\beta}(\bm{y})<F_{\beta}(\bm{x})$
whose energy is at most $F_{\beta}(\bm{x})+\delta$. Since $\Xi_{N}$
is finite, we may concatenate such paths to finally arrive at $\mathcal{F}_{N}^{\llbracket0,q\rrbracket}$.

We start with two useful lemmas. Define $\phi:(0,1]\to\mathbb{R}$
as
\begin{equation}
\phi(t):=t-\frac{1}{\beta}\log t.\label{eq:phi-def}
\end{equation}
Note that $\phi$ decreases on $(0,\frac{1}{\beta}]$, increases on
$[\frac{1}{\beta},1]$, $\phi(1)=1$, and $\phi(0+)=\infty$.
\begin{lem}
\label{lem:jump-bound-1}Given $\bm{x}\in\Xi_{N}$ and two distinct
indices $k,\ell\in\llbracket1,q\rrbracket$ such that $x_{k}>0$,
\[
F_{\beta}(\bm{x})-F_{\beta}\left(\bm{x}-\frac{\mathfrak{e}_{k}}{N}+\frac{\mathfrak{e}_{\ell}}{N}\right)\ge\frac{\phi\left(x_{\ell}+\frac{1}{N}\right)-\phi\left(x_{k}-\frac{1}{N}\right)-\frac{1}{N}}{N}.
\]
\end{lem}

\begin{proof}
Let us write $a:=x_{k}$ and $b:=x_{\ell}$. By \eqref{eq:FbetaN-def},
\begin{equation}
F_{\beta}(\bm{x})-F_{\beta}\left(\bm{x}-\frac{\mathfrak{e}_{k}}{N}+\frac{\mathfrak{e}_{\ell}}{N}\right)=\frac{b-a+\frac{1}{N}}{N}+\frac{1}{\beta}\left(a\log\frac{a}{a-\frac{1}{N}}+b\log\frac{b}{b+\frac{1}{N}}+\frac{1}{N}\log\frac{a-\frac{1}{N}}{b+\frac{1}{N}}\right).\label{eq:diff-calc}
\end{equation}
Simplifying via \eqref{eq:phi-def}, the right-hand side equals
\[
-\frac{1}{N^{2}}+\frac{\phi\left(b+\frac{1}{N}\right)-\phi\left(a-\frac{1}{N}\right)}{N}+\frac{1}{\beta}\left(a\log\frac{1}{1-\frac{1}{Na}}+b\log\frac{1}{1+\frac{1}{Nb}}\right).
\]
Using $\log(1+t)\le t$, this is bounded from below by
\[
-\frac{1}{N^{2}}+\frac{\phi\left(b+\frac{1}{N}\right)-\phi\left(a-\frac{1}{N}\right)}{N}+\frac{1}{\beta}\left(\frac{a}{Na}-\frac{b}{Nb}\right)=-\frac{1}{N^{2}}+\frac{\phi\left(b+\frac{1}{N}\right)-\phi\left(a-\frac{1}{N}\right)}{N}.
\]
\end{proof}
\begin{center}
\begin{figure}
\begin{centering}
\begin{tikzpicture}
\begin{scope}[scale=0.8]
\draw[very thick] (-4,-4) rectangle (4,4); \draw (4.1,0) node[right]{$\Xi_N$};
\draw (-3.2,-3.2) rectangle (3.2,3.2); \draw (3.6,2.6) node{$\mathcal{O}_N^\epsilon$}; \draw (2.8,2.6) node{$\mathcal{I}_N^\epsilon$};
\draw (1,2.6) node{$\mathcal{M}_N^{\epsilon,\alpha}$};
\end{scope}

\begin{scope}[scale=1.2]
\fill[black!40!white] (0,0) circle (1); \draw (0,0) circle (1); \draw[thick] (-0.06,-0.06)--(0.06,0.06); \draw[thick] (0.06,-0.06)--(-0.06,0.06);
\foreach \i in {36,108,180,252,324} {
\fill[black!15!white] ({1.4*cos(\i)},{1.4*sin(\i)}) circle (0.4); \draw ({1.4*cos(\i)},{1.4*sin(\i)}) circle (0.4); \fill ({1.4*cos(\i)},{1.4*sin(\i)}) circle (0.05);
\draw[thick] ({0.925*cos(\i)},{0.925*sin(\i)})--({1.075*cos(\i)},{1.075*sin(\i)});
}
\draw ({1.4*cos(36)},{1.4*sin(36)}) node[above]{${\bf u}_k$};
\draw ({1*cos(36)},{1*sin(36)}) node[left]{${\bf v}_k$};
\draw (0,0) node[below]{$\bf e$};
\draw[teal,very thick,densely dotted,->] ({0.1*cos(36)},{0.1*sin(36)})--({1.3*cos(36)},{1.3*sin(36)});
\draw[thick] ({0.925*cos(36)},{0.925*sin(36)})--({1.075*cos(36)},{1.075*sin(36)});
\end{scope}

\fill(0.8,-2) circle (0.06);
\draw (0.8,-2) node[below]{$\bm{c}$};
\draw[teal,very thick,densely dotted,->] ({0.8+0.1*cos(-15)},{-2+0.1*sin(-15)})--({0.8+0.4*cos(-15)},{-2+0.4*sin(-15)});

\begin{scope}[shift={(8,0)},scale=0.8]
\draw[very thick] (-4,-4) rectangle (4,4);
\draw (-3.2,-3.2) rectangle (3.2,3.2);
\end{scope}

\begin{scope}[shift={(8,0)},scale=1.2]
\draw[thick] (-0.06,-0.06)--(0.06,0.06); \draw[thick] (0.06,-0.06)--(-0.06,0.06);
\foreach \i in {36,108,180,252,324} {
\fill[black!40!white] ({0.7/cos(54)*cos(\i)},{0.7/cos(54)*sin(\i)}) circle (0.7); \draw ({0.7/cos(54)*cos(\i)},{0.7/cos(54)*sin(\i)}) circle (0.7); \fill ({0.7/cos(54)*cos(\i)},{0.7/cos(54)*sin(\i)}) circle (0.05);
}
\foreach \i in {0,72,144,216,288} {
\draw[thick] ({0.7/tan(36)*cos(\i)+0.075*cos(\i+90)},{0.7/tan(36)*sin(\i)+0.075*sin(\i+90)})--({0.7/tan(36)*cos(\i)-0.075*cos(\i+90)},{0.7/tan(36)*sin(\i)-0.075*sin(\i+90)});
}
\draw ({0.7/sin(36)*cos(36)},{0.7/sin(36)*sin(36)}) node[above]{${\bf u}_k$};
\fill ({-0.15/sin(36)*cos(36)},{-0.15/sin(36)*sin(36)}) circle (0.05);
\draw ({-0.15/sin(36)*cos(36)+0.1},{-0.15/sin(36)*sin(36)}) node[above left]{${\bf v}_k$};
\draw ({0.7/tan(36)*cos(72)},{0.7/tan(36)*sin(72)}) node[above]{${\bf u}_{k,\ell}$};
\draw (0,0) node[below right]{$\bf e$};
\draw[teal,very thick,densely dotted,->] ({0.1*cos(36)},{0.1*sin(36)})--({(0.7/sin(36)-0.1)*cos(36)},{(0.7/sin(36)-0.1)*sin(36)});
\draw[teal,very thick,densely dotted,->] ({(-0.15/sin(36)-0.1)*cos(36)},{(-0.15/sin(36)-0.1)*sin(36)})--({(-0.15/sin(36)-0.3)*cos(36)},{(-0.15/sin(36)-0.3)*sin(36)});
\draw[teal,very thick,densely dotted,->] ({0.7/tan(36)*cos(72)+0.1*cos(-5)},{0.7/tan(36)*sin(72)+0.1*sin(-5)})--({0.7/tan(36)*cos(72)+0.4*cos(-5)},{0.7/tan(36)*sin(72)+0.4*sin(-5)});
\end{scope}

\begin{scope}[shift={(8,0)}]
\fill(0.8,-2) circle (0.06);
\draw (0.8,-2) node[below]{$\bm{c}$};
\draw[teal,very thick,densely dotted,->] ({0.8+0.1*cos(-15)},{-2+0.1*sin(-15)})--({0.8+0.4*cos(-15)},{-2+0.4*sin(-15)});
\end{scope}
\end{tikzpicture}
\par\end{centering}
\caption{\label{fig4.1}The cases of $\beta\in(\beta_{1},\beta_{2})$ (left)
and $\beta\in(q,\infty)$ (right) for $q\ge5$. From each critical
point ${\bf e}$, ${\bf v}_{k}$, ${\bf u}_{k,\ell}$ (if it exists),
or $\bm{c}\in\mathcal{C}_{4}$, the teal arrow represents the descending
trajectory constructed in Lemmas \ref{l3}, \ref{l4}, or \ref{l5}.}
\end{figure}
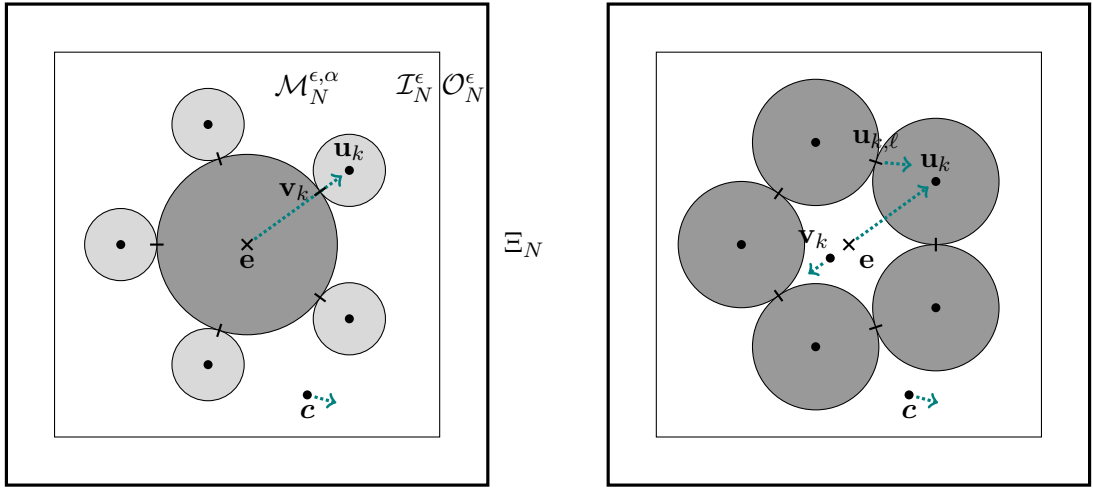
\par\end{center}

Now, we begin the decomposition of $\Xi$ and $\Xi_{N}$. See Figure
\ref{fig4.1} for an illustration.

Below, the constants $\epsilon,\alpha>0$ are chosen to be sufficiently
small in the sequel. Recall \eqref{eq:Xi-def} and define
\[
\mathcal{O}^{\epsilon}:=\{\bm{x}\in\Xi:x_{k}<\epsilon\enspace\text{for some}\enspace k\},\qquad\mathcal{I}^{\epsilon}:=\Xi\setminus\mathcal{O}^{\epsilon}=\{\bm{x}\in\Xi:x_{1},\dots,x_{q}\ge\epsilon\}.
\]
Then, define $\mathcal{O}_{N}^{\epsilon}:=\mathcal{O}^{\epsilon}\cap\Xi_{N}$
and $\mathcal{I}_{N}^{\epsilon}:=\mathcal{I}^{\epsilon}\cap\Xi_{N}$.
We start our analysis from $\mathcal{O}_{N}^{\epsilon}$.
\begin{lem}
\label{l1}For small $\epsilon>0$, the following statement holds
for all sufficiently large $N$. For every $\bm{x}\in\mathcal{O}_{N}^{\epsilon}$,
there exists $\bm{y}\in\Xi_{N}$ such that $r_{N}(\bm{x},\bm{y})>0$
and $F_{\beta}(\bm{x})>F_{\beta}(\bm{y})$.
\end{lem}

\begin{proof}
Let $\epsilon\in(0,\frac{1}{q})$ and fix $\bm{x}\in\mathcal{O}_{N}^{\epsilon}$.
Then, there exist $k,\ell\in\llbracket1,q\rrbracket$ such that $x_{k}>\frac{1}{q}$
(via the pigeonhole principle) and $x_{\ell}<\epsilon$. By Lemma
\ref{lem:jump-bound-1},
\[
F_{\beta}(\bm{x})-F_{\beta}\left(\bm{x}-\frac{\mathfrak{e}_{k}}{N}+\frac{\mathfrak{e}_{\ell}}{N}\right)\ge\frac{\phi\left(x_{\ell}+\frac{1}{N}\right)-\phi\left(x_{k}-\frac{1}{N}\right)-\frac{1}{N}}{N}.
\]
Let $\epsilon\in(0,\frac{1}{\beta}\wedge\frac{1}{q})$ such that $\phi(\epsilon)>\phi(\frac{1}{q})\vee\phi(1)$,
which is possible since $\phi(0+)=\infty$. Then since $x_{\ell}<\epsilon<\frac{1}{\beta}\wedge\frac{1}{q}\le\frac{1}{q}<x_{k}<1$,
$\phi(x_{\ell})>\phi(\epsilon)>\phi(\frac{1}{q})\vee\phi(1)>\phi(x_{k})$
so that the numerator in the right-hand side is positive for all sufficiently
large $N$. Thus, we may take $\bm{y}:=\bm{x}-\frac{\mathfrak{e}_{k}}{N}+\frac{\mathfrak{e}_{\ell}}{N}$.
\end{proof}
Next, we focus on the inner part $\mathcal{I}_{N}^{\epsilon}$. Recall
from Definition \ref{def:C-critical} that $\mathcal{C}$ is the set
of critical points of $F_{\beta}$. According to Lemma \ref{lem:critical-pts},
$\mathcal{C}$ is a finite set. Thus, we may choose $\epsilon>0$
small enough such that $\mathcal{C}$ is contained in the interior
of $\mathcal{I}^{\epsilon}$.

For $\alpha>0$ and $\bm{c}\in\mathcal{C}$, define
\[
\mathcal{N}_{\bm{c}}^{\alpha}:=\{\bm{y}\in\Xi:\|\bm{y}-\bm{c}\|_{L^{\infty}}<\alpha\}\qquad\text{and}\qquad\mathcal{N}^{\alpha}:=\bigcup_{\bm{c}\in\mathcal{C}}\mathcal{N}_{\bm{c}}^{\alpha},
\]
where $\|\cdot\|_{L^{\infty}}$ is the supremum norm. Here, $\alpha$
can be chosen small such that the sets $\mathcal{N}_{\bm{c}}^{\alpha},\bm{c}\in\mathcal{C}$
are disjoint and $\mathcal{N}^{\alpha}\subset\mathcal{I}^{\epsilon}$.
\begin{equation}
\text{the sets}\,\mathcal{N}_{\bm{c}}^{\alpha},\bm{c}\in\mathcal{C}\,\text{are disjoint and}\,\mathcal{N}^{\alpha}\subset\mathcal{I}^{\epsilon}.\label{eq:alpha}
\end{equation}
Write
\[
\mathcal{N}_{\bm{c},N}^{\alpha}:=\mathcal{N}_{\bm{c}}^{\alpha}\cap\Xi_{N},\qquad\mathcal{N}_{N}^{\alpha}:=\mathcal{N}^{\alpha}\cap\Xi_{N}=\bigcup_{\bm{c}\in\mathcal{C}}\mathcal{N}_{\bm{c},N}^{\alpha}.
\]
Recall \eqref{eq:FNk} and \eqref{eq:FN0}. For small enough $\alpha$,
we have $\mathcal{N}_{{\bf u}_{k},N}^{\alpha}\subset\mathcal{F}_{N}^{k}$
for each $k\in\llbracket1,q\rrbracket$ and $\mathcal{N}_{{\bf e},N}^{\alpha}\subset\mathcal{F}_{N}^{0}$
if $\beta\in(\beta_{1},q)$.

Let $\mathcal{M}^{\epsilon,\alpha}:=\mathcal{I}^{\epsilon}\setminus\mathcal{N}^{\alpha}$,
which is a compact set. Note that the map
\[
\bm{x}\mapsto\sum_{k=1}^{q-1}|\partial_{k}F_{\beta}(\bm{x})|^{2}
\]
is continuous on $\mathcal{M}^{\epsilon,\alpha}$ and does not vanish
since $\mathcal{C}\cap\mathcal{M}^{\epsilon,\alpha}=\varnothing$.
Thus, there exists a constant $c_{\epsilon,\alpha}>0$ such that
\begin{equation}
\sum_{k=1}^{q-1}|\partial_{k}F_{\beta}(\bm{x})|^{2}\ge c_{\epsilon,\alpha}\qquad\text{for all}\quad\bm{x}\in\mathcal{M}^{\epsilon,\alpha}.\label{eq:M-bdd-away}
\end{equation}
Let $\mathcal{M}_{N}^{\epsilon,\alpha}:=\mathcal{M}^{\epsilon,\alpha}\cap\Xi_{N}$.
From $\mathcal{M}_{N}^{\epsilon,\alpha}$, there always exists a jump
which lowers the energy:
\begin{lem}
\label{l2}For all sufficiently large $N$, given any $\bm{x}\in\mathcal{M}_{N}^{\epsilon,\alpha}$
there exists $\bm{y}\in\Xi_{N}$ with $r_{N}(\bm{x},\bm{y})>0$ such
that $F_{\beta}(\bm{x})>F_{\beta}(\bm{y})$.
\end{lem}

\begin{proof}
According to \eqref{eq:M-bdd-away}, there exists $k\in\llbracket1,q-1\rrbracket$
such that $|\partial_{k}F_{\beta}(\bm{x})|\ge\sqrt{\frac{c_{\epsilon,\alpha}}{q-1}}=:c_{\epsilon,\alpha}'$.
By \eqref{eq:Fbeta-diff},
\[
\partial_{k}F_{\beta}(\bm{x})=\phi(x_{q})-\phi(x_{k}).
\]
First assume that $\partial_{k}F_{\beta}(\bm{x})\ge c_{\epsilon,\alpha}'$.
By Lemma \ref{lem:jump-bound-1},
\[
F_{\beta}(\bm{x})-F_{\beta}\left(\bm{x}-\frac{\mathfrak{e}_{k}}{N}+\frac{\mathfrak{e}_{q}}{N}\right)\ge\frac{\phi\left(x_{q}+\frac{1}{N}\right)-\phi\left(x_{k}-\frac{1}{N}\right)-\frac{1}{N}}{N}.
\]
Since $\partial_{k}F_{\beta}(\bm{x})=\phi(x_{q})-\phi(x_{k})\ge c_{\epsilon,\alpha}'>0$
and $\phi$ is uniformly continuous on $\mathcal{M}^{\epsilon,\alpha}$,
the numerator in the right-hand side is strictly positive for all
large $N$. Thus, setting $\bm{y}:=\bm{x}-\frac{\mathfrak{e}_{k}}{N}+\frac{\mathfrak{e}_{q}}{N}$
proves the lemma. If $\partial_{k}F_{\beta}(\bm{x})\le-c_{\epsilon,\alpha}'$
then similarly setting $\bm{y}:=\bm{x}-\frac{\mathfrak{e}_{q}}{N}+\frac{\mathfrak{e}_{k}}{N}$
concludes the proof.
\end{proof}
We are left to handle $\mathcal{N}_{N}^{\alpha}$. Recall the decomposition
$\mathcal{C}=\mathcal{C}_{1}\cup\mathcal{C}_{2}\cup\mathcal{C}_{3}\cup\mathcal{C}_{4}$
from Definition \ref{def:critical-pts}. For $n\in\llbracket1,4\rrbracket$,
define
\[
\mathcal{N}_{n,N}^{\alpha}:=\bigcup_{\bm{c}\in\mathcal{C}_{n}}\mathcal{N}_{\bm{c},N}^{\alpha}\qquad\text{so that}\qquad\mathcal{N}_{N}^{\alpha}=\mathcal{N}_{1,N}^{\alpha}\cup\mathcal{N}_{2,N}^{\alpha}\cup\mathcal{N}_{3,N}^{\alpha}\cup\mathcal{N}_{4,N}^{\alpha}.
\]
Note that we have the following decomposition:
\begin{equation}
\Xi_{N}=\mathcal{O}_{N}^{\epsilon}\cup\mathcal{M}_{N}^{\epsilon,\alpha}\cup\mathcal{N}_{1,N}^{\alpha}\cup\mathcal{N}_{2,N}^{\alpha}\cup\mathcal{N}_{3,N}^{\alpha}\cup\mathcal{N}_{4,N}^{\alpha}.\label{eq:XiN-decomp}
\end{equation}
First, we deal with $\mathcal{N}_{1,N}^{\alpha}=\mathcal{N}_{{\bf e},N}^{\alpha}$.
Mind that $\mathcal{F}_{N}^{0}=\varnothing$ if and only if $\beta\ge q$.
\begin{lem}
\label{l3}Let $\beta\in[q,\infty)$. There exists $\alpha_{1}>0$
such that for all $\alpha\in(0,\alpha_{1})$, we have the following:
For sufficiently large $N$, from each $\bm{x}\in\mathcal{N}_{1,N}^{\alpha}$
there exists a path to some $\bm{y}\in\Xi_{N}$ such that $F_{\beta}(\bm{y})<\min_{\mathcal{N}_{1,N}^{\alpha}}F_{\beta}$
and its energy is at most $\max_{\mathcal{N}_{1,N}^{\alpha}}F_{\beta}$.
\end{lem}

\begin{proof}
Take $\alpha_{1}>0$ small enough so that for all $\alpha\in(0,\alpha_{1})$,
\[
F_{\beta}({\bf u}_{1})+\frac{q-1}{N}\|\nabla F_{\beta}\|_{L^{\infty}(\mathcal{I}^{\epsilon})}<\min_{\mathcal{N}_{1,N}^{\alpha}}F_{\beta}.
\]
Starting from any $\bm{x}\in\mathcal{N}_{1,N}^{\alpha}$, there exists
a path inside $\mathcal{N}_{1,N}^{\alpha}$ from $\bm{x}$ to ${\bf e}(N)\in\mathcal{N}_{1,N}^{\alpha}$,
where ${\bf e}(N)$ is chosen to be any point in $\mathcal{N}_{1,N}^{\alpha}$
that satisfies $\|{\bf e}(N)-{\bf e}\|_{L^{\infty}}<\frac{1}{N}$.

Next, define $\Phi:[0,1]\to\mathbb{R}$ as $\Phi(t):=F_{\beta}({\bf e}+t({\bf u}_{1}-{\bf e}))$,
a function on the linear trajectory from ${\bf e}$ to ${\bf u}_{1}$.
Then by a simple calculation,
\[
\Phi'(t)=(q-1)\left(\frac{1}{q}-u_{1}\right)\left[\phi\left(u_{1}+t\left(\frac{1}{q}-u_{1}\right)\right)-\phi\left(1-(q-1)\left(u_{1}+t\left(\frac{1}{q}-u_{1}\right)\right)\right)\right].
\]
According to Lemma \ref{lem:u1v1q}-(2), $\Phi'(t)<0$ for all $t\in(0,1)$.
This implies that $\Phi$ is a descending trajectory (which is quite
natural looking at the energy landscape in Figure \ref{fig4.1}).

Now, we define a path from ${\bf e}(N)$ as follows. At each step,
jump in the direction of $\frac{1}{N}(\mathfrak{e}_{1}-\mathfrak{e}_{k})$
for each $k\in\llbracket2,q\rrbracket$ consecutively. Repeating this
step $\lfloor N(\frac{1}{q}-u_{1})\rfloor$ times, we arrive at
\[
{\bf u}_{1}(N):={\bf e}(N)+\left\lfloor N\left(\frac{1}{q}-u_{1}\right)\right\rfloor \left(\frac{q-1}{N}\mathfrak{e}_{1}-\frac{1}{N}\mathfrak{e}_{2}-\cdots-\frac{1}{N}\mathfrak{e}_{q}\right).
\]
Then,
\[
\|{\bf u}_{1}(N)-{\bf u}_{1}\|_{L^{\infty}}<\left\Vert \frac{q-1}{N}\mathfrak{e}_{1}-\frac{1}{N}\mathfrak{e}_{2}-\cdots-\frac{1}{N}\mathfrak{e}_{q}\right\Vert _{L^{\infty}}=\frac{q-1}{N}.
\]
In addition, every element of this path is at most $\frac{1}{N}$
$L^{\infty}$-distance away from the linear trajectory from ${\bf e}$
to ${\bf u}_{1}$. This gives us a path from ${\bf e}(N)$ to ${\bf u}_{1}(N)$
whose energy is at most $F_{\beta}({\bf e})+\frac{1}{N}\|\nabla F_{\beta}\|_{L^{\infty}(\mathcal{I}^{\epsilon})}$
and 
\[
F_{\beta}({\bf u}_{1}(N))\le F_{\beta}({\bf u}_{1})+\frac{q-1}{N}\|\nabla F_{\beta}\|_{L^{\infty}(\mathcal{I}^{\epsilon})}.
\]
Thus for all large $N$
\[
F_{\beta}({\bf u}_{1}(N))<\min_{\mathcal{N}_{1,N}^{\alpha}}F_{\beta}\qquad\text{and}\qquad F_{\beta}({\bf e})+\frac{1}{N}\|\nabla F_{\beta}\|_{L^{\infty}(\mathcal{I}^{\epsilon})}\le\max_{\mathcal{N}_{1,N}^{\alpha}}F_{\beta}.
\]
Concatenating the two paths as $\bm{x}\to{\bf e}(N)\to{\bf u}_{1}(N)$
and letting $\bm{y}:={\bf u}_{1}(N)$ finish the proof.
\end{proof}
We go forward. since $\mathcal{N}_{2,N}^{\alpha}\subset\mathcal{F}_{N}^{\llbracket1,q\rrbracket}$,
we proceed to $\mathcal{N}_{3,N}^{\alpha}$.
\begin{lem}
\label{l4}There exists $\alpha_{2}>0$ such that for all $\alpha\in(0,\alpha_{2})$,
we have the following: For sufficiently large $N$, from each $\bm{x}\in\mathcal{N}_{3,N}^{\alpha}$
there exists a path to some $\bm{y}\in\Xi_{N}$ such that $F_{\beta}(\bm{y})<\min_{\mathcal{N}_{3,N}^{\alpha}}F_{\beta}$
and its energy is bounded by $\max_{\mathcal{N}_{3,N}^{\alpha}}F_{\beta}$.
\end{lem}

\begin{proof}
Recall that $\mathcal{C}_{3}=\{{\bf v}_{1},\dots,{\bf v}_{q}\}$.
Using the same logic as in Lemma \ref{l3}, it suffices to find for
each $k\in\llbracket1,q\rrbracket$ a linear descending trajectory
from ${\bf v}_{k}$. If $\beta\in(\beta_{1},q]$ then by a similar
argument as in Lemma \ref{l3}, $F_{\beta}$ decreases along the linear
trajectory from ${\bf v}_{k}$ to ${\bf u}_{k}$, so we are done.

Suppose that $\beta\in(q,\infty)$. Without loss of generality, let
$k=q$. Consider $\Phi:t\mapsto F_{\beta}({\bf v}_{q}+\gamma t(\mathfrak{e}_{2}-\mathfrak{e}_{1}))$,
$t\in[0,1]$, for small $\gamma>0$. Differentiating,
\begin{equation}
\Phi'(t)=-\gamma\left(\frac{1}{\beta}-\phi(v_{1}-\gamma t)\right)+\gamma\left(\frac{1}{\beta}-\phi(v_{1}+\gamma t)\right)=\gamma(\phi(v_{1}-\gamma t)-\phi(v_{1}+\gamma t)).\label{eq:l4}
\end{equation}
Since $v_{1}>\frac{1}{q}>\frac{1}{\beta}$ (cf. \eqref{eq:vi-ineq})
in this regime and $\phi$ increases on $(\frac{1}{\beta},1]$, the
right-hand side of \eqref{eq:l4} is negative for all $t\in(0,1)$
if $\gamma$ is chosen as $\gamma:=v_{1}-\frac{1}{\beta}$. This gives
a descending trajectory from ${\bf v}_{q}$ to ${\bf v}_{q}+\gamma(\mathfrak{e}_{2}-\mathfrak{e}_{1})$.
Taking $\alpha_{2}>0$ sufficiently small so that $F_{\beta}({\bf v}_{q}+\gamma(\mathfrak{e}_{2}-\mathfrak{e}_{1}))<\min_{\mathcal{N}_{3,N}^{\alpha}}F_{\beta}$
for all $\alpha\in(0,\alpha_{2})$ completes the proof.
\end{proof}
Finally, we handle $\mathcal{N}_{4,N}^{\alpha}$.
\begin{lem}
\label{l5}There exists $\alpha_{3}>0$ such that for all $\bm{c}\in\mathcal{C}_{4}$,
$\alpha\in(0,\alpha_{3})$, we have the following: For sufficiently
large $N$, from each $\bm{x}\in\mathcal{N}_{\bm{c},N}^{\alpha}$
there exists a path to some $\bm{y}\in\Xi_{N}$ such that $F_{\beta}(\bm{y})<\min_{\mathcal{N}_{\bm{c},N}^{\alpha}}F_{\beta}$
and its energy is bounded by $\max_{\mathcal{N}_{\bm{c},N}^{\alpha}}F_{\beta}$.
\end{lem}

\begin{proof}
As before, let us find for each $\bm{c}\in\mathcal{C}_{4}$ a linear
trajectory from $\bm{c}$ to another $\bm{y}$ such that $F_{\beta}(\bm{y})<F_{\beta}(\bm{c})$.

According to Lemma \ref{lem:two-coordinates}, for any critical point
in $\bm{c}=(c_{1},\dots,c_{q})\in\mathcal{C}_{4}$, there exist two
distinct indices $k,\ell\in\llbracket1,q\rrbracket$ such that $c_{k}=c_{\ell}>\frac{1}{\beta}$.
Thus as verified in \eqref{eq:l4}, the linear trajectory from $\bm{c}$
to $\bm{c}+\gamma(\mathfrak{e}_{\ell}-\mathfrak{e}_{k})$ where $\gamma=\gamma(\bm{c}):=c_{k}-\frac{1}{\beta}$
descends. Finally, taking $\alpha_{3}>0$ sufficiently small so that
$F_{\beta}(\bm{c}+\gamma(\mathfrak{e}_{\ell}-\mathfrak{e}_{k}))<\min_{\mathcal{N}_{\bm{c},N}^{\alpha}}F_{\beta}$
for all $\bm{c}\in\mathcal{C}_{4}$ and $\alpha\in(0,\alpha_{3})$
completes the proof.
\end{proof}
Now, we are ready to prove Lemma \ref{lem:almost-desc-path} for $q\ge3$.
\begin{proof}[Proof of Lemma \ref{lem:almost-desc-path}: $q\ge3$]
 Fix $\delta>0$. Observe that for any $\bm{x}\in\mathcal{N}_{\bm{c}}^{\alpha}$
where $\bm{c}\in\mathcal{C}$,
\[
|F_{\beta}(\bm{x})-F_{\beta}(\bm{c})|\le\alpha\|\nabla F_{\beta}\|_{L^{\infty}(\mathcal{I}^{\epsilon})}.
\]
Let $\alpha_{1}$, $\alpha_{2}$, and $\alpha_{3}$ be given in Lemmas
\eqref{l3}, \ref{l4}, and \ref{l5}. Since $\mathcal{I}^{\epsilon}$
is compact, we can take $\alpha\in(0,\alpha_{1}\wedge\alpha_{2}\wedge\alpha_{3})$
small enough such that $\max_{\mathcal{N}_{\bm{c},N}^{\alpha}}F_{\beta}-\min_{\mathcal{N}_{\bm{c},N}^{\alpha}}F_{\beta}<\delta$
for all $\bm{c}\in\mathcal{C}$. In this way, combining Lemmas \ref{l1},
\ref{l2}, \ref{l3}, \ref{l4}, and \ref{l5} constructs a path from
any $\bm{x}\in\Xi_{N}$ (cf. \eqref{eq:XiN-decomp}) to $\mathcal{F}_{N}^{\llbracket0,q\rrbracket}$
whose energy is at most $F_{\beta}(\bm{x})+\delta$. This concludes
the proof.
\end{proof}

\subsection{\label{sec4.3}Proof of Lemma \ref{lem:rec2}}

In this subsection, we assume throughout that $q\ge3$ and $\beta\in(\beta_{2},q)$
to prove Lemma \ref{lem:rec2}. In this specific case, there exists
a shallow valley $\mathcal{E}_{N}^{0}$ containing ${\bf e}$ which
is negligible in the deepest metastable time-scale $\theta_{N}$ but
still contributes an exponential slowdown to mixing because of its
stability. To handle this effect, we need to formulate another level
of metastable transitions regarding the escape from ${\bf e}$ to
the deeper well $\widehat{\mathcal{W}}_{1}$ that contains all $\mathcal{E}_{N}^{k}$
for $k\in\llbracket1,q\rrbracket$.

To define new metastable sets, we divide into two cases: $\beta\in(\beta_{2},\beta_{3}]$
and $\beta\in(\beta_{3},q)$. Refer again to Figure \ref{fig1.2}.
\begin{itemize}
\item $\beta\in(\beta_{2},\beta_{3}]$: There are $q+1$ connected components
$\mathcal{W}_{0},\mathcal{W}_{1},\dots,\mathcal{W}_{q}$ of $\{F_{\beta}<F_{\beta}({\bf v}_{1})\}$.
In this case, define
\[
\widehat{\mathcal{W}}_{k}:=\mathcal{W}_{k}\qquad\text{for}\quad k\in\llbracket0,q\rrbracket.
\]
\item $\beta\in(\beta_{3},q)$ (when $q\ge5$): As summarized in Case 4
of Section \ref{sec1.4}, there exist $2$ connected components of
$\{F_{\beta}<F_{\beta}({\bf v}_{1})\}$, $\mathcal{W}_{0}$ (which
contains ${\bf e}$) and $\widehat{\mathcal{W}}_{1}$ (which contains
${\bf u}_{1},\dots,{\bf u}_{q}$). Define $\widehat{\mathcal{W}}_{0}:=\mathcal{W}_{0}$.
See Figure \ref{fig4.2} for this case.
\end{itemize}
\begin{center}
\begin{figure}
\begin{centering}
\begin{tikzpicture}
\begin{scope}[scale=0.7]
\draw[thick] (0,4) sin (2,0) cos (3,2) sin (4,4) cos (4.5,3.5) sin (5,3) cos (5.5,3.5) sin (6,4) cos (7,2) sin (8,0) cos (10,4);
\fill[blue!50!white] (0.1,0) rectangle (5,1.6);
\fill[teal!50!white] (9.9,0) rectangle (5,1.6);
\fill[white,domain=0:2,variable=\x] (0,-0.1)--(0,4)--plot({\x},{4-4*sin(\x*45)})--(2,-0.1)--cycle;
\fill[white,domain=2:8,variable=\x] (2,-0.1)--plot({\x},{1+sin((\x-3.5)*60)})--(8,-0.1)--cycle;
\fill[white,domain=8:10,variable=\x] (8,-0.1)--(8,0)--plot({\x},{4-4*sin((\x-6)*45)})--(10,-0.1)--cycle;

\draw[thick] (0,4) sin (2,0);
\draw[thick,densely dashed] (2,0) cos (3,2) sin (4,4) cos (4.5,3.5) sin (5,3) cos (5.5,3.5) sin (6,4) cos (7,2) sin (8,0);
\draw[thick] (8,0) cos (10,4);
\draw[thick] (2,0) cos (3.5,1) sin (5,2) cos (6.5,1) sin (8,0);
\fill (5,3) circle (0.06); \draw (5,3-0.1) node[below]{$\bf e$};
\fill (2,0) circle (0.06); \draw (2,0-0.1) node[below]{${\bf u}_k$};
\fill (8,0) circle (0.06); \draw (8,0-0.1) node[below]{${\bf u}_\ell$};
\fill (4,4) circle (0.06); \draw (4,4+0.1) node[above]{${\bf v}_k$};
\fill (6,4) circle (0.06); \draw (6,4+0.1) node[above]{${\bf v}_\ell$};
\fill (5,2) circle (0.06); \draw (5,2-0.1) node[below]{${\bf u}_{k,\ell}$};

\draw[blue] (2,1.7) node[above]{$\mathcal{W}_k$};
\draw[teal] (8,1.7) node[above]{$\mathcal{W}_\ell$};
\end{scope}

\begin{scope}[shift={(8,0)},scale=0.7]
\draw[thick] (0,4) sin (2,0) cos (3,2) sin (4,4) cos (4.5,3.5) sin (5,3) cos (5.5,3.5) sin (6,4) cos (7,2) sin (8,0) cos (10,4);
\fill[red!50!white] (0.1,0) rectangle (9.9,3.6);
\fill[white,domain=0:2,variable=\x] (0,-0.1)--(0,4)--plot({\x},{4-4*sin(\x*45)})--(2,-0.1)--cycle;
\fill[white,domain=2:8,variable=\x] (2,-0.1)--(2,0)--plot({\x},{1+sin((\x-3.5)*60)})--(8,-0.1)--cycle;
\fill[white,domain=8:10,variable=\x] (8,-0.1)--(8,0)--plot({\x},{4-4*sin((\x-6)*45)})--(10,-0.1)--cycle;

\draw[thick] (0,4) sin (2,0);
\draw[thick,densely dashed] (2,0) cos (3,2) sin (4,4) cos (4.5,3.5) sin (5,3) cos (5.5,3.5) sin (6,4) cos (7,2) sin (8,0);
\draw[thick] (8,0) cos (10,4);
\draw[thick] (2,0) cos (3.5,1) sin (5,2) cos (6.5,1) sin (8,0);
\fill (5,3) circle (0.06); \draw (5,3-0.1) node[below]{$\bf e$};
\fill (2,0) circle (0.06); \draw (2,0-0.1) node[below]{${\bf u}_k$};
\fill (8,0) circle (0.06); \draw (8,0-0.1) node[below]{${\bf u}_\ell$};
\fill (4,4) circle (0.06); \draw (4,4+0.1) node[above]{${\bf v}_k$};
\fill (6,4) circle (0.06); \draw (6,4+0.1) node[above]{${\bf v}_\ell$};
\fill (5,2) circle (0.06); \draw (5,2-0.1) node[below]{${\bf u}_{k,\ell}$};

\draw[red] (8,3.7) node[above]{$\widehat{\mathcal{W}}_1$};
\end{scope}
\end{tikzpicture}
\par\end{centering}
\caption{\label{fig4.2}The case of $q\ge5$ and $\beta\in(\beta_{3},q)$.
The bigger well $\widehat{\mathcal{W}}_{1}$ contains $\mathcal{W}_{k}$
for all $k\in\llbracket1,q\rrbracket$.}
\end{figure}
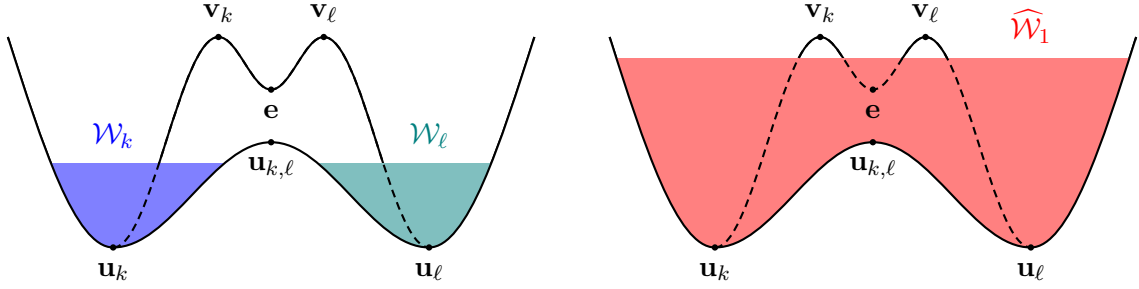
\par\end{center}

Let
\[
\widehat{\mathfrak{S}}_{\beta}:=\begin{cases}
\llbracket0,q\rrbracket & \text{if}\quad\beta\in(\beta_{2},\beta_{3}],\\
\{0,1\} & \text{if}\quad\beta\in(\beta_{3},q).
\end{cases}
\]
Recall the definition \eqref{eq:eta-def} of $\eta$. For $k\in\widehat{\mathfrak{S}}_{\beta}$,
let 
\begin{equation}
\widehat{\mathcal{E}}_{N}^{k}:=\Pi_{N}^{-1}\left(\widehat{\mathcal{W}}_{k}\cap\{F_{\beta}<F_{\beta}({\bf v}_{1})-\eta\}\cap\Xi_{N}\right),\label{eq:ENk-hat-def}
\end{equation}
define $\widehat{\mathcal{E}}_{N}:=\bigcup_{k\in\widehat{\mathfrak{S}}_{\beta}}\widehat{\mathcal{E}}_{N}^{k}$,
and
\[
\widehat{\mathcal{E}}_{N}^{A}:=\bigcup_{k\in A}\widehat{\mathcal{E}}_{N}^{k}\qquad\text{for}\quad A\subset\widehat{\mathfrak{S}}_{\beta}.
\]
Denote by $\widehat{D}_{\beta}$ the depth of $\widehat{\mathcal{W}}_{0}$:
\[
\widehat{D}_{\beta}:=F_{\beta}({\bf v}_{1})-F_{\beta}({\bf e}).
\]
Note that $\widehat{D}_{\beta}<D_{\beta}$ by \eqref{eq:depth-bigger}.
Let
\[
\widehat{\theta}_{N}=\widehat{\theta}_{N}^{\beta}:=2\pi Ne^{N\widehat{D}_{\beta}},
\]
such that $\widehat{\theta}_{N}\ll\theta_{N}$ (cf. \eqref{eq:thetaN-CWP}).
Let $\{\widehat{\mathfrak{X}}_{\beta}(t)\}_{t\ge0}$ be the Markov
chain in $\widehat{\mathfrak{S}}_{\beta}$ with jump rates $\widehat{\mathfrak{r}}_{\beta}:\widehat{\mathfrak{S}}_{\beta}\times\widehat{\mathfrak{S}}_{\beta}\to[0,\infty)$
defined as (cf. \eqref{eq:omega-def} and \eqref{eq:nu-def})
\[
\widehat{\mathfrak{r}}_{\beta}(k,\ell):=\begin{cases}
\frac{\omega}{\nu}\bm{1}_{\{k=0\}} & \text{if}\quad\beta\in(\beta_{2},\beta_{3}],\\
\frac{q\omega}{\nu}\bm{1}_{\{k=0\}} & \text{if}\quad\beta\in(\beta_{3},q).
\end{cases}
\]
Denote by $\{\widetilde{\widehat{\bm{\sigma}}}_{N}^{\beta}(t)\}_{t\ge0}$
its trace process in $\widehat{\mathcal{E}}_{N}$. Define $\widehat{\Psi}_{N}:\widehat{\mathcal{E}}_{N}\to\widehat{\mathfrak{S}}_{\beta}$
by
\[
\widehat{\Psi}_{N}(\sigma):=\sum_{k\in\widehat{\mathfrak{S}}_{\beta}}k\bm{1}_{\left\{ \sigma\in\widehat{\mathcal{E}}_{N}^{k}\right\} }.
\]
The following theorem is due to \cite[Theorem 4.5]{Lee22}.
\begin{thm}
\label{thm:meta-0}Let $q\ge3$, $\beta\in(\beta_{2},q)$, and fix
$k\in\widehat{\mathfrak{S}}_{\beta}$.
\begin{enumerate}
\item For any $\sigma_{N}\in\widehat{\mathcal{E}}_{N}^{k}$, the law of
$\{\widehat{\Psi}_{N}(\widetilde{\widehat{\bm{\sigma}}}_{N}^{\beta}(\widehat{\theta}_{N}t))\}_{t\ge0}$
starting from $\sigma_{N}$ converges to the law of $\{\widehat{\mathfrak{X}}_{\beta}(t)\}_{t\ge0}$
starting from $k$ as $N\to\infty$.
\item For any $T>0$,
\[
\lim_{N\to\infty}\sup_{\sigma\in\widehat{\mathcal{E}}_{N}^{k}}\mathbb{E}_{\sigma}^{N,\beta}\left[\int_{0}^{T}\bm{1}_{\left\{ \bm{\sigma}_{N}^{\beta}\left(\widehat{\theta}_{N}t\right)\in\Omega_{N}\setminus\widehat{\mathcal{E}}_{N}\right\} }\,{\rm d}t\right]=0.
\]
\end{enumerate}
\end{thm}

We need an additional lemma.
\begin{lem}
\label{lem:descend-to-min}Assume $q\ge5$ and $\beta\in(\beta_{3},q)$.
There exists $\varrho_{N}\ll\theta_{N}$ such that
\[
\lim_{N\to\infty}\sup_{\sigma\in\widehat{\mathcal{E}}_{N}^{1}}\mathbb{P}_{\sigma}^{N,\beta}[\mathcal{H}_{\mathcal{E}_{N}}>\varrho_{N}]=0.
\]
\end{lem}

\begin{proof}
Recall that $\mathfrak{S}_{\beta}=\llbracket1,q\rrbracket$ thus $\mathcal{E}_{N}=\mathcal{E}_{N}^{\llbracket1,q\rrbracket}$
in this regime. By Lemma \ref{lem:rec1}, there exists $\varrho_{N}\ll\theta_{N}$
such that
\[
\lim_{N\to\infty}\sup_{\sigma\in\widehat{\mathcal{E}}_{N}^{1}}\mathbb{P}_{\sigma}^{N,\beta}\left[\mathcal{H}_{\mathcal{E}_{N}^{\llbracket0,q\rrbracket}}>\varrho_{N}\right]=0.
\]
Thus to prove the lemma, it suffices to prove that
\[
\lim_{N\to\infty}\sup_{\sigma\in\widehat{\mathcal{E}}_{N}^{1}}\mathbb{P}_{\sigma}^{N,\beta}\left[\mathcal{H}_{\mathcal{E}_{N}^{0}}<\mathcal{H}_{\mathcal{E}_{N}}\right]=0,
\]
or equivalently, for the proportions chain,
\begin{equation}
\lim_{N\to\infty}\sup_{\bm{x}\in\widehat{\mathcal{F}}_{N}^{1}}{\rm P}_{\bm{x}}^{N,\beta}\left[\mathcal{H}_{\mathcal{F}_{N}^{0}}<\mathcal{H}_{\mathcal{F}_{N}}\right]=0,\label{eq:descend-prop-chain}
\end{equation}
where $\widehat{\mathcal{F}}_{N}^{k}:=\Pi_{N}(\widehat{\mathcal{E}}_{N}^{k})$
for $k\in\widehat{\mathfrak{S}}_{\beta}$. By \eqref{eq:renewal},
for each $\bm{x}\in\widehat{\mathcal{F}}_{N}^{1}$,
\begin{equation}
{\rm P}_{\bm{x}}^{N,\beta}\left[\mathcal{H}_{\mathcal{F}_{N}^{0}}<\mathcal{H}_{\mathcal{F}_{N}}\right]\le\frac{{\rm cap}_{N}\left(\bm{x},\mathcal{F}_{N}^{0}\right)}{{\rm cap}_{N}(\bm{x},\mathcal{F}_{N})}.\label{eq:descend-1}
\end{equation}
For the denominator in \eqref{eq:descend-1}, Lemma \ref{lem:almost-desc-path}
with $\delta=\frac{\eta}{2}>0$ implies that, for all large $N$,
there exists a path from $\bm{x}$ to $\mathcal{F}_{N}^{\llbracket0,q\rrbracket}$
whose energy is at most $F_{\beta}(\bm{x})+\frac{\eta}{2}$. By definition,
this path cannot arrive at $\mathcal{F}_{N}^{0}$ since $\Phi_{\beta}(\bm{x},\mathcal{F}_{N}^{0})=F_{\beta}({\bf v}_{1})>F_{\beta}(\bm{x})+\eta$
(see \eqref{eq:ENk-hat-def}). Thus, applying the Thomson principle
\eqref{eq:TP} to this path as done in \eqref{eq:pf3},
\begin{equation}
{\rm cap}_{N}(\bm{x},\mathcal{F}_{N})\ge\frac{\pi_{N}^{\beta}(\bm{x})}{c_{1}e^{\beta}N^{2q+1}e^{\frac{\beta N\eta}{2}}}.\label{eq:descend-2}
\end{equation}
Next, consider the numerator in \eqref{eq:descend-1}. Apply the Dirichlet
principle \eqref{eq:DP} to $f:=1-{\bf 1}_{\mathcal{W}_{0}\cap\Xi_{N}}$
which clearly satisfies $f(\bm{x})=1$ and $f=0$ on $\mathcal{F}_{N}^{0}=\widehat{\mathcal{F}}_{N}^{0}$.
Then,
\begin{equation}
{\rm cap}_{N}\left(\bm{x},\mathcal{F}_{N}^{0}\right)\le\mathscr{D}_{N}\left(1-{\bf 1}_{\mathcal{W}_{0}\cap\Xi_{N}}\right)=\sum_{\bm{y}\in\mathcal{W}_{0}\cap\Xi_{N}}\sum_{\bm{z}\notin\mathcal{W}_{0}\cap\Xi_{N}}\pi_{N}^{\beta}(\bm{y})r_{N}(\bm{y},\bm{z}).\label{eq:DP-appl-1}
\end{equation}
Since $\mathcal{W}_{0}$ is a connected component of $\{F_{\beta}<F_{\beta}({\bf v}_{1})\}$,
for all such pairs $(\bm{y},\bm{z})$ with $r_{N}(\bm{y},\bm{z})>0$
we have $F_{\beta}(\bm{z})\ge F_{\beta}({\bf v}_{1})-\frac{c}{N}>F_{\beta}(\bm{x})+\eta-\frac{c}{N}$,
where $c>0$ does not depend on $N$. Thus by Lemma \ref{lem:piNbeta-rough}
and the fact that $r_{N}(\bm{y},\bm{z})\le e^{\beta}$, the right-hand
side of \eqref{eq:DP-appl-1} is bounded by
\begin{equation}
\sum_{\bm{y}\in\mathcal{W}_{0}\cap\Xi_{N}}\sum_{\bm{z}\notin\mathcal{W}_{0}\cap\Xi_{N}}c_{1}N^{q}e^{-\beta N\eta}e^{\beta c}\pi_{N}^{\beta}(\bm{x})\times e^{\beta}\le c_{1}e^{\beta(c+1)}N^{3q}e^{-\beta\eta N}\pi_{N}^{\beta}(\bm{x}),\label{eq:DP-appl-2}
\end{equation}
where we have also used the fact that $|\mathcal{F}_{N}^{0}|,|\Xi_{N}\setminus\mathcal{F}_{N}^{0}|<|\Xi_{N}|\le N^{q}$.
Combining \eqref{eq:descend-1}, \eqref{eq:descend-2}, \eqref{eq:DP-appl-1},
and \eqref{eq:DP-appl-2},
\[
{\rm P}_{\bm{x}}^{N,\beta}\left[\mathcal{H}_{\mathcal{F}_{N}^{0}}<\mathcal{H}_{\mathcal{F}_{N}}\right]\le c_{1}^{2}e^{\beta(c+2)}N^{5q+1}e^{-\frac{\beta\eta N}{2}}\ll1.
\]
Taking the supremum over all $\bm{x}\in\widehat{\mathcal{F}}_{N}^{1}$
proves \eqref{eq:descend-prop-chain} and thus the lemma.
\end{proof}
Now, we are ready to prove Lemma \ref{lem:rec2}.
\begin{proof}[Proof of Lemma \ref{lem:rec2}]
 Fix $q\ge3$, $\beta\in(\beta_{2},q)$, and $\sigma\in\mathcal{E}_{N}^{0}$.
Let ${\bf P}_{\sigma}^{N,\beta}$ and $\widehat{{\bf P}}_{0}^{\beta}$
be the laws of $\{\widehat{\Psi}_{N}(\widetilde{\widehat{\bm{\sigma}}}_{N}^{\beta}(\widehat{\theta}_{N}t))\}_{t\ge0}$
and $\{\widehat{\mathfrak{X}}_{\beta}(t)\}_{t\ge0}$ starting from
$\sigma$ and $0$, respectively. By Theorem \ref{thm:meta-0}, ${\bf P}_{\sigma}^{N,\beta}$
converges weakly to $\widehat{{\bf P}}_{0}^{\beta}$ in the Skorokhod
topology. Take any $\varrho_{N}$ such that $\widehat{\theta}_{N}\ll\varrho_{N}\ll\theta_{N}$.
Since $\{\mathcal{H}_{\widehat{\mathfrak{S}}_{\beta}\setminus\{0\}}\ge s\}$
is closed in the Skorokhod topology for all $s>0$,
\begin{equation}
\begin{aligned}\limsup_{N\to\infty}\sup_{\sigma\in\mathcal{E}_{N}^{0}}\mathbb{P}_{\sigma}^{N,\beta}\left[\mathcal{H}_{\widehat{\mathcal{E}}_{N}\setminus\widehat{\mathcal{E}}_{N}^{0}}\ge\varrho_{N}\right] & \le\limsup_{N\to\infty}\sup_{\sigma\in\mathcal{E}_{N}^{0}}\mathbb{P}_{\sigma}^{N,\beta}\left[\mathcal{H}_{\widehat{\mathcal{E}}_{N}\setminus\widehat{\mathcal{E}}_{N}^{0}}\ge\widehat{\theta}_{N}s\right]\\
 & \le\limsup_{N\to\infty}\sup_{\sigma\in\mathcal{E}_{N}^{0}}{\bf P}_{\sigma}^{N,\beta}\left[\mathcal{H}_{\widehat{\mathfrak{S}}_{\beta}\setminus\{0\}}\ge s\right]\le\widehat{{\bf P}}_{0}^{\beta}\left[\mathcal{H}_{\widehat{\mathfrak{S}}_{\beta}\setminus\{0\}}\ge s\right].
\end{aligned}
\label{eq:rec2}
\end{equation}
Since $0$ is a transient state of $\{\widehat{\mathfrak{X}}_{\beta}(t)\}_{t\ge0}$,
\[
\lim_{s\to\infty}\widehat{{\bf P}}_{0}^{\beta}\left[\mathcal{H}_{\widehat{\mathfrak{S}}_{\beta}\setminus\{0\}}\ge s\right]=0.
\]
Since \eqref{eq:rec2} holds for all $s>0$,
\begin{equation}
\limsup_{N\to\infty}\sup_{\sigma\in\mathcal{E}_{N}^{0}}\mathbb{P}_{\sigma}^{N,\beta}\left[\mathcal{H}_{\widehat{\mathcal{E}}_{N}\setminus\widehat{\mathcal{E}}_{N}^{0}}\ge\varrho_{N}\right]=0.\label{eq:rec2-2}
\end{equation}
If $\beta\in(\beta_{2},\beta_{3}]$, then $\widehat{\mathcal{E}}_{N}\setminus\widehat{\mathcal{E}}_{N}^{0}=\widehat{\mathcal{E}}_{N}^{\llbracket1,q\rrbracket}=\mathcal{E}_{N}$,
so that the last display proves Lemma \ref{lem:rec2}. Thus, it remains
to prove it in the case of $q\ge5$ and $\beta\in(\beta_{3},q)$,
in which case $\widehat{\mathcal{E}}_{N}\setminus\widehat{\mathcal{E}}_{N}^{0}=\widehat{\mathcal{E}}_{N}^{1}$.
Lemma \ref{lem:descend-to-min} implies that
\begin{equation}
\lim_{N\to\infty}\sup_{\sigma'\in\widehat{\mathcal{E}}_{N}^{1}}\mathbb{P}_{\sigma'}^{N,\beta}[\mathcal{H}_{\mathcal{E}_{N}}>\varrho_{N}']=0,\label{eq:rec2-3}
\end{equation}
for some $\varrho_{N}'\ll\theta_{N}$. The two displays \eqref{eq:rec2-2}
and \eqref{eq:rec2-3} conclude the proof via the strong Markov property
and the fact that $\varrho_{N}+\varrho_{N}'\ll\theta_{N}$.
\end{proof}

\section{\label{sec5}Proof of Theorem \ref{thm:CWP-M}}

In this section, we prove Theorem \ref{thm:CWP-M}. For each $k\in\mathfrak{S}_{\beta}$,
define $\mathcal{B}_{N}^{k}\subset\mathcal{E}_{N}^{k}$ as
\begin{equation}
\mathcal{B}_{N}^{k}:=\Pi_{N}^{-1}(\mathcal{W}_{k}\cap\{F_{\beta}<F_{\beta}({\bf u}_{k})+\epsilon\}\cap\Xi_{N}),\label{eq:BNk-def}
\end{equation}
where $\epsilon>0$ is a small constant to be specified in the sequel.
Then, write
\[
\mathcal{A}_{N}^{k}:=\Pi_{N}\left(\mathcal{B}_{N}^{k}\right)=\mathcal{W}_{k}\cap\{F_{\beta}<F_{\beta}({\bf u}_{k})+\epsilon\}\cap\Xi_{N}.
\]
We start with the first part.
\begin{proof}[Proof of Theorem \ref{thm:CWP-M}-(1)]
This part is almost identical to Lemma \ref{lem:descend-to-min}.
Namely, the same logic of proving Lemma \ref{lem:rec1-1} allows us
to prove here that
\[
\lim_{N\to\infty}\sup_{\bm{x}\in\mathcal{F}_{N}^{k}}{\rm P}_{\bm{x}}^{N,\beta}\left[\mathcal{H}_{\mathcal{A}_{N}^{\llbracket0,q\rrbracket}}>\delta\theta_{N}\right]=0.
\]
Thus, it suffices to prove that
\[
\lim_{N\to\infty}\sup_{\bm{x}\in\mathcal{F}_{N}^{k}}{\rm P}_{\bm{x}}^{N,\beta}\left[\mathcal{H}_{\mathcal{A}_{N}^{\llbracket0,q\rrbracket}\setminus\mathcal{A}_{N}^{k}}<\mathcal{H}_{\mathcal{A}_{N}^{k}}\right]=0.
\]
The rest of the proof uses the same steps as those in the proof of
\eqref{eq:descend-prop-chain}. The tedious details are omitted.
\end{proof}
Next, we deal with the second part of Theorem \ref{thm:CWP-M}.
\begin{proof}[Proof of Theorem \ref{thm:CWP-M}-(2)]
 We apply Lemma \ref{lem:prob-cap} with $\bm{x}\in\mathcal{A}_{N}^{k}$,
$B=\Xi_{N}\setminus\mathcal{F}_{N}^{k}$, and $\gamma^{-1}=2\varrho_{N}>0$.
Then,
\begin{equation}
{\rm P}_{\bm{x}}^{N,\beta}\left[\mathcal{H}_{\Xi_{N}\setminus\mathcal{F}_{N}^{k}}\le2\varrho_{N}\right]^{2}\le2e^{2}\varrho_{N}\frac{{\rm cap}_{N}\left(\bm{x},\Xi_{N}\setminus\mathcal{F}_{N}^{k}\right)}{\pi_{N}^{\beta}(\bm{x})}.\label{eq:part-2-1}
\end{equation}
For the capacity in the numerator, applying the Dirichlet principle
\eqref{eq:DP} with $f:=1-\bm{1}_{\Xi_{N}\setminus\mathcal{F}_{N}^{k}}$
gives
\[
{\rm cap}_{N}\left(\bm{x},\Xi_{N}\setminus\mathcal{F}_{N}^{k}\right)\le\sum_{\bm{y}\in\mathcal{F}_{N}^{k}}\sum_{\bm{z}\in\Xi_{N}\setminus\mathcal{F}_{N}^{k}}\pi_{N}^{\beta}(\bm{z})r_{N}(\bm{z},\bm{y}).
\]
Taking $\epsilon\in(0,\eta)$ such that $F_{\beta}({\bf u}_{1})+5\epsilon<F_{\beta}({\bf v}_{1})-\eta$,
we have $F_{\beta}(\bm{x})<F_{\beta}({\bf v}_{1})-\eta-4\epsilon$.
Thus, we obtain
\begin{equation}
{\rm cap}_{N}\left(\bm{x},\Xi_{N}\setminus\mathcal{F}_{N}^{k}\right)\le c_{1}e^{\beta(c+1)}N^{3q}e^{-4\beta\epsilon N}\pi_{N}^{\beta}(\bm{x}),\label{eq:part-2-2}
\end{equation}
as done in \eqref{eq:DP-appl-1} and \eqref{eq:DP-appl-2}. Combining
\eqref{eq:part-2-1} and \eqref{eq:part-2-2},
\[
{\rm P}_{\bm{x}}^{N,\beta}\left[\mathcal{H}_{\Xi_{N}\setminus\mathcal{F}_{N}^{k}}\le2\varrho_{N}\right]^{2}\le2e^{2}c_{1}e^{\beta(c+1)}N^{3q}e^{-4\beta N\epsilon}\varrho_{N}.
\]
Therefore, taking $\varrho_{N}:=e^{3\beta N\epsilon}$ concludes the
proof.
\end{proof}
All it remains is to prove part (3) of Theorem \ref{thm:CWP-M} for
$\varrho_{N}=e^{3\beta N\epsilon}$.
\begin{proof}[Proof of Theorem \ref{thm:CWP-M}-(3)]
 By the definition of mixing time, to prove part (3) it suffices
to prove that
\[
\sup_{k\in\mathfrak{S}_{\beta}}T_{\frac{1}{N}}^{{\rm mix}}\left(\bm{\sigma}_{N}^{\mathcal{E}_{N}^{k}}\right)\le\varrho_{N}.
\]
By \cite[Theorem 20.6]{LPW17},
\[
T_{\frac{1}{N}}^{{\rm mix}}\left(\bm{\sigma}_{N}^{\mathcal{E}_{N}^{k}}\right)\le cT^{{\rm rel}}\left(\bm{\sigma}_{N}^{\mathcal{E}_{N}^{k}}\right)\log\frac{N}{\min_{\sigma\in\mathcal{E}_{N}^{k}}\mu_{N}^{\mathcal{E}_{N}^{k}}(\sigma)}.
\]
Here, $T^{{\rm rel}}(\bm{\sigma}_{N}^{\mathcal{E}_{N}^{k}})$ denotes
the \emph{relaxation time} (cf. \cite[Section 12.2]{LPW17}) of the
dynamics $\bm{\sigma}_{N}^{\mathcal{E}_{N}^{k}}$ and $c>0$ is a
constant depending on the uniform bound of the holding rate of $\bm{\sigma}_{N}^{\mathcal{E}_{N}^{k}}$,
which does not depend on $N$. Note that $-\frac{N}{2}\le\mathbb{H}_{N}(\cdot)\le-\frac{N}{2q}$.
Thus,
\[
\mu_{N}^{\mathcal{E}_{N}^{k}}(\sigma)=\frac{\mu_{N}^{\beta}(\sigma)}{\sum_{\xi\in\mathcal{E}_{N}^{k}}\mu_{N}^{\beta}(\xi)}\ge\frac{e^{\frac{\beta N}{2q}}}{N^{q}e^{\frac{\beta N}{2}}}\gg Ne^{-\frac{\beta N}{2}},
\]
where we used $|\mathcal{E}_{N}^{k}|<|\Xi_{N}|\le N^{q}$. Thus, it
remains to prove for each $k\in\mathfrak{S}_{\beta}$ that
\[
T^{{\rm rel}}\left(\bm{\sigma}_{N}^{\mathcal{E}_{N}^{k}}\right)\ll\frac{\varrho_{N}}{N}.
\]
A global bound from \cite[Theorem 1.1 and Section 7]{Her23} implies
that
\[
T^{{\rm rel}}\left(\bm{\sigma}_{N}^{\mathcal{E}_{N}^{k}}\right)\le c'\max_{\mathcal{A}\subsetneq\mathcal{E}_{N}^{k}:\,\mu_{N}^{\mathcal{E}_{N}^{k}}(\mathcal{A})\ge\frac{1}{2}}\overline{\mathbb{E}}_{\mu_{N}^{\mathcal{E}_{N}^{k}\setminus\mathcal{A}}}^{N,\beta}[\mathcal{H}_{\mathcal{A}}],
\]
where $c'>0$ is a universal constant and $\mu_{N}^{\mathcal{E}_{N}^{k}\setminus\mathcal{A}}$
is the measure $\mu_{N}^{\beta}$ conditioned on $\mathcal{E}_{N}^{k}\setminus\mathcal{A}$.
Here, $\overline{\mathbb{E}}_{\cdot}^{N,\beta}$ denotes the expectation
with respect to the law of the reflected dynamics $\bm{\sigma}_{N}^{\mathcal{E}_{N}^{k}}$.
By an elementary property of expectations,
\[
\overline{\mathbb{E}}_{\mu_{N}^{\mathcal{E}_{N}^{k}\setminus\mathcal{A}}}^{N,\beta}[\mathcal{H}_{\mathcal{A}}]\le\max_{\sigma\in\mathcal{E}_{N}^{k}\setminus\mathcal{A}}\overline{\mathbb{E}}_{\sigma}^{N,\beta}[\mathcal{H}_{\mathcal{A}}]\le\max_{\sigma\in\mathcal{E}_{N}^{k}}\overline{\mathbb{E}}_{\sigma}^{N,\beta}[\mathcal{H}_{\mathcal{A}}].
\]
Note that $\mu_{N}^{\mathcal{E}_{N}^{k}}(\mathcal{B}_{N}^{k})\simeq1$,
thus any $\mathcal{A}\subsetneq\mathcal{E}_{N}^{k}$ with $\mu_{N}^{\mathcal{E}_{N}^{k}}(\mathcal{A})\ge\frac{1}{2}$
satisfies $\mathcal{A}\cap\mathcal{B}_{N}^{k}\ne\varnothing$ for
all large $N$. Hence,
\[
\max_{\sigma\in\mathcal{E}_{N}^{k}}\overline{\mathbb{E}}_{\sigma}^{N,\beta}[\mathcal{H}_{\mathcal{A}}]\le\max_{\sigma\in\mathcal{E}_{N}^{k}}\overline{\mathbb{E}}_{\sigma}^{N,\beta}[\mathcal{H}_{\mathcal{A}\cap\mathcal{B}_{N}^{k}}]\le\max_{\sigma\in\mathcal{E}_{N}^{k}}\max_{\xi\in\mathcal{A}\cap\mathcal{B}_{N}^{k}}\overline{\mathbb{E}}_{\sigma}^{N,\beta}[\mathcal{H}_{\xi}]\le\max_{\sigma\in\mathcal{E}_{N}^{k}}\max_{\xi\in\mathcal{B}_{N}^{k}}\overline{\mathbb{E}}_{\sigma}^{N,\beta}[\mathcal{H}_{\xi}].
\]
Thus, all that remains is to prove that
\[
\max_{\sigma\in\mathcal{E}_{N}^{k}}\max_{\xi\in\mathcal{B}_{N}^{k}}\overline{\mathbb{E}}_{\sigma}^{N,\beta}[\mathcal{H}_{\xi}]\ll\frac{\varrho_{N}}{N},
\]
or equivalently,
\[
\max_{\bm{x}\in\mathcal{F}_{N}^{k}}\max_{\bm{y}\in\mathcal{A}_{N}^{k}}\overline{{\rm E}}_{\bm{x}}^{N,\beta}[\mathcal{H}_{\bm{y}}]\ll\frac{\varrho_{N}}{N},
\]
where $\overline{{\rm E}}_{\bm{x}}^{N,\beta}$ denotes the expectation
with respect to the law of the reflected proportions chain in $\mathcal{F}_{N}^{k}$.
The proof of the final displayed asymptotics is identical to the proof
of Lemma \ref{lem:rec1-1}. Namely, we apply \eqref{eq:magic-formula}
to obtain
\[
\overline{{\rm E}}_{\bm{x}}^{N,\beta}[\mathcal{H}_{\bm{y}}]=\frac{\sum_{\bm{z}\in\mathcal{F}_{N}^{k}}\pi_{N}^{\mathcal{F}_{N}^{k}}(\bm{z})\overline{\mathfrak{h}}_{\bm{x},\bm{y}}(\bm{z})}{\overline{{\rm cap}}_{N}(\bm{x},\bm{y})},
\]
where $\overline{\mathfrak{h}}_{\bm{x},\bm{y}}$ and $\overline{{\rm cap}}_{N}(\bm{x},\bm{y})$
are the equilibrium potential and capacity, respectively, defined
for the reflected proportions chain. For the denominator, we construct
a path from $\bm{x}$ to $\bm{y}$ inside $\mathcal{F}_{N}^{k}$ whose
energy is at most $F_{\beta}(\bm{x})+\epsilon$ via Lemma \ref{lem:descend-to-min},
and apply the Thomson principle \eqref{eq:TP}. For the numerator,
we apply \eqref{eq:renewal} to get
\[
\overline{\mathfrak{h}}_{\bm{x},\bm{y}}(\bm{z})\le\frac{\overline{{\rm cap}}_{N}(\bm{z},\bm{x})}{\overline{{\rm cap}}_{N}(\bm{z},\bm{y})},
\]
and apply two principles \eqref{eq:DP} and \eqref{eq:TP} to the
numerator and denominator, respectively. Then, we obtain that
\[
\max_{\bm{x}\in\mathcal{F}_{N}^{k}}\max_{\bm{y}\in\mathcal{A}_{N}^{k}}\overline{{\rm E}}_{\bm{x}}^{N,\beta}[\mathcal{H}_{\bm{y}}]\le CN^{5q+2}e^{2\beta N\epsilon}\ll\frac{\varrho_{N}}{N},
\]
where $C>0$ does not depend on $N$. We omit the repetition of technical
details.
\end{proof}
\begin{acknowledgement*}
The authors would like to thank Instituto Superior T\'ecnico (Lisbon)
for their warm hospitality during their stay in February 2026, during
which the manuscript was completed. SK has been supported by the Basic
Science Research Program through the National Research Foundation
of Korea funded by the Ministry of Science and ICT (RS-2025-00518980),
the Yonsei University Research Fund of 2025 (2025-22-0133), and the
POSCO Science Fellowship of POSCO TJ Park Foundation.
\end{acknowledgement*}

\appendix

\section{\label{secA}Energy Landscape of the CWP Model}

In Appendix \ref{secA}, we put together the energy landscape of the
CWP model. Assume throughout that $q\ge3$. Many of the results are
taken from \cite{Lee22} and are stated without proof.

Recall from \eqref{eq:phi-def} that $\phi:(0,1]\to\mathbb{R}$ is
given as $\phi(t)=t-\frac{1}{\beta}\log t$. In the language of Notation
\ref{nota:q-to-q-1}, an elementary calculation yields that
\begin{equation}
\partial_{k}F_{\beta}(\bm{x})=\phi(x_{q})-\phi(x_{k})\qquad\text{for}\quad k\in\llbracket1,q-1\rrbracket.\label{eq:Fbeta-diff}
\end{equation}
Thus, $\bm{x}\in\Xi$ is a critical point of $F_{\beta}$ if and only
if
\begin{equation}
\phi(x_{k})=\phi(x_{\ell})\qquad\text{for all}\quad k,\ell\in\llbracket1,q\rrbracket.\label{eq:phi-const}
\end{equation}
Since $\phi(t)=c$ has at most two solutions in $(0,1)$, any critical
point of $F_{\beta}$ is either the equiproportional vector ${\bf e}$
defined in \eqref{eq:e-def}, or of the form
\[
\left(\overbrace{t,\dots,t}^{q-i},\overbrace{\frac{1-(q-i)t}{i},\dots,\frac{1-(q-i)t}{i}}^{i}\right)
\]
or its permutation, where $i\in\llbracket1,\frac{q}{2}\rrbracket$
and $t\ne\frac{1}{q}$. This element satisfies \eqref{eq:phi-const}
if and only if $g_{i}(t)=\beta$ where $g_{i}:(0,\frac{1}{q-i})\to\mathbb{R}$
is defined as
\begin{equation}
g_{i}(t):=\begin{cases}
\frac{i}{1-qt}\log\frac{1-(q-i)t}{it} & \text{if}\quad t\ne\frac{1}{q},\\
q & \text{if}\quad t=\frac{1}{q}.
\end{cases}\label{eq:gi-def}
\end{equation}
Note that $g_{i}$ is continuous. It is elementary to verify that
(see \cite[Lemma 6.1]{Lee22})
\begin{itemize}
\item $g_{i}$ has a unique minimum, denoted by $m_{i}\in(0,\frac{1}{q})$;
\item $g_{i}$ decreases on $(0,m_{i}]$ and increases on $[m_{i},\frac{1}{q-i})$
with $g_{i}(0+)=g_{i}(\frac{1}{q-i}-)=\infty$;
\item if $\beta>g_{i}(m_{i})=:\beta_{s,i}$, the equation $g_{i}(t)=\beta$
has two solutions $u_{i}(\beta)=u_{i}<v_{i}=v_{i}(\beta)$ such that
$u_{i}\in(0,m_{i})$ and $v_{i}\in(m_{i},\frac{1}{q-i})$;
\item $\beta\mapsto u_{i}(\beta)$ decreases on $(\beta_{s,i},\infty)$
and $\beta\mapsto v_{i}(\beta)$ increases on $(\beta_{s,i},\infty)$.
Moreover, $v_{i}(q)=\frac{1}{q}$;
\end{itemize}
Since $u_{i}<m_{i}<\frac{1}{q}$, we have $u_{i}<\frac{1-(q-i)u_{i}}{i}$,
thus according to the graph of $\phi$,
\begin{equation}
u_{i}<\frac{1}{\beta}<\frac{1-(q-i)u_{i}}{i}.\label{eq:ui-ineq}
\end{equation}
Similarly, since $\beta\mapsto v_{i}(\beta)$ increases and $v_{i}(q)=\frac{1}{q}$,
\begin{equation}
\begin{cases}
v_{i}<\frac{1}{\beta}<\frac{1-(q-i)v_{i}}{i} & \text{if}\quad\beta\in(\beta_{s,i},q),\\
v_{i}=\frac{1}{\beta}=\frac{1-(q-i)v_{i}}{i} & \text{if}\quad\beta=q,\\
v_{i}>\frac{1}{\beta}>\frac{1-(q-i)v_{i}}{i} & \text{if}\quad\beta\in(q,\infty).
\end{cases}\label{eq:vi-ineq}
\end{equation}
Gathering the discussions above, we arrive at the following result.
Recall from Definition \ref{def:C-critical} that $\mathcal{C}$ denotes
the set of critical points of $F_{\beta}$.
\begin{lem}
\label{lem:critical-pts}We have $\bm{c}\in\mathcal{C}$ if and only
if $\bm{c}$ is a permutation of one of the following points: ${\bf e}$
(cf. \eqref{eq:e-def}), or for some $t_{i}\in\{u_{i},v_{i}\}$ with
$i\in\llbracket1,\frac{q}{2}\rrbracket$,
\[
\left(\overbrace{t_{i},\dots,t_{i}}^{q-i},\overbrace{\frac{1-(q-i)t_{i}}{i},\dots,\frac{1-(q-i)t_{i}}{i}}^{i}\right).
\]
\end{lem}

Now, we classify the critical points. For $\beta>\beta_{s,1}$ and
$k\in\llbracket1,q\rrbracket$, define
\begin{equation}
\begin{aligned}{\bf u}_{k}={\bf u}_{k}(\beta) & :=(u_{1},\dots,1-(q-1)u_{1},\dots,u_{1}),\\
{\bf v}_{k}={\bf v}_{k}(\beta) & :=(v_{1},\dots,1-(q-1)v_{1},\dots,v_{1}),
\end{aligned}
\label{eq:ukvk-def}
\end{equation}
whose $k$-th components are $1-(q-1)u_{1}$ or $1-(q-1)v_{1}$ and
the others are all $u_{1}$ or $v_{1}$. For $\beta>\beta_{s,2}$
and $k,\ell\in\llbracket1,q\rrbracket$ with $k\ne\ell$, define
\begin{equation}
{\bf u}_{k,\ell}={\bf u}_{k,\ell}(\beta):=\left(u_{2},\dots,\frac{1-(q-2)u_{2}}{2},\dots,\frac{1-(q-2)u_{2}}{2},\dots,u_{2}\right),\label{eq:ukl-def}
\end{equation}
whose $k$-th and $\ell$-th components are $\frac{1-(q-2)u_{2}}{2}$
and the others are all $u_{2}$.
\begin{defn}
\label{def:critical-pts}Define $\mathcal{C}_{1}:=\{{\bf e}\}$, $\mathcal{C}_{2}:=\{{\bf u}_{1},\dots,{\bf u}_{q}\}$,
$\mathcal{C}_{3}:=\{{\bf v}_{1},\dots,{\bf v}_{q}\}$, and
\[
\mathcal{C}_{4}:=\bigcup_{i\in\left\llbracket 2,\frac{q}{2}\right\rrbracket }\bigcup_{t_{i}\in\{u_{i},v_{i}\}}\left\{ \text{permutations of}\enspace\left(t_{i},\dots,t_{i},\frac{1-(q-i)t_{i}}{i},\dots,\frac{1-(q-i)t_{i}}{i}\right)\right\} \setminus\{{\bf e}\}.
\]
By Lemma \ref{lem:critical-pts}, we have a decomposition $\mathcal{C}=\mathcal{C}_{1}\cup\mathcal{C}_{2}\cup\mathcal{C}_{3}\cup\mathcal{C}_{4}$.
Note that $\mathcal{C}_{4}\ne\varnothing$ only if $q\ge4$.
\end{defn}

For the critical points in $\mathcal{C}_{4}$, the following property
will be exploited in Lemma \ref{l5}.
\begin{lem}
\label{lem:two-coordinates}For any $\bm{x}\in\mathcal{C}_{4}$, there
exist two different indices $k,\ell\in\llbracket1,q\rrbracket$ such
that $x_{k},x_{\ell}>\frac{1}{\beta}$.
\end{lem}

\begin{proof}
May assume that $q\ge4$. Take $i\in\llbracket2,\frac{q}{2}\rrbracket$
and $t_{i}\in\{u_{i},v_{i}\}$ when $\beta\ne q$, or\footnote{Note that since $v_{i}=\frac{1}{q}$ when $\beta=q$, the critical
point corresponding to $v_{i}$ is ${\bf e}$, which belongs to $\mathcal{C}_{1}$.} $t_{i}=u_{i}$ when $\beta=q$. By \eqref{eq:ui-ineq} and \eqref{eq:vi-ineq},
exactly one of $t_{i}$ or $\frac{1-(q-i)t_{i}}{i}$ is bigger than
$\frac{1}{\beta}$. As both of them appear at least twice in the coordinates
since $i\ge2$ and $q-i\ge2$, the result follows.
\end{proof}
Finally, we define the critical (inverse) temperatures. Define $\beta_{c}=\beta_{c}(q)$
as (cf. \cite[display (3.3)]{CET05})
\[
\beta_{c}:=\frac{2(q-1)}{q-2}\log(q-1).
\]
By \cite[Lemmas 6.4 and 6.5]{Lee22}, we have the following comparison:
\begin{itemize}
\item for $q=3$, $\beta_{s,1}<\beta_{c}<q$;
\item for $q=4$, $\beta_{s,1}<\beta_{c}<\beta_{s,2}=q$;
\item for $q\ge5$, $\beta_{s,1}<\beta_{c}<\beta_{s,2}<q$.
\end{itemize}
By \cite[Lemma 6.6]{Lee22},
\begin{equation}
\begin{aligned}F_{\beta}({\bf u}_{1})=\cdots=F_{\beta}({\bf u}_{q})>F_{\beta}({\bf e}) & \qquad\text{for}\quad\beta\in(\beta_{s,1},\beta_{c}),\\
F_{\beta}({\bf u}_{1})=\cdots=F_{\beta}({\bf u}_{q})=F_{\beta}({\bf e}) & \qquad\text{for}\quad\beta=\beta_{c},\\
F_{\beta}({\bf u}_{1})=\cdots=F_{\beta}({\bf u}_{q})<F_{\beta}({\bf e}) & \qquad\text{for}\quad\beta\in(\beta_{c},\infty).
\end{aligned}
\label{eq:beta-c}
\end{equation}
By \cite[Lemma 6.7]{Lee22}, for $q\ge5$, there exists one more critical
temperature $\beta_{m}\in(\beta_{s,2},q)$ such that
\begin{equation}
\begin{aligned}F_{\beta}({\bf u}_{1,2})=\cdots=F_{\beta}({\bf u}_{q-1,q})>F_{\beta}({\bf v}_{1})=\cdots=F_{\beta}({\bf v}_{q}) & \qquad\text{for}\quad\beta\in(\beta_{s,2},\beta_{m}),\\
F_{\beta}({\bf u}_{1,2})=\cdots=F_{\beta}({\bf u}_{q-1,q})=F_{\beta}({\bf v}_{1})=\cdots=F_{\beta}({\bf v}_{q}) & \qquad\text{for}\quad\beta=\beta_{m},\\
F_{\beta}({\bf u}_{1,2})=\cdots=F_{\beta}({\bf u}_{q-1,q})<F_{\beta}({\bf v}_{1})=\cdots=F_{\beta}({\bf v}_{q}) & \qquad\text{for}\quad\beta\in(\beta_{m},q).
\end{aligned}
\label{eq:beta-m}
\end{equation}

\begin{defn}
\label{def:cri-temp}For $q\ge3$, define the critical temperatures
as
\[
\beta_{1}:=\beta_{s,1}\qquad\text{and}\qquad\beta_{2}:=\beta_{c}.
\]
For $q\in\{3,4\}$, define $\beta_{3}:=q$. For $q\ge5$, define $\beta_{3}:=\beta_{m}$.
\end{defn}

Next, we prove \eqref{eq:depth-bigger}. Again refer to Figure \ref{fig1.2}.
\begin{lem}
\label{lem:depth-bigger}Suppose that $q\ge5$ and $\beta\in(\beta_{3},q)$.
Then, the depth of $\mathcal{W}_{k}$, $k\in\llbracket1,q\rrbracket$
is strictly bigger than the depth of $\mathcal{W}_{0}$:
\[
F_{\beta}({\bf u}_{1,2})-F_{\beta}({\bf u}_{1})>F_{\beta}({\bf v}_{1})-F_{\beta}({\bf e}).
\]
\end{lem}

\begin{proof}
Recall from \eqref{eq:Dbeta-def} that $D_{\beta}=F_{\beta}({\bf u}_{1,2})-F_{\beta}({\bf u}_{1})$.
Define, for $\beta\in[\beta_{3},q]$,
\[
\widehat{D}_{\beta}:=F_{\beta}({\bf v}_{1})-F_{\beta}({\bf e}).
\]
Our objective is to prove that $D_{\beta}-\widehat{D}_{\beta}>0$
for $\beta\in(\beta_{3},q)$. Note that if $\beta=\beta_{3}$ then
by \eqref{eq:beta-c} and \eqref{eq:beta-m},
\[
D_{\beta_{3}}-\widehat{D}_{\beta_{3}}=(F_{\beta_{3}}({\bf u}_{1,2})-F_{\beta_{3}}({\bf u}_{1}))-(F_{\beta_{3}}({\bf v}_{1})-F_{\beta_{3}}({\bf e}))=F_{\beta_{3}}({\bf e})-F_{\beta_{3}}({\bf u}_{1})>0.
\]
Therefore, it suffices to prove that
\begin{equation}
\frac{\partial D_{\beta}}{\partial\beta}-\frac{\partial\widehat{D}_{\beta}}{\partial\beta}>0\qquad\text{for}\quad\beta\in(\beta_{3},q).\label{eq:depth-WTS}
\end{equation}
First, from \cite[display (8.5)]{Lee22} we infer that
\begin{equation}
\beta^{2}\left(\frac{\partial D_{\beta}}{\partial\beta}-\frac{\partial\widehat{D}_{\beta}}{\partial\beta}\right)=k_{1}(u_{1})-k_{2}(u_{2})+k_{1}(v_{1})-\log\frac{1}{q},\label{eq:depth-1}
\end{equation}
where for $i\in\{1,2\}$,
\[
k_{i}(t):=(1-(q-i)t)\log\frac{1-(q-i)t}{it}+\log t\qquad\text{for}\quad t\in\left(0,\frac{1}{q-i}\right).
\]
By \cite[display (8.7)]{Lee22}, $k_{i}'(t)<0$ for $t\in(0,\frac{1}{q})$.
Since $0<u_{i}<m_{i}<v_{i}<\frac{1}{q}$ if $\beta\in(\beta_{3},q)$,
\begin{equation}
k_{1}(v_{1})>k_{1}\left(\frac{1}{q}\right)=\log\frac{1}{q}.\label{eq:depth-2}
\end{equation}
Next, we claim that
\begin{equation}
u_{1}<u_{2}.\label{eq:u1u2}
\end{equation}
Indeed, by the first display in the proof of \cite[Lemma 6.4]{Lee22},
$g_{1}<g_{2}$ on $(0,\frac{1}{q})$, thus
\[
g_{1}(u_{1})=\beta=g_{2}(u_{2})>g_{1}(u_{2}).
\]
By the definition of $u_{1}$, the inequality $g_{1}(u_{1})>g_{1}(u_{2})$
automatically implies that $u_{2}\in(u_{1},v_{1})$, which proves
\eqref{eq:u1u2}. This, along with the decreasing property of $k_{2}$,
gives
\begin{equation}
k_{2}(u_{2})<k_{2}(u_{1}).\label{eq:depth-3}
\end{equation}
Finally, we claim that $k_{1}(u_{1})\ge k_{2}(u_{1})$, which proves
\eqref{eq:depth-WTS} when combined with \eqref{eq:depth-1}, \eqref{eq:depth-2},
and \eqref{eq:depth-3}. Thus, we are left to prove that
\[
k_{1}(t)\ge k_{2}(t)\qquad\text{for all}\quad t\in\left(0,\frac{1}{q}\right).
\]
Letting $k_{1,2}(t):=k_{1}(t)-k_{2}(t)$, it suffices to prove that
\begin{equation}
k_{1,2}\left(\frac{1}{q}\right)=0\qquad\text{and}\qquad k_{1,2}'(t)\le0\qquad\text{for all}\quad t\in\left(0,\frac{1}{q}\right).\label{eq:depth-WTS-2}
\end{equation}
Clearly, $k_{1,2}(\frac{1}{q})=k_{1}(\frac{1}{q})-k_{2}(\frac{1}{q})=0$.
In addition,
\[
k_{1,2}'(t)=-(q-2)\log2+\log\frac{t}{1-(q-1)t}+(q-2)\log\frac{1-(q-2)t}{1-(q-1)t}.
\]
From this formula it is obvious that $k_{1,2}'(t)$ increases as $t$
increases. Thus,
\[
k_{1,2}'(t)\le k_{1,2}'\left(\frac{1}{q}\right)=-(q-2)\log2+(q-2)\log\frac{2/q}{1/q}=0,
\]
which finishes the proof of \eqref{eq:depth-WTS-2} and thus of the
lemma.
\end{proof}
The following property for $\phi$ is used in Section \ref{sec4.2}.
\begin{lem}
\label{lem:u1v1q}
\begin{enumerate}
\item If $\beta\in(\beta_{1},q)$, then $\phi(t)<\phi(1-(q-1)t)$ for all
$t\in(u_{1},v_{1})$.
\item If $\beta\in[q,\infty)$, then $\phi(t)<\phi(1-(q-1)t)$ for all $t\in(u_{1},\frac{1}{q})\cup(v_{1},\frac{1}{q-1})$.
\end{enumerate}
\end{lem}

\begin{proof}
Define $\psi(t):=\phi(t)-\phi(1-(q-1)t)$ for $t\in(0,\frac{1}{q-1})$.
Note that $\psi(0+)=\infty$ and $\psi(\frac{1}{q-1}-)=-\infty$.
In addition,
\begin{equation}
\psi'(t)=\phi'(t)+(q-1)\phi'(1-(q-1)t)=q-\frac{1}{\beta t(1-(q-1)t)}.\label{eq:psi'}
\end{equation}
This means that the sign of $\psi'$ changes at most two times in
$(0,\frac{1}{q-1})$.

Observe from \eqref{eq:phi-const} and Lemma \ref{lem:critical-pts}
that $\psi(u_{1})=\psi(v_{1})=\psi(\frac{1}{q})=0$. If $\beta\in(\beta_{1},q)$,
then $u_{1}<v_{1}<\frac{1}{q}$ are three zeros of $\psi$. This automatically
implies that $\psi<0$ on $(u_{1},v_{1})$, thus part (1) is verified.

If $\beta\in(q,\infty)$, then $u_{1}<\frac{1}{q}<v_{1}$ are three
zeros of $\psi$, thus similarly part (2) follows as well.

If $\beta=q$, then since $v_{1}=\frac{1}{q}$, to conclude the proof
it additionally requires that $\psi'(u_{1})<0$ and $\psi'(\frac{1}{q})=0$.
For the first statement, we may infer from \eqref{eq:gi-def} that
$g_{1}'(u_{1})<0$, which is equivalent to
\[
\frac{1}{1-qu_{1}}\left(\frac{q}{1-qu_{1}}\log\frac{1-(q-1)u_{1}}{u_{1}}-\frac{1}{u_{1}(1-(q-1)u_{1})}\right)<0.
\]
Substituting $g_{1}(u_{1})=\beta$ here, we get
\[
\frac{1}{1-qu_{1}}\left(\beta q-\frac{1}{u_{1}(1-(q-1)u_{1})}\right)<0.
\]
Since $u_{1}<\frac{1}{q}$, via \eqref{eq:psi'}, this gives us $\psi'(u_{1})<0$.
Finally, $\psi'(\frac{1}{q})=0$ follows easily from \eqref{eq:psi'}
and the fact that $\beta=q$.
\end{proof}

\section{\label{secB}Total Variation Distance}

In Appendix \ref{secB}, we review general results on the total variation
distance derived in \cite[Section 5]{Lee25}. The first lemma is from
\cite[Lemma 5.1]{Lee25}. Mind that its proof works even if the space
changes as $N$ varies.
\begin{lem}
\label{lem:TV-1}Suppose that $n\ge1$. Let $(E_{N})_{N\ge1}$ be
a sequence of finite sets and let $(\pi_{N})_{N\ge1}$ be a sequence
of measures on $E_{N}$. Let $A_{N}^{i}\subset E_{N}$, $i\in\llbracket1,n\rrbracket$,
be disjoint subsets. Suppose that the following limits exist:
\[
\nu_{i}:=\lim_{N\to\infty}\pi_{N}\left(A_{N}^{i}\right)\in[0,1)\qquad\text{for each}\quad i\in\llbracket1,n\rrbracket.
\]
Suppose further that $\sum_{i=1}^{n}\nu_{i}=1$. Then for any $a_{i}$,
$i\in\llbracket1,n\rrbracket$ such that $a_{i}\ge0$ and $\sum_{i=1}^{n}a_{i}=1$,
\[
\lim_{N\to\infty}d_{{\rm TV}}\left(\sum_{i=1}^{n}a_{i}\pi_{N}^{A_{N}^{i}},\pi_{N}\right)=\frac{1}{2}\sum_{i=1}^{n}|a_{i}-\nu_{i}|,
\]
where $\pi_{N}^{A_{N}^{i}}$ is the conditioned measure of $\pi_{N}$
on $A_{N}^{i}$.
\end{lem}

Suppose that $({\bf Q}_{t})_{t\ge0}$ is a probability semigroup defined
in a finite set $E$. The following is from \cite[Lemma 5.2]{Lee25}.
\begin{lem}
\label{lem:TV-2}Let $\mu,\nu$ be two probability measures on $E$.
Then for any $t\ge0$, $d_{{\rm TV}}(\mu,\nu)\ge d_{{\rm TV}}(\mu\mathbf{Q}_{t},\nu\mathbf{Q}_{t})$.
\end{lem}

The following final result is from \cite[Lemma 5.3]{Lee25}.
\begin{lem}
\label{lem:TV-3}Suppose that $({\bf Q}_{t})_{t\ge0}$ induces only
one irreducible class. Let $\mu\ne\nu$ be probability measures on
$E$. Then for any $t>0$,
\[
d_{{\rm TV}}(\mu,\nu)>d_{{\rm TV}}(\mu\mathbf{Q}_{t},\nu\mathbf{Q}_{t}).
\]
In particular, if $\nu$ is the stationary distribution,
\[
d_{{\rm TV}}(\mu,\nu)>d_{{\rm TV}}(\mu\mathbf{Q}_{t},\nu).
\]
\end{lem}

\section{\label{secC}Cyclic Decomposition of CWP Glauber Dynamics}

In Appendix \ref{secC}, we present the cyclic decomposition of the
proportions chain $\{\bm{S}_{N}^{\beta}(t)\}_{t\ge0}$ defined in
\eqref{eq:proportions-chain}. Let $\mathscr{L}_{N}$ be its infinitesimal
generator. As explained in Section \ref{sec3}, the model reduction
ingredient (Theorem \ref{thm:CWP-CD}) was proved in \cite{LS16,Lee22}
and its fundamental idea is to understand the proportions chain as
a (weighted) cyclic random walk on a potential field in the terminology
of \cite{LS18}. We briefly summarize it below.

Recall $F_{\beta,N}:\Xi_{N}\to\mathbb{R}$ from \eqref{eq:FbetaN-def}
and define a transition rate function $a_{\beta,N}:\Xi_{N}\times\Xi_{N}\to\mathbb{R}$
as
\[
a_{\beta,N}\left(\bm{x},\bm{x}-\frac{\mathfrak{e}_{k}}{N}+\frac{\mathfrak{e}_{\ell}}{N}\right):=e^{-\frac{N\beta}{2}\left[F_{\beta,N}\left(\bm{x}-\frac{\mathfrak{e}_{k}}{N}+\frac{\mathfrak{e}_{\ell}}{N}\right)-F_{\beta,N}(\bm{x})\right]}.
\]
The rate $a_{\beta,N}$ generates a cyclic random walk on each $2$-cycle
$\bm{x}\to\bm{x}-\frac{\mathfrak{e}_{k}}{N}+\frac{\mathfrak{e}_{\ell}}{N}\to\bm{x}$
on the potential field $\beta F_{\beta,N}$ (see \cite[eq. (2.2)]{LS18}).
In this sense, for $1\le k<\ell\le q$ and $\bm{x}\in\Xi_{N}$, let
$\mathscr{L}_{N,\bm{x}}^{k,\ell}$ be the infinitesimal generator
of this cyclic random walk defined by only two rates
\[
\begin{cases}
\bm{x}\to\bm{x}-\frac{\mathfrak{e}_{k}}{N}+\frac{\mathfrak{e}_{\ell}}{N} & \text{with rate}\quad a_{\beta,N}\left(\bm{x},\bm{x}-\frac{\mathfrak{e}_{k}}{N}+\frac{\mathfrak{e}_{\ell}}{N}\right),\\
\bm{x}-\frac{\mathfrak{e}_{k}}{N}+\frac{\mathfrak{e}_{\ell}}{N}\to\bm{x} & \text{with rate}\quad a_{\beta,N}\left(\bm{x}-\frac{\mathfrak{e}_{k}}{N}+\frac{\mathfrak{e}_{\ell}}{N},\bm{x}\right).
\end{cases}
\]
A simple computation via \eqref{eq:piNbeta-formula} and \eqref{eq:rN-value}
shows
\begin{equation}
\frac{r_{N}\left(\bm{x},\bm{x}-\frac{\mathfrak{e}_{k}}{N}+\frac{\mathfrak{e}_{\ell}}{N}\right)}{a_{\beta,N}\left(\bm{x},\bm{x}-\frac{\mathfrak{e}_{k}}{N}+\frac{\mathfrak{e}_{\ell}}{N}\right)}=\frac{r_{N}\left(\bm{x}-\frac{\mathfrak{e}_{k}}{N}+\frac{\mathfrak{e}_{\ell}}{N},\bm{x}\right)}{a_{\beta,N}\left(\bm{x}-\frac{\mathfrak{e}_{k}}{N}+\frac{\mathfrak{e}_{\ell}}{N},\bm{x}\right)}=\sqrt{x_{k}\left(x_{\ell}+\frac{1}{N}\right)}.\label{eq:dec-cond}
\end{equation}
Thus, we have a cyclic decomposition
\begin{equation}
\mathscr{L}_{N}=\sum_{1\le k<\ell\le q}\sum_{\bm{x}\in\Xi_{N}}w_{N}^{k,\ell}(\bm{x})\mathscr{L}_{N,\bm{x}}^{k,\ell},\label{eq:cyc-dec}
\end{equation}
where $w_{N}^{k,\ell}:\Xi\to\mathbb{R}$ is defined as
\[
w_{N}^{k,\ell}(\bm{x}):=\sqrt{x_{k}\left(x_{\ell}+\frac{1}{N}\right)}.
\]
It is obvious that $w_{N}^{k,\ell}$ is uniformly Lipschitz and converges
uniformly to $w^{k,\ell}(\bm{x}):=\sqrt{x_{k}x_{\ell}}$ on every
compact subset of $\Xi^{\circ}$. Thus, the proportions chain for
the original Glauber dynamics falls into the class of cyclic random
walks on potential fields in the sense of \cite[Remarks 2.10 and 2.11]{LS18}.
Note that the cyclic decomposition \eqref{eq:cyc-dec} is possible
due to the first equality in \eqref{eq:dec-cond}.

From this perspective, we verify Remark \ref{rem:other}. To this
end, it suffices to prove that \eqref{eq:dec-cond} holds for the
proportions chains corresponding to the heat-bath and Metropolis dynamics.
First for the heat-bath dynamics given in \eqref{eq:rate-HB-def},
the corresponding proportions chain via $\Phi_{N}:\Omega_{N}\to\Xi_{N}$
has its transition rate $r_{N}^{{\rm HB}}:\Xi_{N}\times\Xi_{N}\to\mathbb{R}$
given as
\[
r_{N}^{{\rm HB}}\left(\bm{x},\bm{x}-\frac{\mathfrak{e}_{k}}{N}+\frac{\mathfrak{e}_{\ell}}{N}\right)=\frac{x_{k}e^{-\beta NH\left(\bm{x}-\frac{\mathfrak{e}_{k}}{N}+\frac{\mathfrak{e}_{\ell}}{N}\right)}}{\sum_{m=1}^{q}e^{-\beta NH\left(\bm{x}-\frac{\mathfrak{e}_{k}}{N}+\frac{\mathfrak{e}_{m}}{N}\right)}}.
\]
Then, one can verify that
\[
\frac{r_{N}^{{\rm HB}}\left(\bm{x},\bm{x}-\frac{\mathfrak{e}_{k}}{N}+\frac{\mathfrak{e}_{\ell}}{N}\right)}{a_{\beta,N}\left(\bm{x},\bm{x}-\frac{\mathfrak{e}_{k}}{N}+\frac{\mathfrak{e}_{\ell}}{N}\right)}=\sqrt{x_{k}\left(x_{\ell}+\frac{1}{N}\right)}\frac{e^{\frac{\beta}{2}\left(x_{k}+x_{\ell}-\frac{1}{N}\right)}}{\sum_{m=1}^{q}e^{\beta x_{m}}}=\frac{r_{N}^{{\rm HB}}\left(\bm{x}-\frac{\mathfrak{e}_{k}}{N}+\frac{\mathfrak{e}_{\ell}}{N},\bm{x}\right)}{a_{\beta,N}\left(\bm{x}-\frac{\mathfrak{e}_{k}}{N}+\frac{\mathfrak{e}_{\ell}}{N},\bm{x}\right)},
\]
such that \eqref{eq:dec-cond} holds. Thus, the proportions chain
for the heat-bath Glauber dynamics can be written as
\[
\mathscr{L}_{N}^{{\rm HB}}=\sum_{1\le k<\ell\le q}\sum_{\bm{x}\in\Xi_{N}}w_{N}^{{\rm HB},k,\ell}(\bm{x})\mathscr{L}_{N,\bm{x}}^{k,\ell}
\]
where 
\[
w_{N}^{{\rm HB},k,\ell}(\bm{x}):=\sqrt{x_{k}\left(x_{\ell}+\frac{1}{N}\right)}\frac{e^{\frac{\beta}{2}\left(x_{k}+x_{\ell}-\frac{1}{N}\right)}}{\sum_{m=1}^{q}e^{\beta x_{m}}}.
\]
This indicates that Theorem \ref{thm:main} holds for the heat-bath
dynamics as well, provided that in \eqref{eq:Akl-def} the term $\sqrt{x_{k}x_{\ell}}$
is replaced by 
\[
w^{{\rm HB},k,\ell}(\bm{x}):=\sqrt{x_{k}x_{\ell}}\frac{e^{\frac{\beta}{2}(x_{k}+x_{\ell})}}{\sum_{m=1}^{q}e^{\beta x_{m}}}.
\]

Next, consider the Metropolis dynamics given in \eqref{eq:rate-MP-def}.
Its proportions chain rate $r_{N}^{{\rm MH}}:\Xi_{N}\times\Xi_{N}\to\mathbb{R}$
becomes
\[
r_{N}^{{\rm MH}}\left(\bm{x},\bm{x}-\frac{\mathfrak{e}_{k}}{N}+\frac{\mathfrak{e}_{\ell}}{N}\right)=x_{k}\exp\left\{ -N\beta\left[H\left(\bm{x}-\frac{\mathfrak{e}_{k}}{N}+\frac{\mathfrak{e}_{\ell}}{N}\right)-H(\bm{x})\right]_{+}\right\} .
\]
Then, we have
\[
\frac{r_{N}^{{\rm HB}}\left(\bm{x},\bm{x}-\frac{\mathfrak{e}_{k}}{N}+\frac{\mathfrak{e}_{\ell}}{N}\right)}{a_{\beta,N}\left(\bm{x},\bm{x}-\frac{\mathfrak{e}_{k}}{N}+\frac{\mathfrak{e}_{\ell}}{N}\right)}=\sqrt{x_{k}\left(x_{\ell}+\frac{1}{N}\right)}e^{-\frac{\beta}{2}\left|x_{k}-x_{\ell}-N^{-1}\right|}=\frac{r_{N}^{{\rm HB}}\left(\bm{x}-\frac{\mathfrak{e}_{k}}{N}+\frac{\mathfrak{e}_{\ell}}{N},\bm{x}\right)}{a_{\beta,N}\left(\bm{x}-\frac{\mathfrak{e}_{k}}{N}+\frac{\mathfrak{e}_{\ell}}{N},\bm{x}\right)},
\]
so \eqref{eq:dec-cond} holds again. This gives us
\[
\mathscr{L}_{N}^{{\rm MH}}=\sum_{1\le k<\ell\le q}\sum_{\bm{x}\in\Xi_{N}}w_{N}^{{\rm MH},k,\ell}(\bm{x})\mathscr{L}_{N,\bm{x}}^{k,\ell}
\]
where 
\[
w_{N}^{{\rm MH},k,\ell}(\bm{x}):=\sqrt{x_{k}\left(x_{\ell}+\frac{1}{N}\right)}e^{-\frac{\beta}{2}\left|x_{k}-x_{\ell}-N^{-1}\right|}.
\]
Therefore, Theorem \ref{thm:main} holds for the Metropolis dynamics
where in \eqref{eq:Akl-def} the term $\sqrt{x_{k}x_{\ell}}$ should
be replaced by $w^{{\rm MH},k,\ell}(\bm{x}):=\sqrt{x_{k}x_{\ell}}e^{-\frac{\beta}{2}|x_{k}-x_{\ell}|}$.
Note that the smoothness condition of each limit weight function $w^{k,\ell}$
in \cite[Remark 2.11]{LS18} can be dropped; a uniform convergence
condition on compact subsets of $\Xi^{\circ}$ suffices.

\section{\label{secD}Potential Theory}

In Appendix \ref{secD}, we review a few basic concepts and results
in potential theory. Consider an irreducible, reversible Markov chain
$\{\bm{S}(t)\}_{t\ge0}$ on a finite state space $\Xi$ with transition
rate $r(\cdot,\cdot)$. Denote by $\pi$ its unique stationary distribution.

Given two nonempty disjoint subsets $A,B\subset\Xi$, the \emph{equilibrium
potential} $\mathfrak{h}_{A,B}:\Xi\to[0,1]$ between $A$ and $B$
is defined as
\[
\mathfrak{h}_{A,B}(\bm{x}):={\rm P}_{\bm{x}}[\mathcal{H}_{A}<\mathcal{H}_{B}],
\]
where ${\rm P}_{\bm{x}}$ is the law starting from $\bm{x}\in\Xi$.
For any given function $f:\Xi\to\mathbb{R}$, its \emph{Dirichlet
form} $\mathscr{D}(f)$ is given as
\[
\mathscr{D}(f):=\frac{1}{2}\sum_{\bm{x},\bm{y}\in\Xi}\pi(\bm{x})r(\bm{x},\bm{y})(f(\bm{y})-f(\bm{x}))^{2}.
\]
Then, the \emph{capacity} ${\rm cap}(A,B)$ between $A$ and $B$
is defined as ${\rm cap}(A,B):=\mathscr{D}(\mathfrak{h}_{A,B})$.

The \emph{Dirichlet principle} (cf. \cite[Theorem 7.33]{BdH15}) implies
that for any functions $f:\Xi\to\mathbb{R}$ such that $f\equiv1$
on $A$ and $f\equiv0$ on $B$,
\begin{equation}
{\rm cap}(A,B)\le\mathscr{D}(f).\label{eq:DP}
\end{equation}
A sequence $\varphi:\llbracket0,n\rrbracket\to\Xi$ is a \emph{path}
from $\varphi(0)$ to $\varphi(n)$ if $r(\varphi(i),\varphi(i+1))>0$
for all $i\in\llbracket0,n-1\rrbracket$. The \emph{Thomson principle}
(cf. \cite[Theorem 7.37]{BdH15}) implies that for any path $\varphi$
from $A$ to $B$,
\begin{equation}
{\rm cap}(A,B)\ge\left(\sum_{i=0}^{n-1}\frac{1}{\pi(\varphi(i))r(\varphi(i),\varphi(i+1))}\right)^{-1}.\label{eq:TP}
\end{equation}
Inserting suitable test objects to \eqref{eq:DP} or \eqref{eq:TP},
we are able to obtain appropriate upper/lower bounds for the capacity.

A renewal estimate (cf. \cite[Lemma 8.4]{BdH15}) implies that
\begin{equation}
\mathfrak{h}_{A,B}(\bm{x})\le\frac{{\rm cap}(\bm{x},A)}{{\rm cap}(\bm{x},B)}.\label{eq:renewal}
\end{equation}
The so-called magic formula (cf. \cite[Corollary 7.30]{BdH15} or
\cite[Proposition 6.10]{BL10}) allows us to estimate the mean hitting
time via capacity:
\begin{equation}
{\rm E}_{\bm{x}}[\mathcal{H}_{B}]=\frac{\sum_{\bm{y}\in\Xi}\pi(\bm{y})\mathfrak{h}_{\bm{x},B}(\bm{y})}{{\rm cap}(\bm{x},B)}\qquad\text{for}\quad\bm{x}\notin B,\label{eq:magic-formula}
\end{equation}
where ${\rm E}_{\bm{x}}$ denotes the expectation with respect to
${\rm P}_{\bm{x}}$. 

The following lemma will be used in Section \ref{sec5}.
\begin{lem}
\label{lem:prob-cap}For any disjoint $A,B\subset\Xi$, a probability
measure $\nu$ concentrated on $A$, and $\gamma>0$,
\[
{\rm P}_{\nu}\left[\mathcal{H}_{B}\le\gamma^{-1}\right]^{2}\le\frac{e^{2}}{\gamma^{-1}}{\rm E}_{\pi}\left[\frac{\nu^{2}}{\pi^{2}}\right]{\rm cap}(A,B).
\]
In particular, if $A=\{\bm{x}\}$ is a singleton, then
\[
{\rm P}_{\bm{x}}\left[\mathcal{H}_{B}\le\gamma^{-1}\right]^{2}\le\frac{e^{2}}{\gamma^{-1}}\frac{{\rm cap}(\bm{x},B)}{\pi(\bm{x})}.
\]
\end{lem}

\begin{proof}
This is a special case of \cite[Proposition 8.4]{LMS23} where the
set $\mathcal{E}_{N}^{x}$ therein is chosen to be $A$ and $\breve{\mathcal{E}}_{N}^{x}$
is chosen to be $B$. The idea is to use the concept of the so-called
\emph{$\gamma$-enlarged process} \cite{BL15}; we refer to \cite[Corollary 4.2]{BL15}
for more detail.
\end{proof}

\end{document}